\documentclass[11pt]{amsart}

\usepackage{latexsym, url}
\usepackage{amsthm}
\usepackage{amsmath}
\usepackage{amsfonts}
\usepackage{amssymb}

\usepackage{mathabx} %left lifting prop
\usepackage{nameref} %naming sections

\usepackage{hyperref}

\usepackage{quiver}

\input xy
\xyoption{all}

\usepackage{adjustbox}

\usepackage{url}
\usepackage{tikz-cd}
\usetikzlibrary{babel}
\usepackage{fancyhdr}
\usepackage{bm}
\usepackage{caption}
\usepackage{mathtools}
\usepackage{xcolor}
\usepackage{paralist}
\usepackage{opakuj}
\usepackage{graphicx}
\usepackage{booktabs}
\usepackage{graphicx}
\usepackage{tabu}
\usepackage{multirow}

\usepackage{titletoc}
\usepackage[normalem]{ulem}

\theoremstyle{definition}
\newtheorem{theo}{Theorem}[section]

\newtheorem{lemma}[theo]{Lemma}
\newtheorem{prop}[theo]{Proposition}
\newtheorem{defi}[theo]{Definition}
\newtheorem{cor}[theo]{Corollary}
\newtheorem{pozn}[theo]{Remark}

\newtheorem{nota}[theo]{Notation}

\newtheorem{pr}[theo]{Example}

\newcommand{\ddashv}{\dashv\!\!\!\!\!\!\dashv}
\newcommand{\ttop}{\rotatebox{90}{$\ddashv$}}
\newcommand{\tbot}{\rotatebox{-90}{$\ddashv$}}

\newcommand{\talg}{\text{T-Alg}}
\newcommand{\talgl}{\text{T-Alg}_l}

\newcommand{\salgl}{\text{S-Alg}_l}
\newcommand{\talgs}{\text{T-Alg}_s}

\newcommand{\talgc}{\text{T-Alg}_c}

\newcommand{\psdalg}{\text{Ps-D-Alg}}

\newcommand{\rel}{\text{Rel}}

\newcommand\psd[2]{\text{Psd}[#1,#2]}
\newcommand\lax[2]{\text{Lax}[#1, #2]}

\newcommand\res[1]{\text{Res}(#1)}

\newcommand{\set}{\text{Set}}
\newcommand{\parf}{\text{Par}}

\newcommand{\cat}{\text{Cat}}
\newcommand{\CAT}{\text{CAT}}
\newcommand{\prof}{\text{PROF}}
\newcommand{\PROF}{\text{PROF}}

\newcommand\ob{\text{ob }}

\newcommand\sstackrel[2]{\stackrel{\mathmakebox[\widthof{#2}]{#1}}{#2}}

\newcommand\ca{\mathcal {A}}
\newcommand\cb{\mathcal {B}}
\newcommand\cc{\mathcal {C}}
\newcommand\cd{\mathcal {D}}

\newcommand\ce{\mathcal {E}}
\newcommand\cj{\mathcal {J}}
\newcommand\ck{\mathcal {K}}
\newcommand\cl{\mathcal {L}}

\newcommand\cp{\mathcal {P}}

\newcommand\cv{\mathcal {V}}

\newcommand\mbbd{\mathbb{D}}
\newcommand\mbbb{\mathbb{B}}
\newcommand\mbbc{\mathbb{C}}

\newcommand\mbba{\mathbb{A}}
\newcommand\mbbl{\mathbb{L}}

\definecolor{arsenic}{rgb}{0.23, 0.27, 0.29}
\definecolor{ashgrey}{rgb}{0.7, 0.75, 0.71}
\definecolor{charcoal}{rgb}{0.12, 0.18, 0.22}
\definecolor{bulgarianrose}{rgb}{0.28, 0.02, 0.03}
\definecolor{darkblue}{rgb}{0, 0, 0.5}

\date{\today}

\newcommand\blfootnote[1]{%
  \begingroup
  \renewcommand\thefootnote{}\footnote{#1}%
  \addtocounter{footnote}{-1}%
  \endgroup
}

\begin{document}
\title[Colax adjunctions and lax-idempotent pseudomonads]
{Colax adjunctions and lax-idempotent pseudomonads}
\author[Miloslav Štěpán]
{Miloslav Štěpán}

\makeatletter
\dottedcontents{section}[1.5em]{\bfseries}{1.3em}{.6em}
\makeatother

\begin{abstract}
We prove a generalization of a theorem of Bunge and Gray about forming colax adjunctions out of relative Kan extensions and apply it to the study of the Kleisli 2-category for a lax-idempotent pseudomonad. For instance, we establish the weak completeness of the Kleisli 2-category and describe colax change-of-base adjunctions between Kleisli 2-categories. Our approach covers such examples as the bicategory of small profunctors and the 2-category of lax triangles in a 2-category. The duals of our results provide lax analogues of classical results in two-dimensional monad theory: for instance, establishing the weak cocompleteness of the 2-category of strict algebras and lax morphisms and the existence of colax change-of-base adjunctions.
\end{abstract} 
%\keywords{}
%\subjclass{}

\blfootnote{Supported by the Grant agency of the Czech republic under the grant 22-02964S.}

\maketitle
\tableofcontents
%\addtocontents{toc}{~\hfill\textbf{Page}\par}

\newpage

\section{Introduction}%\label{SEKCE_Intro}  %386 OLD INTRO

The primary motivation for this paper is to develop \textbf{lax analogues} of classical results in two-dimensional algebra, in particular two-dimensional monad theory as studied in \cite{twodim}. The examples commonly studied in this area include 2-categories of categories with structure and pseudo morphisms between them -- functors that preserve the structure up to coherent isomorphism. For instance categories equipped with a class of colimits and colimit-preserving functors, or monoidal categories and monoidal functors. Such 2-categories can be described as the 2-category $\talg$ of $T$-algebras and pseudo-$T$-morphisms for a 2-monad $T$. Various results have been proven in \cite{twodim} about $T$-algebras and pseudo-$T$-morphisms, for instance their bicocompleteness or the existence of change-of-base biadjunctions between 2-categories of algebras and pseudo-morphisms for two different 2-monads $S,T$.

On the other hand, there are fewer known results about 2-categories of categories with structure and lax morphisms between them. These still include interesting examples, for instance categories equipped with a class of colimits and all functors between them, or monoidal categories and lax monoidal functors. They can also be described using 2-monads, this time as the 2-category $\talgl$ of $T$-algebras and lax $T$-algebra morphisms. While limits in $\talgl$ have been well-understood (\cite{limitsforlax}, \cite{enhanced2cats}), not much has been proven about colimits. This was for a good reason: 2-colimits or even bicolimits often do not exist in those 2-categories. Our task in this paper is to suitably weaken the notion of a bicolimit and show that 2-categories of lax morphisms are in fact cocomplete in this weak sense. Another task we have is to establish change-of-base theorems for algebras and lax morphisms. Again, the notion that works for pseudo-morphisms -- biadjunctions -- will have to be replaced by a weaker one -- colax adjunctions.

The 2-category $\talg$ of algebras and pseudo-morphisms can often be described as the Kleisli 2-category for a certain pseudo-idempotent 2-comonad. A key observation to be made is that many statements and proofs about $\talg$ in papers \cite{twodim}, \cite{onsemiflexible} are very formal and are in fact true for any pseudo-idempotent 2-comonad on a 2-category. They also easily dualize to pseudo-idempotent 2-monads. Since we are interested in the lax world, we are naturally led to the study of Kleisli 2-categories for lax-idempotent pseudomonads, using the formalism of left Kan pseudomonads \cite{kanext}. The usage of pseudomonads instead of 2-monads will allow us to consider a wider array of examples such as the small presheaf pseudomonad, and lets us prove that the bicategory $\text{PROF}$ of locally small categories and small profunctors is weakly complete in the sense of the previous paragraph.

As mentioned, colax adjunctions are inevitable when working with lax morphisms. The definition of a (co)lax adjunction is hard to work with because it contains a large amount of data. Our first main result, Theorem \ref{THM_Bunge}, shows that a left colax adjoint $F$ to a pseudofunctor $U$ can be more conveniently given by a collection of 1-cells $y_A: A \to UFA$ satisfying certain ``relative $U$-left Kan extension” conditions. This is an extension of the work of Bunge and Gray (\cite{bunge}, \cite{formalcat}) where this has been proven for the case when $U$ is a 2-functor. A result of this kind is similar to how left Kan pseudomonads provide a more convenient description of lax-idempotent pseudomonads. We will use this theorem to obtain results on colax adjunctions involving the Kleisli 2-category for a lax-idempotent pseudomonad (Theorem \ref{THM_BIG_lax_adj_thm}), and the dual of \textit{this} result will be used to obtain results on colax adjunctions involving $T$-algebras and lax $T$-morphisms for a 2-monad (Theorem \ref{THM_big_colax_adj_THM}).

The paper is organized as follows. In Section \ref{SEKCE_koncepty_2kategorie} we recall the necessary concepts that we will need in this paper. With the small exception of \textit{left Kan 2-monads}, everything here is well-known.

In Section \ref{SEKCE_Bunge} we prove the generalization of Bunge's and Gray's results on colax adjunctions to the setting of pseudofunctors: we show that there is a correspondence between left colax adjoint pseudofunctors to a pseudofunctor $U$ and collections of 1-cells $y_A: A \to UFA$ satisfying the aforementioned relative $U$-left Kan extension conditions (Theorem \ref{THM_korespondence_lax_adj_u_coh_ext}).

In Section \ref{SEKCE_Kleisli2cat} we first give (an essentially folklore) characterization of algebras for a lax-idempotent pseudomonad in terms of the existence of certain adjoints (Proposition \ref{THM_characterization_D-algs_D-corefl-incl}). We then use this characterization and the generalized Bunge's and Gray's result to prove that when given a lax-idempotent pseudomonad $D$ on $\ck$, any left biadjoint $\ck \to \cl$ that factorizes through the Kleisli 2-category $\ck_D$ gives rise to a colax left adjoint $\ck_D \to \cl$ (Theorem \ref{THM_BIG_lax_adj_thm}). We list various applications, for instance the weak completeness of $\ck_D$ (Theorem \ref{THM_Kleisli2cat_is_corefl_complete}) provided that $\ck$ is bicomplete, or that there is a canonical colax adjunction between $\ck_D$ and the 2-category of pseudo-$D$-algebras (Corollary \ref{THM_lax_adjunkce_EM_Kleisli}).

In Section \ref{SEKCE_applications_to_twodim} we spell out what these results in particular say about the 2-category $\talgl$ of strict algebras and lax morphisms for a 2-monad $T$. This includes the aforementioned colax base-of-change theorem (Corollary \ref{THM_COR_change_of_base_SalgTalg}) as well as the weak cocompleteness result for $\talgl$ (Theorem \ref{THM_talgl_reflector_cocomplete}).

\medskip

\noindent \textbf{Prerequisities}: We assume the reader is familiar with 2-monads and pseudomonads and their pseudo and strict algebras. We also assume the familiarity with lax-idempotent pseudomonads.

\medskip

\noindent \textbf{Acknowledgements}: I want to thank my Ph.D. supervisor John Bourke for his careful guidance and all the feedback I have received. I also want to thank Nathanael Arkor for sharing his knowledge with me.

\section{Background}\label{SEKCE_koncepty_2kategorie}

\subsection{Colax functors and transformations}\label{SUBS_2_lax_transf}

In this text we will primarily use the \textbf{colax} versions of concepts such as lax functors, lax transformations. The motivation for this is that we are building on the work of \cite{bunge} which uses colax structures, as opposed to lax ones\footnote{In \cite{bunge}, colax natural transformations are referred to as \textit{lax}. In this paper we are following the modern terminology.}. %The motivation for this decision is the fact that colax adjunctions appear in the work\footnote{They have been called \textit{lax adjunctions} there, however.} \cite{bunge} our Section \ref{SEKCE_Bunge} is based on.

\begin{defi}
Let $\ca, \cb$ be 2-categories. A \textit{colax functor} $F: \ca \to \cb$ consists of:
\begin{itemize}
\item A function $F_0: \ob \ca \to \ob \cb$,
\item for every pair $A.B$ of objects of $\ca$ a functor $F_{A,B}: \ca(A,B) \to \cb(FA,FB)$,
\item for every composable pair $(f,g)$ of morphisms in $\ca$ a 2-cell (\textit{associator})\\ $\gamma_{f,g}: F(g \circ f) \Rightarrow Fg \circ Ff$,
\item for every object $A \in \ca$ a 2-cell (\textit{unitor}) $\iota_A: F1_A \Rightarrow 1_{FA}$,
\end{itemize}
subject to associativity and unit axioms, see for instance \cite[Definition 4.1.2]{2dimensionalcats}. If $\gamma$ and $\iota$ go in the other direction, we obtain the notion of a \textit{lax functor}. In case $\gamma$, $\iota$ are invertible, this is called a \textit{pseudofunctor}.
\end{defi}
%\cite[(4.1)]{introbicats}

\noindent For simplicity, we will always use the letters $\gamma, \iota$ for the associator and the unitor of a colax functor, and always omit the index for any of its components.

\begin{defi}
Given 2-categories $\ca, \cb$ and two pseudofunctors $F,G: \ca \to \cb$, a \textit{colax natural transformation} $\alpha: F \Rightarrow G$ consists of the following data:
\begin{itemize}
\item For every $A \in \ca$ a 1-cell $\alpha_A : FA \to GA$,
\item For every $f: A \to B \in \ca$ a 2-cell:
% https://q.uiver.app/#q=WzAsNCxbMCwwLCJGQSJdLFsyLDAsIkdBIl0sWzAsMSwiRkIiXSxbMiwxLCJHQiJdLFsyLDMsIlxcYWxwaGFfQiIsMl0sWzAsMSwiXFxhbHBoYV9BIl0sWzAsMiwiRmYiLDJdLFsxLDMsIkdmIl0sWzQsNSwiXFxhbHBoYV9mIiwyLHsic2hvcnRlbiI6eyJzb3VyY2UiOjIwLCJ0YXJnZXQiOjIwfX1dXQ==
\[\begin{tikzcd}
	FA && GA \\
	FB && GB
	\arrow[""{name=0, anchor=center, inner sep=0}, "{\alpha_B}"', from=2-1, to=2-3]
	\arrow[""{name=1, anchor=center, inner sep=0}, "{\alpha_A}", from=1-1, to=1-3]
	\arrow["Ff"', from=1-1, to=2-1]
	\arrow["Gf", from=1-3, to=2-3]
	\arrow["{\alpha_f}"', shorten <=4pt, shorten >=4pt, Rightarrow, from=0, to=1]
\end{tikzcd}\]
\end{itemize}
These must satisfy certain {unit}, {composition}, local naturality conditions, see \cite[Definition 4.3.1]{2dimensionalcats}. If the 2-cells $\alpha_f$ go in the other direction, this is referred to as a \textit{lax natural transformation}. If $\alpha_f$ is invertible for all morphisms $f$, $\alpha$ is called a \textit{pseudo-natural transformation}. If the $\alpha_f$'s are the identities, we use the term \textit{2-natural transformation}.
%340
\end{defi}

\begin{defi}
Given two pseudonatural transformations $\alpha, \beta$ between pseudofunctors $F,G: \ck \to \cl$, a \textit{modification} $\Gamma: \alpha \to \beta$ consists of a 2-cell $\Gamma_A: \alpha_A \Rightarrow \beta_A$ for every object $A \in \ck$, subject to the modification axiom for each 1-cell in $\ck$, see \cite[Definition 4.4.1]{2dimensionalcats}.
%376
\end{defi}

\begin{pr}
Given an endofunctor $T: \ca \to \ca$, any colax natural transformation $c: T \Rightarrow 1_\ca$ induces a modification $(cc): c \circ Tc \to c \circ cT$, whose component at $A \in \ca$ is given by:
\[
(cc)_A := c_{c_A} :c_A \circ Tc_A \Rightarrow c_A \circ c_{TA}.
\]
\end{pr}

\begin{pozn}
Pseudofunctors preserve colax natural transformations. If $H: \cc \to \cd$ is a pseudofunctor and $\alpha: F \Rightarrow G: \cb \to \cc$ is colax natural, there is an induced colax natural transformation $H\alpha$ whose 1-cell component at $A$ is $H\alpha_A$ and whose 2-cell component at a morphism $f: A \to B$ is the following composite 2-cell that we denote by $(H\alpha)_f$:
%259
% https://q.uiver.app/#q=WzAsNCxbMCwwLCJGQSJdLFsyLDAsIkdBIl0sWzAsMywiRkIiXSxbMiwzLCJHQiJdLFsyLDMsIkhcXGFscGhhX0IiLDJdLFswLDIsIkhGZiIsMl0sWzAsMSwiSFxcYWxwaGFfQSJdLFsxLDMsIkhHZiJdLFswLDMsIiIsMSx7ImN1cnZlIjotM31dLFswLDMsIiIsMSx7ImN1cnZlIjozfV0sWzksOCwiSFxcYWxwaGFfZiIsMix7InNob3J0ZW4iOnsic291cmNlIjoyMCwidGFyZ2V0IjoyMH19XSxbOCwxLCJcXGdhbW1hXnstMX0iLDAseyJzaG9ydGVuIjp7InNvdXJjZSI6MjB9fV0sWzIsOSwiXFxnYW1tYSIsMix7InNob3J0ZW4iOnsidGFyZ2V0IjoyMH19XV0=
\[\begin{tikzcd}
	FA && GA \\
	\\
	\\
	FB && GB
	\arrow["{H\alpha_B}"', from=4-1, to=4-3]
	\arrow["HFf"', from=1-1, to=4-1]
	\arrow["{H\alpha_A}", from=1-1, to=1-3]
	\arrow["HGf", from=1-3, to=4-3]
	\arrow[""{name=0, anchor=center, inner sep=0}, curve={height=-18pt}, from=1-1, to=4-3]
	\arrow[""{name=1, anchor=center, inner sep=0}, curve={height=18pt}, from=1-1, to=4-3]
	\arrow["{H\alpha_f}"', shorten <=6pt, shorten >=6pt, Rightarrow, from=1, to=0]
	\arrow["{\gamma^{-1}}", shorten <=5pt, Rightarrow, from=0, to=1-3]
	\arrow["\gamma"', shorten >=5pt, Rightarrow, from=4-1, to=1]
\end{tikzcd}\]
%Such a transformation may be also pre-composed with a pseudofunctor $K: \ca \to \cb$, \todo
\end{pozn}

\subsection{Colax adjunctions}\label{SUBS_2_lax_adj}

Lax adjunctions, also called \textit{quasi-adjunctions} in \cite[I,7.1]{formalcat} are a categorification of adjunctions between functors where the unit and the counit are replaced by lax natural transformations, and the triangle identities are replaced by modifications. As in the previous section, we will use the dual notion -- colax adjunctions.
%224

\begin{defi}\label{DEFI_lax_adjoint}
A \textit{colax adjunction} consists of two pseudofunctors $U: \cd \to \cc$ and $F: \cc \to \cd$, two colax natural transformations $\eta : 1 \Rightarrow UF$ and $\epsilon : FU \Rightarrow 1$ and two modifications:
% https://q.uiver.app/#q=WzAsNixbMCwwLCJGIl0sWzIsMCwiRlVGIl0sWzIsMiwiRiJdLFs0LDAsIlUiXSxbNiwwLCJVRlUiXSxbNiwyLCJVIl0sWzAsMSwiRlxcZXRhIl0sWzEsMiwiXFxlcHNpbG9uIEYiXSxbMCwyLCIiLDIseyJsZXZlbCI6Miwic3R5bGUiOnsiaGVhZCI6eyJuYW1lIjoibm9uZSJ9fX1dLFszLDUsIiIsMCx7ImxldmVsIjoyLCJzdHlsZSI6eyJoZWFkIjp7Im5hbWUiOiJub25lIn19fV0sWzMsNCwiXFxldGEgVSJdLFs0LDUsIlUgXFxlcHNpbG9uIl0sWzEsOCwiXFxQc2kiLDIseyJzaG9ydGVuIjp7InNvdXJjZSI6MjAsInRhcmdldCI6MjB9fV0sWzksNCwiXFxQaGkiLDAseyJzaG9ydGVuIjp7InNvdXJjZSI6MjAsInRhcmdldCI6MjB9fV1d
\[\begin{tikzcd}
	F && FUF && U && UFU \\
	\\
	&& F &&&& U
	\arrow["F\eta", from=1-1, to=1-3]
	\arrow["{\epsilon F}", from=1-3, to=3-3]
	\arrow[""{name=0, anchor=center, inner sep=0}, Rightarrow, no head, from=1-1, to=3-3]
	\arrow[""{name=1, anchor=center, inner sep=0}, Rightarrow, no head, from=1-5, to=3-7]
	\arrow["{\eta U}", from=1-5, to=1-7]
	\arrow["{U \epsilon}", from=1-7, to=3-7]
	\arrow["\Psi"', shorten <=5pt, shorten >=5pt, Rightarrow, from=1-3, to=0]
	\arrow["\Phi", shorten <=5pt, shorten >=5pt, Rightarrow, from=1, to=1-7]
\end{tikzcd}\]

Before stating the axioms required, let us fix a convention: we will use the symbol $U \Psi$ to denote the modification obtained from $\Psi$ by not just applying $U$, but also by pre- and post-composing it with the associator and the unitor for $U$ so that its domain and codomain are $U\epsilon F \circ UF\eta$, $1_{UF}$. Let us use the same convention for $F \Phi$. The axioms are the \textit{swallowtail identities}, which assert that the two composite modifications below are the identities on $\eta$ and $\epsilon$:
\[
\adjustbox{scale=0.9,center}{
\begin{tikzcd}
	{1_\cc} && UF &&& FU \\
	&&& {} &&&& {} & {} \\
	UF && UFUF & {} &&& {} & FUFU && FU \\
	& {} & {} &&&& {} \\
	&&&& UF &&& FU && {1_{\cd}}
	\arrow[curve={height=24pt}, Rightarrow, no head, from=3-1, to=5-5]
	\arrow[curve={height=-24pt}, Rightarrow, no head, from=1-3, to=5-5]
	\arrow["\eta"', from=1-1, to=3-1]
	\arrow["\eta", from=1-1, to=1-3]
	\arrow["UF\eta", from=1-3, to=3-3]
	\arrow["{\eta UF}"', from=3-1, to=3-3]
	\arrow["{\eta \eta}"{description}, shorten <=12pt, shorten >=12pt, Rightarrow, from=3-1, to=1-3]
	\arrow["{U\epsilon F}"{description}, from=3-3, to=5-5]
	\arrow["{F\eta U}"{description}, from=1-6, to=3-8]
	\arrow["FU\epsilon", from=3-8, to=3-10]
	\arrow["{\epsilon_{FU}}"', from=3-8, to=5-8]
	\arrow["\epsilon"', from=5-8, to=5-10]
	\arrow["\epsilon", from=3-10, to=5-10]
	\arrow[curve={height=24pt}, Rightarrow, no head, from=1-6, to=5-8]
	\arrow[curve={height=-24pt}, Rightarrow, no head, from=1-6, to=3-10]
	\arrow["{\epsilon \epsilon}", shorten <=12pt, shorten >=12pt, Rightarrow, from=3-10, to=5-8]
	\arrow[""{name=0, anchor=center, inner sep=0}, draw=none, from=3-3, to=3-4]
	\arrow[""{name=0p, anchor=center, inner sep=0}, phantom, from=3-3, to=3-4, start anchor=center, end anchor=center]
	\arrow[""{name=1, anchor=center, inner sep=0}, draw=none, from=4-2, to=4-3]
	\arrow[""{name=1p, anchor=center, inner sep=0}, phantom, from=4-2, to=4-3, start anchor=center, end anchor=center]
	\arrow[""{name=2, anchor=center, inner sep=0}, draw=none, from=2-8, to=2-9]
	\arrow[""{name=2p, anchor=center, inner sep=0}, phantom, from=2-8, to=2-9, start anchor=center, end anchor=center]
	\arrow[""{name=3, anchor=center, inner sep=0}, shift left=3, draw=none, from=3-7, to=3-8]
	\arrow[""{name=3p, anchor=center, inner sep=0}, phantom, from=3-7, to=3-8, start anchor=center, end anchor=center, shift left=3]
	\arrow["U\Psi"', shorten >=10pt, Rightarrow, from=0p, to=2-4]
	\arrow["{\Phi F}"', shorten >=4pt, Rightarrow, from=1p, to=3-3]
	\arrow["F\Phi"', shorten <=4pt, shorten >=2pt, Rightarrow, from=2p, to=3-8]
	\arrow["{\Psi U}"', shorten <=5pt, shorten >=9pt, Rightarrow, from=3p, to=4-7]
\end{tikzcd}
}\]

%378

\begin{nota}
We will denote a colax adjunction as follows and say that $F$ is a \textit{left colax adjoint} to $U$:
% https://q.uiver.app/#q=WzAsMyxbMCwwLCIoXFxQc2ksIFxcUGhpKTogKFxcZXBzaWxvbiwgXFxldGEpOiJdLFsxLDAsIlxcY2MiXSxbMywwLCJcXGNkIl0sWzEsMiwiRiIsMCx7ImN1cnZlIjotNH1dLFsyLDEsIlUiLDAseyJjdXJ2ZSI6LTR9XSxbMyw0LCJcXHRib3QiLDEseyJzaG9ydGVuIjp7InNvdXJjZSI6MjAsInRhcmdldCI6MjB9LCJzdHlsZSI6eyJib2R5Ijp7Im5hbWUiOiJub25lIn0sImhlYWQiOnsibmFtZSI6Im5vbmUifX19XV0=
\[\begin{tikzcd}
	{(\Psi, \Phi): (\epsilon, \eta):} & \cc && \cd
	\arrow[""{name=0, anchor=center, inner sep=0}, "F", curve={height=-24pt}, from=1-2, to=1-4]
	\arrow[""{name=1, anchor=center, inner sep=0}, "U", curve={height=-24pt}, from=1-4, to=1-2]
	\arrow["\tbot"{description}, draw=none, from=0, to=1]
\end{tikzcd}\]
\end{nota}

\noindent There are several important variations or special cases:
\begin{itemize}
\item if $\epsilon, \eta$ are lax natural, $\Psi,\Phi$ go in the other directions and an appropriate dual of the swallowtail identities holds, we will call it a \textit{lax adjunction},
\item in case that $\epsilon, \eta$ are pseudonatural transformations and $\Psi ,\Phi$ are isomorphisms, we will use the term \textit{biadjunction}.
\item if $U,F$ are 2-functors, $\epsilon, \eta$ are 2-natural and $\Psi, \Phi$ are the identities, we will call this a \textit{2-adjunction}.
\end{itemize}
Since the last two cases are the more usual notion, we will use the usual symbol $\dashv$ instead of $\ddashv$\phantom{.} for them.
\end{defi}

\begin{pozn}\label{POZN_lax_adjs_not_unique_up_to_equiv}%vvv
Contrary to the case of biadjunctions, left colax adjoints are not unique up to an equivalence, not even when $U$ is a 2-functor, $\eta$ is 2-natural and $\Psi$, $\Phi$ are the identities. An example will be given in Remark \ref{POZN_rali_limits_not_unique_up_to_equiv}.
\end{pozn}

\subsection{Lax-idempotent and left Kan pseudomonads}\label{SUBS_2_left_kan_psmon}

The notion of a lax-idempotent pseudomonad (see \cite[Section 2]{kanext}) contains a large amount of data and axioms. A major simplification can be achieved if one works with \textit{left Kan pseudomonads} instead. In this section we recall all the basic definitions and mention the equivalence of left Kan pseudomonads and lax-idempotent pseudomonads. We also define a special class of left Kan pseudomonads that we call \textit{left Kan 2-monads} -- this is the obvious strict version of the notion.

\begin{defi}
A \textit{left Kan pseudomonad} (\cite{kanext}) $(D,y)$ on a 2-category $\ck$ consists of:

\begin{itemize}
\item A function $D: \ob \ck \to \ob \ck$,
\item For every $A \in \ck$ a 1-cell $y_A: A \to DA$ called its \textit{unit},
\item For every 1-cell $f: A \to DB$ a left Kan extension of $f$ along $y_B$ such that the accompanying 2-cell is invertible:
% https://q.uiver.app/#q=WzAsMyxbMCwwLCJBIl0sWzIsMCwiREEiXSxbMiwyLCJEQiJdLFsxLDIsImZeXFxtYmJkIl0sWzAsMSwieV9CIl0sWzAsMiwiZiIsMl0sWzUsMSwiXFxtYmJkX2YiLDAseyJzaG9ydGVuIjp7InNvdXJjZSI6MjB9fV1d
\begin{equation}\label{EQ_left_kan_ext_mbbd}
\begin{tikzcd}
	A && DA \\
	\\
	&& DB
	\arrow["{f^\mbbd}", from=1-3, to=3-3]
	\arrow["{y_A}", from=1-1, to=1-3]
	\arrow[""{name=0, anchor=center, inner sep=0}, "f"', from=1-1, to=3-3]
	\arrow["{\mbbd_f}", shorten <=6pt, Rightarrow, from=0, to=1-3]
\end{tikzcd}
\end{equation}
\end{itemize}
These are subject to the axioms:
\begin{itemize}
\item For every $A \in \ck$, the identity 2-cell $1_{y_B}$ on $y_B$ exhibits $1_{DA}$ as a left Kan extension of $y_A$ along $y_A$:
% https://q.uiver.app/#q=WzAsMyxbMCwwLCJCIl0sWzIsMCwiREIiXSxbMiwyLCJEQiJdLFsxLDIsIiIsMCx7ImxldmVsIjoyLCJzdHlsZSI6eyJoZWFkIjp7Im5hbWUiOiJub25lIn19fV0sWzAsMSwieV9CIl0sWzAsMiwieV9CIiwyXV0=
%398

\item for every $g: B \to DC$, $f: A \to DB$, $g^\mbbd$ preserves the left Kan extension \eqref{EQ_left_kan_ext_mbbd}.
%351
\end{itemize}
\end{defi}

\begin{defi}\label{DEFI_D_algebra}
A \textit{pseudo-$D$-algebra} consists of an object $C \in \ck$ together with a mapping that sends every 1-cell $f: B \to C$ to the left Kan extension of $f$ along $y_B$ such that the accompanying 2-cell is invertible:
% https://q.uiver.app/#q=WzAsMyxbMCwwLCJCIl0sWzIsMCwiREIiXSxbMiwyLCJDIl0sWzEsMiwiZl5cXG1iYmMiXSxbMCwxLCJ5X0IiXSxbMCwyLCJmIiwyXSxbNSwxLCJcXG1iYmNfZiIsMCx7InNob3J0ZW4iOnsic291cmNlIjoyMH19XV0=
\begin{equation}\label{EQ_left_kan_ext_mbbb}
\begin{tikzcd}
	B && DB \\
	\\
	&& C
	\arrow["{f^\mbbc}", from=1-3, to=3-3]
	\arrow["{y_B}", from=1-1, to=1-3]
	\arrow[""{name=0, anchor=center, inner sep=0}, "f"', from=1-1, to=3-3]
	\arrow["{\mbbc_f}", shorten <=6pt, Rightarrow, from=0, to=1-3]
\end{tikzcd}
\end{equation}
and such that for every $f: A \to DB$, $g^\mbbb$ preserves the left Kan extension \eqref{EQ_left_kan_ext_mbbd}.
%352

A \textit{$D$-pseudomorphism} $h: B \to A$ between pseudo-$D$-algebras $C,X$ is a 1-cell $h: C \to X$ that preserves the Left Kan extension \eqref{EQ_left_kan_ext_mbbb}. A \textit{pseudo-$D$-algebra 2-cell} $\alpha: h \Rightarrow h' : B \to A$ is just a 2-cell in $\ck$. All this data assembles into a 2-category that we denote by $\psdalg$.
\end{defi}
%353

\begin{defi}
By the \textit{Kleisli 2-category} $\ck_D$ associated to the left Kan pseudomonad $(D,y)$ we mean the full sub-2-category of $\psdalg$ spanned by \textit{free $D$-algebras}, that is, algebras whose underlying object is of form $DA$ for some object $A \in \ck$ and the extension operation is given by $(-)^\mbbd$.
\end{defi}

\begin{pozn}\label{POZN_Kleisli_bicat}
We may also define the \textit{Kleisli bicategory} associated to a left Kan pseudomonad $(D,y)$, where objects are the objects in $\ck$ and a morphism $A \rightsquigarrow B$ in $\ck_D$ corresponds to a morphism $A \to DB$ in $\ck$. The unit is given by the unit of the pseudomonad, while the composition is defined using the extension operation:
% https://q.uiver.app/#q=WzAsNixbMCwwLCJBIl0sWzEsMCwiQiJdLFsyLDAsIkMiXSxbMywwLCJcXG1hcHN0byJdLFs0LDAsIkEiXSxbNiwwLCJCIl0sWzAsMSwiZiIsMCx7InN0eWxlIjp7ImJvZHkiOnsibmFtZSI6InNxdWlnZ2x5In19fV0sWzEsMiwiZyIsMCx7InN0eWxlIjp7ImJvZHkiOnsibmFtZSI6InNxdWlnZ2x5In19fV0sWzQsNSwiZ15cXG1iYmQgXFxjaXJjIGYiLDAseyJzdHlsZSI6eyJib2R5Ijp7Im5hbWUiOiJzcXVpZ2dseSJ9fX1dXQ==
\[\begin{tikzcd}
	A & B & C & \mapsto & A && B
	\arrow["f", squiggly, from=1-1, to=1-2]
	\arrow["g", squiggly, from=1-2, to=1-3]
	\arrow["{g^\mbbd \circ f}", squiggly, from=1-5, to=1-7]
\end{tikzcd}\]
Denote this bicategory by $\text{Kl}(D)$. It is routine to verify that there is a pseudofunctor $N: \text{Kl}(D) \to \ck_D$ sending the Kleisli morphism $f: A \rightsquigarrow B$ to $f^\mbbd : DA \to DB$ and that it is a biequivalence of bicategories. In this paper we will for the most part use the 2-category presentation since it is easier to work with.
\end{pozn}

\begin{prop}\label{THM_biadjunkce_ck_ck_D}
There is a ``free-forgetful” biadjunction given as follows:
% https://q.uiver.app/#q=WzAsMyxbMSwwLCJcXGNrIl0sWzMsMCwiXFxja19EIl0sWzAsMCwiKFxcUHNpLCBcXFBoaSk6KHAscSk6Il0sWzAsMSwiSl9EIiwyLHsiY3VydmUiOjR9XSxbMSwwLCJVX0QiLDIseyJjdXJ2ZSI6NH1dLFszLDQsIiIsMCx7ImxldmVsIjoxLCJzdHlsZSI6eyJuYW1lIjoiYWRqdW5jdGlvbiJ9fV1d
\[\begin{tikzcd}
	{(\Psi, \Phi):(p,q):} & \ck && {\ck_D}
	\arrow[""{name=0, anchor=center, inner sep=0}, "{J_D}"', curve={height=24pt}, from=1-2, to=1-4]
	\arrow[""{name=1, anchor=center, inner sep=0}, "{U_D}"', curve={height=24pt}, from=1-4, to=1-2]
	\arrow["\dashv"{anchor=center, rotate=90}, draw=none, from=0, to=1]
\end{tikzcd}\]

\begin{itemize}
\item The right biadjoint $U_D$ is the forgetful 2-functor sending an algebra to its underlying object,
\item the left biadjoint is a normal pseudofunctor sending:
\[
(f: A \to B) \mapsto ((y_B f)^\mbbd: DA \to DB),
\]
\item the counit $p: J_D U_D \Rightarrow 1$ evaluated at the object $DA$ is the following algebra homomorphism:
\[
p_{DA} := (1_{DA})^\mbbd: D^2A \to DA,
\]
With its pseudonaturality square at an algebra morphism $h$ being the canonical isomorphism between $1^\mbbd_{DB} Dh$ and $h 1_{DA}^\mbbd$, as both are the left Kan extensions of $h$ along $y_{DA}$.
\item the unit is given by the unit of the left Kan pseudomonad $y: 1 \Rightarrow U_D J_D$, with the pseudonaturality square at a morphism $h: A \to B$ being given by the canonical isomorphism:
% https://q.uiver.app/#q=WzAsNCxbMCwwLCJBIl0sWzIsMCwiREEiXSxbMiwxLCJEQiJdLFswLDEsIkIiXSxbMSwyLCJEaCJdLFswLDEsInlfQSJdLFszLDIsInlfQiIsMl0sWzAsMywiaCIsMl0sWzYsNSwiXFxtYmJkX3t5X0IgaH0iLDIseyJzaG9ydGVuIjp7InNvdXJjZSI6MjAsInRhcmdldCI6MjB9fV1d
\[\begin{tikzcd}[cramped]
	A && DA \\
	B && DB
	\arrow["Dh", from=1-3, to=2-3]
	\arrow[""{name=0, anchor=center, inner sep=0}, "{y_A}", from=1-1, to=1-3]
	\arrow[""{name=1, anchor=center, inner sep=0}, "{y_B}"', from=2-1, to=2-3]
	\arrow["h"', from=1-1, to=2-1]
	\arrow["{\mbbd_{y_B h}}"', shorten <=4pt, shorten >=4pt, Rightarrow, from=1, to=0]
\end{tikzcd}\]

\item the components of the modifications are given by the canonical isomorphisms:
\begin{align*}
\Psi &: pJ_D \circ J_Dy \cong 1_{J_D},\\
\Phi &: 1_{U_D} \cong U_D p \circ y U_D.
\end{align*}
\end{itemize}
\end{prop}

\noindent This biadjunction is moreover \textit{lax-idempotent}, meaning the following:
\begin{prop}\label{THM_prop_kz_adjunkce}
There exist (non-invertible) modifications $\Gamma, \Theta$ that serve as the unit and the counit of the following adjunctions:
\begin{align*}
(\Phi^{-1}, \Gamma) &: U_D p \dashv y U_D, \\
(\Theta, \Psi^{-1}) &: J_Dy \dashv pJ_D.
\end{align*}
\end{prop}
\begin{proof}
The 2-cell $\Gamma_{DA}: 1_{D^2A} \Rightarrow y_{DA} \circ p_{DA}$ is the unique solution to the following equation:
% https://q.uiver.app/#q=WzAsMTAsWzAsMSwiREEiXSxbMiwwLCJEXjJBIl0sWzIsMiwiRF4yQSJdLFszLDEsIkRBIl0sWzQsMSwiPSJdLFs1LDEsIkRBIl0sWzcsMCwiRF4yQSJdLFs4LDEsIkRBIl0sWzcsMiwiRF4yQSJdLFs3LDFdLFswLDEsInlfe0RBfSJdLFswLDIsInlfe0RBfSIsMl0sWzEsMiwiIiwwLHsibGV2ZWwiOjIsInN0eWxlIjp7ImhlYWQiOnsibmFtZSI6Im5vbmUifX19XSxbMSwzLCJwX3tEQX0iXSxbMywyLCJ5X3tEQX0iXSxbNSw2LCJ5X3tEQX0iXSxbNiw3LCJwX3tEQX0iXSxbNyw4LCJ5X3tEQX0iXSxbNSw3LCIiLDAseyJsZXZlbCI6Miwic3R5bGUiOnsiaGVhZCI6eyJuYW1lIjoibm9uZSJ9fX1dLFs1LDgsInlfe0RBfSIsMl0sWzksNiwiXFxQaGlfe0RBfSIsMCx7InNob3J0ZW4iOnsic291cmNlIjozMH0sImxldmVsIjoyfV0sWzEyLDMsIlxcZXhpc3RzICEiLDAseyJzaG9ydGVuIjp7InNvdXJjZSI6MjB9fV1d
\[\begin{tikzcd}
	&& {D^2A} &&&&& {D^2A} \\
	DA &&& DA & {=} & DA && {} & DA \\
	&& {D^2A} &&&&& {D^2A}
	\arrow["{y_{DA}}", from=2-1, to=1-3]
	\arrow["{y_{DA}}"', from=2-1, to=3-3]
	\arrow[""{name=0, anchor=center, inner sep=0}, Rightarrow, no head, from=1-3, to=3-3]
	\arrow["{p_{DA}}", from=1-3, to=2-4]
	\arrow["{y_{DA}}", from=2-4, to=3-3]
	\arrow["{y_{DA}}", from=2-6, to=1-8]
	\arrow["{p_{DA}}", from=1-8, to=2-9]
	\arrow["{y_{DA}}", from=2-9, to=3-8]
	\arrow[Rightarrow, no head, from=2-6, to=2-9]
	\arrow["{y_{DA}}"', from=2-6, to=3-8]
	\arrow["{\Phi_{DA}}", shorten <=3pt, Rightarrow, from=2-8, to=1-8]
	\arrow["{\exists !}", shorten <=5pt, Rightarrow, from=0, to=2-4]
\end{tikzcd}\]
This also proves the first triangle identity. The proof of the second triangle identity is done by pre-composing by $y_A$ and using the appropriate universal properties. By doctrinal adjunction, the collection of adjunctions $(\Phi^{-1}_{DA}, \Gamma_{DA}): p_{DA} \dashv y_{DA}$ lifts to give the claimed adjunction of pseudonatural transformations. The other adjunction is proven in an analogous way.
\end{proof}

\begin{theo}\label{THM_korespondence_leftkanpsmon_laxidemppsmon}
There is a correspondence between:
\begin{itemize}
\item left Kan pseudomonads $(D,y)$ on $\ck$,
\item lax-idempotent pseudomonads $(D,m,y)$ on $\ck$.
\end{itemize}
Moreover, the left Kan pseudomonad and lax-idempotent pseudomonad corresponding to one another have biequivalent 2-categories of algebras, and this biequivalence commutes with the forgetful 2-functors to $\ck$.
%326
\end{theo}
\begin{proof}
For the full proof see \cite[4.1, 4.2]{kanext}, here we sketch only the bits relevant for this paper. Given a left Kan pseudomonad $(D, y)$, the lax-idempotent pseudomonad is given by a normal pseudofunctor $D: \ck \to \ck$ with action on 1-cells and 2-cells given by $U_D \circ J_D$ from Proposition \ref{THM_biadjunkce_ck_ck_D}. The components of the unit $y$ become pseudonatural with the pseudonaturality square given by the Kan extension 2-cell:
% https://q.uiver.app/#q=WzAsNCxbMCwwLCJBIl0sWzIsMCwiREEiXSxbMCwxLCJCIl0sWzIsMSwiREIiXSxbMCwxLCJ5X0EiXSxbMSwzLCIoeV9CZileXFxtYmJkID06RGYiXSxbMCwyLCJmIiwyXSxbMiwzLCJ5X0IiLDJdLFs3LDQsIlxcbWJiZCIsMix7InNob3J0ZW4iOnsic291cmNlIjoyMCwidGFyZ2V0IjoyMH19XV0=
\[\begin{tikzcd}
	A && DA \\
	B && DB
	\arrow[""{name=0, anchor=center, inner sep=0}, "{y_A}", from=1-1, to=1-3]
	\arrow["{(y_Bf)^\mbbd =:Df}", from=1-3, to=2-3]
	\arrow["f"', from=1-1, to=2-1]
	\arrow[""{name=1, anchor=center, inner sep=0}, "{y_B}"', from=2-1, to=2-3]
	\arrow["\mbbd"', shorten <=4pt, shorten >=4pt, Rightarrow, from=1, to=0]
\end{tikzcd}\]
The multiplication at $A \in \ck$ is given by the morphism $p_{DA}$ again as in  Proposition \ref{THM_biadjunkce_ck_ck_D}. On the other hand, given a lax-idempotent pseudomonad $(D,m,y)$, the left Kan extension of $f: A \to DB$ along $y_A : A \to DA$ is given by the composite of the pseudonaturality 2-cell $y_f$ and the pseudomonad unitor 2-cell:
% https://q.uiver.app/#q=WzAsNSxbMCwwLCJBIl0sWzIsMCwiREEiXSxbMCwxLCJEQiJdLFsyLDEsIkReMkIiXSxbMiwyLCJEQiJdLFswLDEsInlfQSJdLFswLDIsImYiLDJdLFsyLDMsInlfe0RCfSIsMV0sWzEsMywiRGYiXSxbMyw0LCJtX0IiXSxbMiw0LCIiLDIseyJsZXZlbCI6Miwic3R5bGUiOnsiaGVhZCI6eyJuYW1lIjoibm9uZSJ9fX1dLFs1LDcsInlfZiIsMCx7InNob3J0ZW4iOnsic291cmNlIjoyMCwidGFyZ2V0IjoyMH0sImVkZ2VfYWxpZ25tZW50Ijp7InNvdXJjZSI6ZmFsc2UsInRhcmdldCI6ZmFsc2V9fV0sWzcsMTAsIlxcY29uZyIsMCx7ImxhYmVsX3Bvc2l0aW9uIjo2MCwib2Zmc2V0IjotNSwic2hvcnRlbiI6eyJzb3VyY2UiOjMwfSwiZWRnZV9hbGlnbm1lbnQiOnsic291cmNlIjpmYWxzZSwidGFyZ2V0IjpmYWxzZX19XV0=
\[\begin{tikzcd}
	A && DA \\
	DB && {D^2B} \\
	&& DB
	\arrow[""{name=0, anchor=center, inner sep=0}, "{y_A}", from=1-1, to=1-3]
	\arrow[""{name=0p, anchor=center, inner sep=0}, phantom, from=1-1, to=1-3, start anchor=center, end anchor=center]
	\arrow["f"', from=1-1, to=2-1]
	\arrow[""{name=1, anchor=center, inner sep=0}, "{y_{DB}}"{description}, from=2-1, to=2-3]
	\arrow[""{name=1p, anchor=center, inner sep=0}, phantom, from=2-1, to=2-3, start anchor=center, end anchor=center]
	\arrow[""{name=1p, anchor=center, inner sep=0}, phantom, from=2-1, to=2-3, start anchor=center, end anchor=center]
	\arrow["Df", from=1-3, to=2-3]
	\arrow["{m_B}", from=2-3, to=3-3]
	\arrow[""{name=2, anchor=center, inner sep=0}, Rightarrow, no head, from=2-1, to=3-3]
	\arrow[""{name=2p, anchor=center, inner sep=0}, phantom, from=2-1, to=3-3, start anchor=center, end anchor=center]
	\arrow["{y_f}", shorten <=4pt, shorten >=4pt, Rightarrow, from=0p, to=1p]
	\arrow["\cong"{pos=0.6}, shift left=5, shorten <=3pt, Rightarrow, from=1p, to=2p]
\end{tikzcd}\]

\end{proof}

\noindent In the rest of the paper we will use the terms ``left Kan pseudomonads” and ``lax-idempotent pseudomonads” interchangeably.

%292

\begin{pozn}[Duals]\label{POZN_duals}
A lax-idempotent pseudomonad $T$ on a 2-category $\ck$ is equivalently:
\begin{itemize}
\item a colax-idempotent pseudomonad $T^{co}$ on $\ck^{co}$,
\item a colax-idempotent pseudo-comonad $T^{op}$ on $\ck^{op}$,
\item a lax-idempotent pseudo-comonad $T^{coop}$ on $\ck^{coop}$.
\end{itemize}
\end{pozn}

%algebra is adjoint to unit
%371

\subsection{Left Kan 2-monads}

There is a class of lax-idempotent pseudomonads that will play a role: the ones for which the pseudomonad is actually a \textit{2-monad}. We will show that these correspond to what we call \textit{left Kan 2-monads}.

\begin{defi}\label{DEFI_leftkan2monad}
A left Kan pseudomonad $(D,y)$ is a \textit{left Kan 2-monad} if:
\begin{itemize}
\item $\mbbd_f$ is the identity 2-cell for every 1-cell $f:B \to DA$, meaning that $f^\mbbd \circ y_B = f$,
\item $g^\mbbd f^\mbbd = (g^\mbbd f)^\mbbd$,
\item $y_A^\mbbd = 1_{DA}$.
\end{itemize}
\end{defi}

\noindent Notice that in case of left Kan 2-monads, the biadjunction from Proposition \ref{THM_biadjunkce_ck_ck_D} becomes a 2-adjunction. Let us also note the following:

\begin{prop}
The correspondence from Theorem \ref{THM_korespondence_leftkanpsmon_laxidemppsmon} restricts to the correspondence between left Kan 2-monads $(D, y)$ and lax-idempotent 2-monads $(D,m,i)$.
\end{prop}
\begin{proof}
``$\Rightarrow$”: Let $(D, y)$ be a left Kan 2-monad. As we outlined in the proof of Theorem \ref{THM_korespondence_leftkanpsmon_laxidemppsmon}, the pseudofunctor $D$ is defined as this left Kan extension:
% https://q.uiver.app/#q=WzAsNCxbMCwwLCJBIl0sWzIsMCwiREEiXSxbMCwxLCJCIl0sWzIsMSwiREIiXSxbMCwxLCJ5X0EiXSxbMSwzLCIoeV9CZileXFxtYmJkID06RGYiXSxbMCwyLCJmIiwyXSxbMiwzLCJ5X0IiLDJdXQ==
\[\begin{tikzcd}
	A && DA \\
	B && DB
	\arrow["{y_A}", from=1-1, to=1-3]
	\arrow["{(y_Bf)^\mbbd =:Df}", from=1-3, to=2-3]
	\arrow["f"', from=1-1, to=2-1]
	\arrow["{y_B}"', from=2-1, to=2-3]
\end{tikzcd}\]
If $(f:A \to B,g: B \to C)$ is a composable pair of morphisms, we have:
\[
D(g f) = (y_C g f)^\mbbd = ((y_C g)^\mbbd y_B f)^\mbbd = (y_C g)^\mbbd (y_B f)^\mbbd = Dg Df
\]
Also, $D1_A = y_A^\mbbd = 1_{DA}$ so $D$ is a 2-functor. This also makes $y$ a 2-natural transformation since the pseudo-naturality square is the identity. Next, the pseudo-naturality square for the multiplication $m: D^2 \Rightarrow D$ is also the identity since both of the triangles below commute:
% https://q.uiver.app/#q=WzAsNCxbMCwwLCJEXjJBIl0sWzAsMSwiRF4yQiJdLFsyLDAsIkRBIl0sWzIsMSwiREIiXSxbMCwxLCJEXjJmIiwyXSxbMSwzLCIxX3tEQn1eXFxtYmJkIiwyXSxbMCwyLCIxX3tEQX1eXFxtYmJkIl0sWzIsMywiRGYiXSxbMCwzLCIoRGYpXlxcbWJiZCIsMV1d
\[\begin{tikzcd}
	{D^2A} && DA \\
	{D^2B} && DB
	\arrow["{D^2f}"', from=1-1, to=2-1]
	\arrow["{1_{DB}^\mbbd}"', from=2-1, to=2-3]
	\arrow["{1_{DA}^\mbbd}", from=1-1, to=1-3]
	\arrow["Df", from=1-3, to=2-3]
	\arrow["{(Df)^\mbbd}"{description}, from=1-1, to=2-3]
\end{tikzcd}\]

``$\Leftarrow$”: If $(D,m,i)$ is a lax-idempotent 2-monad, the corresponding left Kan extension in the proof of Theorem \ref{THM_korespondence_leftkanpsmon_laxidemppsmon} has the 2-cell component equal to the identity. The other identities in Definition \ref{DEFI_leftkan2monad} are shown by a straightforward manipulation using the 2-monad identities.
\end{proof}

\subsection{Examples}\label{SUBS_2_examples}

\begin{pr}\label{PR_presheafpsmonad}
Given a locally small category $\ca$, denote by $P \ca$ the full subcategory of $[\ca^{op},\set]$ spanned by \textit{small} presheaves, that is, presheaves that are small colimits of representables. The assignment $\ca \mapsto P \ca$ defines a left Kan pseudomonad on the (large) 2-category $\CAT$ of locally small categories, with the unit $y_\ca: \ca \to P\ca$ being given by the Yoneda embedding and the extension operation being given by ordinary left Kan extension along $y_\ca$. These are guaranteed to exist because of the cocompleteness of $P \cb$; and since $y_\ca$ is fully faithful, the accompanying 2-cell is invertible):
% https://q.uiver.app/#q=WzAsMyxbMCwwLCJcXGNhIl0sWzIsMCwiUFxcY2EiXSxbMiwyLCJQXFxjYiJdLFswLDEsInlfQSJdLFsxLDIsIlxcdGV4dHtMYW59X3t5X0F9RiJdLFswLDIsIkYiLDJdLFs1LDQsIiIsMix7InNob3J0ZW4iOnsic291cmNlIjoyMCwidGFyZ2V0IjoyMH19XV0=
\[\begin{tikzcd}
	\ca && P\ca \\
	\\
	&& P\cb
	\arrow["{y_A}", from=1-1, to=1-3]
	\arrow[""{name=0, anchor=center, inner sep=0}, "{\text{Lan}_{y_A}F}", from=1-3, to=3-3]
	\arrow[""{name=1, anchor=center, inner sep=0}, "F"', from=1-1, to=3-3]
	\arrow[shorten <=7pt, shorten >=7pt, Rightarrow, from=1, to=0]
\end{tikzcd}\]

%272
A pseudo-$P$-algebra is precisely a cocomplete category, and pseudo-$P$-morphisms are cocontinuous functors. The Kleisli 2-category $\CAT_P$ thus has presheaf categories as objects and cocontinuous functors as morphisms. In fact, it can be seen to be biequivalent to the bicategory $\prof$ whose objects are locally small categories and whose morphisms $\ca \rightsquigarrow \cb$ are small profunctors $H: \cb^{op} \times \ca \to \set$. Here we call a profunctor $H: \cb^{op} \times \ca \to \set$ \textit{small} if for every $a \in \ca$, the presheaf $H(-,a): \cb^{op} \to \set$ is small (belongs to $P \cb$).

Under this identification, the left biadjoint from the Kleisli biadjunction in Proposition \ref{THM_biadjunkce_ck_ck_D} $P: \CAT \to \prof$ sends functor $f: \ca \to \cb$ to the profunctor:
\[
\cb(-,f-): \cb^{op} \times \ca \to \set.
\]
We remark that alternatively there is also a 2-monad presentation for this pseudomonad that uses inaccessible cardinals, see \cite[Chapter 7]{onthemonadicity}.
\end{pr}

%374 cauchy completion pseudomonad

\begin{pr}\label{PR_colax_slice_2cat}
Let $\ck$ be a 2-category with comma objects and fix an object $C \in \ck$. There is a 2-monad $P$ on $\ck / C$ that sends a morphism $f: A \to B$ to the morphism $\pi_f : Pg \to C$ which is a projection of the following comma object in $\ck$:
% https://q.uiver.app/#q=WzAsNCxbMCwwLCJQZiJdLFswLDEsIkMiXSxbMiwwLCJBIl0sWzIsMSwiQyJdLFswLDEsIlxccGlfZiIsMl0sWzAsMiwiXFxyaG9fZiJdLFsxLDMsIiIsMix7ImxldmVsIjoyLCJzdHlsZSI6eyJoZWFkIjp7Im5hbWUiOiJub25lIn19fV0sWzIsMywiZiJdLFs1LDYsIlxcY2hpIiwwLHsic2hvcnRlbiI6eyJzb3VyY2UiOjIwLCJ0YXJnZXQiOjIwfSwiZWRnZV9hbGlnbm1lbnQiOnsic291cmNlIjpmYWxzZSwidGFyZ2V0IjpmYWxzZX19XV0=
\begin{equation}\label{EQ_comma_object_lonely}
\begin{tikzcd}
	Pf && A \\
	C && C
	\arrow[""{name=0, anchor=center, inner sep=0}, "{\rho_f}", from=1-1, to=1-3]
	\arrow[""{name=0p, anchor=center, inner sep=0}, phantom, from=1-1, to=1-3, start anchor=center, end anchor=center]
	\arrow["{\pi_f}"', from=1-1, to=2-1]
	\arrow["f", from=1-3, to=2-3]
	\arrow[""{name=1, anchor=center, inner sep=0}, Rightarrow, no head, from=2-1, to=2-3]
	\arrow[""{name=1p, anchor=center, inner sep=0}, phantom, from=2-1, to=2-3, start anchor=center, end anchor=center]
	\arrow["\chi", shorten <=4pt, shorten >=4pt, Rightarrow, from=0p, to=1p]
\end{tikzcd}
\end{equation}

This 2-monad is known to be colax-idempotent with its algebras being fibrations in $\ck$ (\cite[Proposition 9]{streetfib}). Its Kleisli 2-category can be presented as having the objects functors with codomain $B$, while a morphism $F \rightsquigarrow G$ is a 1-cell $\theta: A \to Pg$ making the triangle below left commute:
% https://q.uiver.app/#q=WzAsOSxbMiwwLCJQZyJdLFsyLDEsIkMiXSxbNCwwLCJCIl0sWzQsMSwiQyJdLFswLDAsIkEiXSxbNiwwLCJBIl0sWzgsMCwiQiJdLFs2LDEsIkMiXSxbOCwxLCJDIl0sWzAsMSwiXFxwaV9nIiwyXSxbMCwyLCJcXHJob19nIl0sWzEsMywiIiwyLHsibGV2ZWwiOjIsInN0eWxlIjp7ImhlYWQiOnsibmFtZSI6Im5vbmUifX19XSxbMiwzLCJnIl0sWzQsMCwiXFx0aGV0YSJdLFs0LDEsImYiLDJdLFs1LDcsImYiLDJdLFs3LDgsIiIsMCx7ImxldmVsIjoyLCJzdHlsZSI6eyJoZWFkIjp7Im5hbWUiOiJub25lIn19fV0sWzUsNiwidSJdLFs2LDgsImciXSxbMTAsMTEsIlxcY2hpIiwwLHsic2hvcnRlbiI6eyJzb3VyY2UiOjIwLCJ0YXJnZXQiOjIwfSwiZWRnZV9hbGlnbm1lbnQiOnsic291cmNlIjpmYWxzZSwidGFyZ2V0IjpmYWxzZX19XSxbMTcsMTYsIlxcYWxwaGEiLDAseyJzaG9ydGVuIjp7InNvdXJjZSI6MjAsInRhcmdldCI6MjB9LCJlZGdlX2FsaWdubWVudCI6eyJzb3VyY2UiOmZhbHNlLCJ0YXJnZXQiOmZhbHNlfX1dXQ==
\[\begin{tikzcd}
	A && Pg && B && A && B \\
	&& C && C && C && C
	\arrow["\theta", from=1-1, to=1-3]
	\arrow["f"', from=1-1, to=2-3]
	\arrow[""{name=0, anchor=center, inner sep=0}, "{\rho_g}", from=1-3, to=1-5]
	\arrow[""{name=0p, anchor=center, inner sep=0}, phantom, from=1-3, to=1-5, start anchor=center, end anchor=center]
	\arrow["{\pi_g}"', from=1-3, to=2-3]
	\arrow["g", from=1-5, to=2-5]
	\arrow[""{name=1, anchor=center, inner sep=0}, "u", from=1-7, to=1-9]
	\arrow[""{name=1p, anchor=center, inner sep=0}, phantom, from=1-7, to=1-9, start anchor=center, end anchor=center]
	\arrow["f"', from=1-7, to=2-7]
	\arrow["g", from=1-9, to=2-9]
	\arrow[""{name=2, anchor=center, inner sep=0}, Rightarrow, no head, from=2-3, to=2-5]
	\arrow[""{name=2p, anchor=center, inner sep=0}, phantom, from=2-3, to=2-5, start anchor=center, end anchor=center]
	\arrow[""{name=3, anchor=center, inner sep=0}, Rightarrow, no head, from=2-7, to=2-9]
	\arrow[""{name=3p, anchor=center, inner sep=0}, phantom, from=2-7, to=2-9, start anchor=center, end anchor=center]
	\arrow["\chi", shorten <=4pt, shorten >=4pt, Rightarrow, from=0p, to=2p]
	\arrow["\alpha", shorten <=4pt, shorten >=4pt, Rightarrow, from=1p, to=3p]
\end{tikzcd}\]
From the definition of the comma object, this corresponds to pairs $(u,\alpha)$ of a 1-cell $u: A \to B$ and a 2-cell $\alpha: gu \Rightarrow f$ as portrayed above right.

\noindent In other words, the Kleisli 2-category for this 2-monad is isomorphic to the \textit{colax slice 2-category} $\ck // C$\footnote{This 2-category can also be presented as the 2-category of strict coalgebras and lax morphisms for the 2-comonad $(-)\times C$, see \cite[Chapter 5]{laxcomma2cats}.}. Under this identification, we have a 2-adjunction:
% https://q.uiver.app/#q=WzAsMixbMCwwLCJcXGNrL0MiXSxbMiwwLCJcXGNrLy9DIl0sWzAsMSwiSiIsMix7ImN1cnZlIjo0fV0sWzEsMCwiVSIsMix7ImN1cnZlIjo0fV0sWzIsMywiIiwwLHsibGV2ZWwiOjEsInN0eWxlIjp7Im5hbWUiOiJhZGp1bmN0aW9uIn19XV0=
\[\begin{tikzcd}
	{\ck/C} && {\ck//C}
	\arrow[""{name=0, anchor=center, inner sep=0}, "J"', curve={height=24pt}, from=1-1, to=1-3]
	\arrow[""{name=1, anchor=center, inner sep=0}, "U"', curve={height=24pt}, from=1-3, to=1-1]
	\arrow["\dashv"{anchor=center, rotate=90}, draw=none, from=0, to=1]
\end{tikzcd}\]

The left 2-adjoint is the canonical inclusion, the right 2-adjoint sends an object $f: A \to C$ to the comma object projection $\pi_f: Pf \to C$. The counit $p: JU \Rightarrow 1_{\ck //C}$ evaluated at an object $f: A \to C$ is the colax commutative triangle $(\rho_f, \chi): \pi_f \to f$ from \eqref{EQ_comma_object_lonely}.
\end{pr}

\noindent In the remainder of this section we recall a class of lax-idempotent 2-comonads that come from two-dimensional monad theory. Recall the 2-categories $\talgs$, $\talg$, $\talgl$ of strict algebras and strict, pseudo and lax morphisms for a 2-monad $T$ from \cite[1.2]{twodim}. Also recall the notions of a \textit{codescent object} and a \textit{lax codescent object} from \cite[Page 228]{codobjcoh}.

\begin{defi}\label{DEFI_resolution}
Let $T$ be a 2-monad on a 2-category $\ck$ and let $(A,a)$ be a strict $T$-algebra. By its \textit{resolution}, denoted $\res{A,a}$, we mean the following diagram in $\talgs$:
% https://q.uiver.app/?q=WzAsMyxbMCwwLCJUXjNBIl0sWzIsMCwiVF4yQSJdLFs0LDAsIlRBIl0sWzIsMSwiVGlfQSIsMV0sWzEsMiwiVGEiLDEseyJvZmZzZXQiOjR9XSxbMSwyLCJtX0EiLDEseyJvZmZzZXQiOi00fV0sWzAsMSwibV97VF4yQX0iLDEseyJvZmZzZXQiOi00fV0sWzAsMSwiVG1fe1RBfSIsMV0sWzAsMSwiVF4yYSIsMSx7Im9mZnNldCI6NH1dXQ==
\[\begin{tikzcd}
	{T^3A} && {T^2A} && TA
	\arrow["{Ti_A}"{description}, from=1-5, to=1-3]
	\arrow["Ta"{description}, shift right=4, from=1-3, to=1-5]
	\arrow["{m_A}"{description}, shift left=4, from=1-3, to=1-5]
	\arrow["{m_{T^2A}}"{description}, shift left=4, from=1-1, to=1-3]
	\arrow["{Tm_{TA}}"{description}, from=1-1, to=1-3]
	\arrow["{T^2a}"{description}, shift right=4, from=1-1, to=1-3]
\end{tikzcd}\]
\end{defi}

\begin{theo}\label{THM_adjoint_talgs_talgl}
Let $T$ be a 2-monad on a 2-category $\ck$ and assume the 2-category $\talgs$ admits lax codescent objects of resolutions of strict algebras. Then the inclusion 2-functor $\talgs \to \talgl$ admits a left 2-adjoint. Similarly, assume the 2-category $\talgs$ admits codescent objects of resolutions of strict algebras. Then the inclusion 2-functor $\talgs \to \talg$ admits a left 2-adjoint:
% https://q.uiver.app/#q=WzAsNCxbMCwwLCJcXHRhbGdzIl0sWzIsMCwiXFx0YWxnbCJdLFs0LDAsIlxcdGFsZ3MiXSxbNiwwLCJcXHRhbGciXSxbMCwxLCJKIiwyLHsiY3VydmUiOjQsInN0eWxlIjp7InRhaWwiOnsibmFtZSI6Imhvb2siLCJzaWRlIjoidG9wIn19fV0sWzEsMCwiKC0pJyIsMix7ImN1cnZlIjo0LCJzdHlsZSI6eyJib2R5Ijp7Im5hbWUiOiJkYXNoZWQifX19XSxbMiwzLCJKX3AiLDIseyJjdXJ2ZSI6NCwic3R5bGUiOnsidGFpbCI6eyJuYW1lIjoiaG9vayIsInNpZGUiOiJ0b3AifX19XSxbMywyLCIoLSleXFxkYWdnZXIiLDIseyJjdXJ2ZSI6NCwic3R5bGUiOnsiYm9keSI6eyJuYW1lIjoiZGFzaGVkIn19fV0sWzUsNCwiIiwyLHsibGV2ZWwiOjEsInN0eWxlIjp7Im5hbWUiOiJhZGp1bmN0aW9uIn19XSxbNyw2LCIiLDAseyJsZXZlbCI6MSwic3R5bGUiOnsibmFtZSI6ImFkanVuY3Rpb24ifX1dXQ==
\[\begin{tikzcd}[cramped]
	\talgs && \talgl && \talgs && \talg
	\arrow[""{name=0, anchor=center, inner sep=0}, "J"', curve={height=24pt}, hook, from=1-1, to=1-3]
	\arrow[""{name=1, anchor=center, inner sep=0}, "{(-)'}"', curve={height=24pt}, dashed, from=1-3, to=1-1]
	\arrow[""{name=2, anchor=center, inner sep=0}, "{J_p}"', curve={height=24pt}, hook, from=1-5, to=1-7]
	\arrow[""{name=3, anchor=center, inner sep=0}, "{(-)^\dagger}"', curve={height=24pt}, dashed, from=1-7, to=1-5]
	\arrow["\dashv"{anchor=center, rotate=-90}, draw=none, from=1, to=0]
	\arrow["\dashv"{anchor=center, rotate=-90}, draw=none, from=3, to=2]
\end{tikzcd}\]

In the first case, the value of a left 2-adjoint at a $T$-algebra $(A,a)$ is given by the lax codescent object of the diagram $\res{A,a}$ in $\talgs$. In the second case, codescent object is used.
\end{theo}
\begin{proof}
See Lemma 3.2 and Theorem 2.6 in \cite{codobjcoh}.
\end{proof}

\begin{pozn}
The assumptions of Theorem \ref{THM_adjoint_talgs_talgl} are satisfied whenever the base 2-category $\ck$ is cocomplete and $T$ is finitary (preserves filtered colimits). This is because the codescent objects of a resolution of a strict algebra is \textit{reflexive}, and so is a filtered colimit by \cite[Proposition 4.3]{codobjcoh}.
\end{pozn}

\begin{defi}\label{DEFI_Ql_Qp}
We denote by $Q_l$ and $Q_p$ the 2-comonads generated by the 2-adjunctions in the above theorem and call them the \textit{lax morphism classifier 2-comonad} and the \textit{pseudo morphism classifier 2-comonad}.
\end{defi}

\noindent It is easy to see that $\talgl$ is isomorphic to $(\talgs)_{Q_l}$, the Kleisli 2-category for the 2-comonad $Q_l$. Similarly, $\talg \cong (\talgs)_{Q_p}$.

\medskip

\begin{prop}\label{THM_Ql_lax_idemp_Qp_ps_idemp}
Let $T$ be a 2-monad on a 2-category $\ck$ such that the left 2-adjoints to the inclusions $\talgs \xhookrightarrow{} \talg$, $\talgs \xhookrightarrow{} \talgl$ exist. Then:
\begin{itemize}
\item If $\ck$ admits oplax limits of arrows, $Q_l$ is lax-idempotent.
\item If $\ck$ admits pseudo limits of arrows, $Q_p$ is pseudo-idempotent.
\end{itemize}
\end{prop}
\begin{proof}
See \cite[Lemma 2.5]{enhanced2cats}.
\end{proof}

\begin{prop}\label{THM_mor_of_comonads_QlQp}
There is a morphism of 2-comonads $Q_l \to Q_p$.
\end{prop}
\begin{proof}

Denote the units of the adjunctions in Theorem \ref{THM_adjoint_talgs_talgl} by $p_A: A \rightsquigarrow A'$ and\\ $p_A^\dagger: A \rightsquigarrow A^\dagger$ respectively. Since $p^\dagger_A$ is a pseudo-morphism, it is in particular a lax morphism and thus there exists a unique strict $T$-algebra morphism $\theta_A : A' \to A^\dagger$ making the diagram commute:
% https://q.uiver.app/#q=WzAsMyxbMCwxLCJBIl0sWzIsMCwiQSciXSxbMiwxLCJBXlxcZGFnZ2VyIl0sWzEsMiwiXFxleGlzdHMgISBcXHRoZXRhX0EiLDAseyJzdHlsZSI6eyJib2R5Ijp7Im5hbWUiOiJkYXNoZWQifX19XSxbMCwxLCJwX0EiLDAseyJzdHlsZSI6eyJib2R5Ijp7Im5hbWUiOiJzcXVpZ2dseSJ9fX1dLFswLDIsInBeXFxkYWdnZXJfQSAiLDIseyJzdHlsZSI6eyJib2R5Ijp7Im5hbWUiOiJzcXVpZ2dseSJ9fX1dXQ==
\[\begin{tikzcd}
	&& {A'} \\
	A && {A^\dagger}
	\arrow["{\exists ! \theta_A}", dashed, from=1-3, to=2-3]
	\arrow["{p_A}", squiggly, from=2-1, to=1-3]
	\arrow["{p^\dagger_A }"', squiggly, from=2-1, to=2-3]
\end{tikzcd}\]

Using the universal property of $A'$, it is readily seen that the maps $\theta_A$'s assemble into a morphism of 2-comonads.
\end{proof}

\section{Relative Kan extensions and colax adjunctions}\label{SEKCE_Bunge}

In \cite{bunge}, Bunge introduced the notion of a relative Kan extensions with respect to a 2-functor $U$ and showed that for a collection $y_A: A \to UFA$ of 1-cells that admit these extensions (and satisfy certain coherence conditions), there is an induced left colax adjoint $F$ to $U$, where $F$ is a colax functor (\cite[Theorem 4.1]{bunge}). She also proves a partial converse to this result (\cite[Theorem 4.3]{bunge}). Note that at the same time these results also appeared in Gray's work (\cite[I,7.8.]{formalcat}).

In this section, we generalize these results to the case where $U$ is a pseudofunctor and, on the other hand, refine it by identifying conditions under which the colax left adjoint $F$ is actually a pseudofunctor. This enables us to describe, in Theorem \ref{THM_korespondence_lax_adj_u_coh_ext}, a symmetric relationship between $U$-extensions and colax adjunctions. We will see an application of these results to the settings of lax-idempotent pseudomonads in Section \ref{SEKCE_Kleisli2cat}.

\begin{defi}
Let $U: \cc \to \cd$ be a pseudofunctor, $y_A: A \to UFA$, $f: A \to UB$ 1-cells of $\cd$. The \textit{left $U$-extension} of $f$ along $y_A$ is a pair $(f',\psi_f)$ with the property that for any pair $(g,\alpha)$ pictured below, there is a unique 2-cell $\theta: f' \Rightarrow g$ such that the following 2-cells are equal:
% https://q.uiver.app/#q=WzAsNyxbMCwwLCJBIl0sWzIsMCwiVUZBIl0sWzIsMiwiVUIiXSxbNCwxLCI9Il0sWzUsMCwiQSJdLFs3LDAsIlVGQSJdLFs3LDIsIlVCIl0sWzAsMSwieV9BIl0sWzAsMiwiZiIsMl0sWzEsMiwiVWYnIiwxLHsibGFiZWxfcG9zaXRpb24iOjcwfV0sWzEsMiwiVWciLDAseyJjdXJ2ZSI6LTR9XSxbNCw1LCJ5X0EiXSxbNCw2LCJmIiwyXSxbNSw2LCJVZyJdLFs4LDEsIlxccHNpX2YiLDIseyJzaG9ydGVuIjp7InNvdXJjZSI6MjB9fV0sWzEyLDUsIlxcYWxwaGEiLDAseyJzaG9ydGVuIjp7InNvdXJjZSI6MjB9fV0sWzksMTAsIlVcXHRoZXRhIiwwLHsic2hvcnRlbiI6eyJzb3VyY2UiOjIwLCJ0YXJnZXQiOjIwfX1dXQ==
\[\begin{tikzcd}
	A && UFA &&& A && UFA \\
	&&&& {=} \\
	&& UB &&&&& UB
	\arrow["{y_A}", from=1-1, to=1-3]
	\arrow[""{name=0, anchor=center, inner sep=0}, "f"', from=1-1, to=3-3]
	\arrow[""{name=1, anchor=center, inner sep=0}, "{Uf'}"{description, pos=0.7}, from=1-3, to=3-3]
	\arrow[""{name=2, anchor=center, inner sep=0}, "Ug", curve={height=-24pt}, from=1-3, to=3-3]
	\arrow["{y_A}", from=1-6, to=1-8]
	\arrow[""{name=3, anchor=center, inner sep=0}, "f"', from=1-6, to=3-8]
	\arrow["Ug", from=1-8, to=3-8]
	\arrow["{\psi_f}"', shorten <=5pt, Rightarrow, from=0, to=1-3]
	\arrow["\alpha", shorten <=5pt, Rightarrow, from=3, to=1-8]
	\arrow["U\theta", shorten <=5pt, shorten >=5pt, Rightarrow, from=1, to=2]
\end{tikzcd}\]
\end{defi}

\begin{defi}\label{DEFI_coherently_closed_for_Uext}
Let $U: \cc \to \cd$ be a pseudofunctor. We say that a collection of 1-cells $y_A: A \to UFA$ for each object $A \in \cc$ are \textit{coherently closed for $U$-extensions} if:
\begin{itemize}
\item for every $f: A \to UB$ we have a choice of an $U$-extension $(f^\mbbd, \mbbd_f)$,
\item the following composite 2-cell exhibits $1_Y^\mbbd \circ (y_Y f)^\mbbd$ as the left $U$-extension of $f$ along $y_X$:
% https://q.uiver.app/#q=WzAsNSxbMCwwLCJYIl0sWzMsMCwiVUZYIl0sWzAsMiwiVVkiXSxbMywyLCJVRlVZIl0sWzMsNCwiVVkiXSxbMCwxLCJ5X1giXSxbMCwyLCJmIiwyXSxbMiwzLCJ5X3tVWX0iLDFdLFsxLDMsIlUoeV97VVl9IGYpXlxcbWJiZCIsMV0sWzIsNCwiIiwwLHsibGV2ZWwiOjIsInN0eWxlIjp7ImhlYWQiOnsibmFtZSI6Im5vbmUifX19XSxbMyw0LCJVMV97VVl9XlxcbWJiZCIsMV0sWzEsNCwiVSgxX3tVWX1eXFxtYmJkIFxcY2lyYyAoeV97VVl9IGYpXlxcbWJiZCkiLDAseyJvZmZzZXQiOi01LCJjdXJ2ZSI6LTV9XSxbMSw0LCIiLDAseyJjdXJ2ZSI6LTEsInN0eWxlIjp7ImJvZHkiOnsibmFtZSI6Im5vbmUifSwiaGVhZCI6eyJuYW1lIjoibm9uZSJ9fX1dLFs3LDUsIlxcbWJiZF97eV97VVl9IGZ9IiwyLHsic2hvcnRlbiI6eyJzb3VyY2UiOjIwLCJ0YXJnZXQiOjIwfSwiZWRnZV9hbGlnbm1lbnQiOnsic291cmNlIjpmYWxzZSwidGFyZ2V0IjpmYWxzZX19XSxbOSw3LCJcXG1iYmRfezFfe1VZfX0iLDIseyJzaG9ydGVuIjp7InNvdXJjZSI6MjAsInRhcmdldCI6MjB9LCJlZGdlX2FsaWdubWVudCI6eyJzb3VyY2UiOmZhbHNlLCJ0YXJnZXQiOmZhbHNlfX1dLFsxMiwxMSwiXFxnYW1tYV57LTF9IiwwLHsibGFiZWxfcG9zaXRpb24iOjcwLCJzaG9ydGVuIjp7InNvdXJjZSI6NDB9LCJlZGdlX2FsaWdubWVudCI6eyJ0YXJnZXQiOmZhbHNlfX1dXQ==
\begin{equation}\label{EQ_DEFI_coherently_closed_for_Uext}
\begin{tikzcd}
	X &&& UFX \\
	\\
	UY &&& UFUY \\
	\\
	&&& UY
	\arrow[""{name=0, anchor=center, inner sep=0}, "{y_X}", from=1-1, to=1-4]
	\arrow[""{name=0p, anchor=center, inner sep=0}, phantom, from=1-1, to=1-4, start anchor=center, end anchor=center]
	\arrow["f"', from=1-1, to=3-1]
	\arrow[""{name=1, anchor=center, inner sep=0}, "{y_{UY}}"{description}, from=3-1, to=3-4]
	\arrow[""{name=1p, anchor=center, inner sep=0}, phantom, from=3-1, to=3-4, start anchor=center, end anchor=center]
	\arrow[""{name=1p, anchor=center, inner sep=0}, phantom, from=3-1, to=3-4, start anchor=center, end anchor=center]
	\arrow["{U(y_{UY} f)^\mbbd}"{description}, from=1-4, to=3-4]
	\arrow[""{name=2, anchor=center, inner sep=0}, Rightarrow, no head, from=3-1, to=5-4]
	\arrow[""{name=2p, anchor=center, inner sep=0}, phantom, from=3-1, to=5-4, start anchor=center, end anchor=center]
	\arrow["{U1_{UY}^\mbbd}"{description}, from=3-4, to=5-4]
	\arrow[""{name=3, anchor=center, inner sep=0}, "{U(1_{UY}^\mbbd \circ (y_{UY} f)^\mbbd)}", shift left=5, curve={height=-30pt}, from=1-4, to=5-4]
	\arrow[""{name=3p, anchor=center, inner sep=0}, phantom, from=1-4, to=5-4, start anchor=center, end anchor=center, shift left=5, curve={height=-30pt}]
	\arrow[""{name=4, anchor=center, inner sep=0}, curve={height=-6pt}, draw=none, from=1-4, to=5-4]
	\arrow["{\mbbd_{y_{UY} f}}"', shorten <=9pt, shorten >=9pt, Rightarrow, from=1p, to=0p]
	\arrow["{\mbbd_{1_{UY}}}"', shorten <=4pt, shorten >=4pt, Rightarrow, from=2p, to=1p]
	\arrow["{\gamma^{-1}}"{pos=0.7}, shorten <=14pt, Rightarrow, from=4, to=3p]
\end{tikzcd}
\end{equation}
\end{itemize}
\end{defi}

\begin{theo}\label{THM_Bunge}
Let $U: \cc \to \cd$ be a pseudofunctor and $y_A : A \to UFA$ a collection of 1-cells coherently closed for $U$-extensions. Then:
\begin{itemize}
\item the mapping $A \mapsto FA$ can be extended to a colax functor $F: \cd \to \cc$,
\item $y$ can be extended to a colax natural transformation $1_\cd \Rightarrow UF$,
\item there exists a colax-natural transformation $\epsilon: FU \Rightarrow 1_\cc$ and a modification\\ $\Phi: 1_U \to U\epsilon \circ yU$.

\medskip

\noindent Assume moreover the \textit{composition} and \textit{unit} axioms for $U$-extensions: the diagram below left is a $U$-extension of $y_A$ along $y_A$, and the diagram below right is the $U$-extension of $y_C gf$ along $y_A$:
% https://q.uiver.app/#q=WzAsOSxbNCwwLCJBIl0sWzYsMCwiVUZBIl0sWzQsMiwiQiJdLFs2LDIsIlVGQiJdLFs0LDQsIkMiXSxbNiw0LCJVRkMiXSxbMCwwLCJBIl0sWzIsMCwiVUZBIl0sWzIsMiwiVUZBIl0sWzAsMSwieV9BIl0sWzEsMywiVUZmIl0sWzIsMywieV9CIiwxXSxbMCwyLCJmIiwyXSxbMiw0LCJnIiwyXSxbNCw1LCJ5X0MiLDJdLFszLDUsIlVGZyJdLFs2LDcsInlfQSJdLFs2LDgsInlfQSIsMix7Im9mZnNldCI6M31dLFs3LDgsIlUxX3tGQX0iLDAseyJjdXJ2ZSI6LTR9XSxbNyw4LCIiLDEseyJjdXJ2ZSI6NCwibGV2ZWwiOjIsInN0eWxlIjp7ImhlYWQiOnsibmFtZSI6Im5vbmUifX19XSxbMSw1LCJVKEZnIFxcY2lyYyBGZikiLDAseyJvZmZzZXQiOi01LCJjdXJ2ZSI6LTV9XSxbMTEsOSwiXFxtYmJkIiwwLHsic2hvcnRlbiI6eyJzb3VyY2UiOjIwLCJ0YXJnZXQiOjIwfX1dLFsxNCwxMSwiXFxtYmJkIiwwLHsic2hvcnRlbiI6eyJzb3VyY2UiOjIwLCJ0YXJnZXQiOjIwfX1dLFsxOSwxOCwiXFxpb3RhXnstMX0iLDAseyJzaG9ydGVuIjp7InNvdXJjZSI6MjAsInRhcmdldCI6MjB9fV0sWzMsMjAsIlxcZ2FtbWFeey0xfSIsMCx7InNob3J0ZW4iOnsidGFyZ2V0IjoyMH19XV0=
\begin{equation}\label{EQ_U_extensions_unit_comp}
\begin{tikzcd}
	A && UFA && A && UFA \\
	&&&& B && UFB \\
	&& UFA && C && UFC
	\arrow[""{name=0, anchor=center, inner sep=0}, "{y_A}", from=1-5, to=1-7]
	\arrow["U(y_Bf)^\mbbd"', from=1-7, to=2-7]
	\arrow["f"', from=1-5, to=2-5]
	\arrow["g"', from=2-5, to=3-5]
	\arrow[""{name=1, anchor=center, inner sep=0}, "{y_C}"', from=3-5, to=3-7]
	\arrow["U(y_Cg)^\mbbd"', from=2-7, to=3-7]
	\arrow["{y_A}", from=1-1, to=1-3]
	\arrow["{y_A}"', shift right=3, from=1-1, to=3-3]
	\arrow[""{name=2, anchor=center, inner sep=0}, "{U1_{FA}}", curve={height=-24pt}, from=1-3, to=3-3]
	\arrow[""{name=3, anchor=center, inner sep=0}, curve={height=24pt}, Rightarrow, no head, from=1-3, to=3-3]
	\arrow[""{name=4, anchor=center, inner sep=0}, "{U((y_Cg)^\mbbd (y_Bf)^\mbbd)}", shift left=5, curve={height=-30pt}, from=1-7, to=3-7]
	\arrow[""{name=4p, anchor=center, inner sep=0}, phantom, from=1-7, to=3-7, start anchor=center, end anchor=center, shift left=5, curve={height=-30pt}]
	\arrow[""{name=5, anchor=center, inner sep=0}, "{y_B}"{description}, from=2-5, to=2-7]
	\arrow["{\iota^{-1}}", shorten <=10pt, shorten >=10pt, Rightarrow, from=3, to=2]
	\arrow["\mbbd", shorten <=4pt, shorten >=4pt, Rightarrow, from=5, to=0]
	\arrow["\mbbd", shorten <=4pt, shorten >=4pt, Rightarrow, from=1, to=5]
	\arrow["{\gamma^{-1}}", shorten >=6pt, Rightarrow, from=2-7, to=4p]
\end{tikzcd}
\end{equation}

\noindent Then:
\item $F$ is a pseudofunctor,
\item there is an invertible modification $\Psi: \epsilon F \circ Fy \to 1_F$ and all this data give a colax adjunction:
\[
(\Psi,\Phi): (\epsilon, y): F \ddashv U: \cc \to \cd.
\]
\end{itemize}
\end{theo}
\begin{proof}
Denote by $(f^\mbbd,\mbbd)$ the choice of a $U$-extension of $f: A \to UB$. Define the colax functor $F: \cd \to \cc$ on a morphism $f: A \to B$ as the following $U$-extension:
% https://q.uiver.app/#q=WzAsNCxbMCwwLCJBIl0sWzIsMCwiVUZBIl0sWzAsMiwiQiJdLFsyLDIsIlVGQiJdLFswLDEsInlfQSJdLFsxLDMsIlVGZiJdLFsyLDMsInlfQiIsMl0sWzAsMiwiZiIsMl0sWzYsNCwiXFxtYmJkIiwyLHsic2hvcnRlbiI6eyJzb3VyY2UiOjIwLCJ0YXJnZXQiOjIwfX1dXQ==
\[\begin{tikzcd}
	A && UFA \\
	\\
	B && UFB
	\arrow[""{name=0, anchor=center, inner sep=0}, "{y_A}", from=1-1, to=1-3]
	\arrow["UFf", from=1-3, to=3-3]
	\arrow[""{name=1, anchor=center, inner sep=0}, "{y_B}"', from=3-1, to=3-3]
	\arrow["f"', from=1-1, to=3-1]
	\arrow["\mbbd"', shorten <=9pt, shorten >=9pt, Rightarrow, from=1, to=0]
\end{tikzcd}\]
Define the action of $F$ on a 2-cell $\alpha$ as the unique 2-cell making the following equal:
% https://q.uiver.app/#q=WzAsOSxbMCwwLCJBIl0sWzIsMCwiVUZBIl0sWzAsMiwiQiJdLFsyLDIsIlVGQiJdLFs0LDEsIj0iXSxbNiwwLCJBIl0sWzYsMiwiQiJdLFs4LDAsIlVGQSJdLFs4LDIsIlVGQiJdLFswLDEsInlfQSJdLFsxLDMsIlVGZiIsMl0sWzIsMywieV9CIiwyXSxbMCwyLCJmIiwyXSxbMSwzLCJVRmciLDAseyJjdXJ2ZSI6LTV9XSxbNSw2LCJmIiwyLHsiY3VydmUiOjV9XSxbNSw2LCJnIl0sWzUsNywieV9BIl0sWzYsOCwieV9CIiwyXSxbNyw4LCJVRmciXSxbMTEsOSwiXFxtYmJkIiwwLHsic2hvcnRlbiI6eyJzb3VyY2UiOjIwLCJ0YXJnZXQiOjIwfSwiZWRnZV9hbGlnbm1lbnQiOnsic291cmNlIjpmYWxzZSwidGFyZ2V0IjpmYWxzZX19XSxbMTAsMTMsIlUoXFxleGlzdHMgISkiLDAseyJzaG9ydGVuIjp7InNvdXJjZSI6MTAsInRhcmdldCI6MTB9LCJlZGdlX2FsaWdubWVudCI6eyJzb3VyY2UiOmZhbHNlLCJ0YXJnZXQiOmZhbHNlfX1dLFsxNCwxNSwiXFxhbHBoYSIsMCx7ImxhYmVsX3Bvc2l0aW9uIjo0MCwic2hvcnRlbiI6eyJzb3VyY2UiOjEwLCJ0YXJnZXQiOjEwfSwiZWRnZV9hbGlnbm1lbnQiOnsic291cmNlIjpmYWxzZSwidGFyZ2V0IjpmYWxzZX19XSxbMTcsMTYsIlxcbWJiZCIsMix7InNob3J0ZW4iOnsic291cmNlIjoyMCwidGFyZ2V0IjoyMH0sImVkZ2VfYWxpZ25tZW50Ijp7InNvdXJjZSI6ZmFsc2UsInRhcmdldCI6ZmFsc2V9fV1d
\begin{equation}\label{EQP_Falpha}
\begin{tikzcd}
	A && UFA &&&& A && UFA \\
	&&&& {=} \\
	B && UFB &&&& B && UFB
	\arrow[""{name=0, anchor=center, inner sep=0}, "{y_A}", from=1-1, to=1-3]
	\arrow[""{name=0p, anchor=center, inner sep=0}, phantom, from=1-1, to=1-3, start anchor=center, end anchor=center]
	\arrow[""{name=1, anchor=center, inner sep=0}, "UFf"', from=1-3, to=3-3]
	\arrow[""{name=1p, anchor=center, inner sep=0}, phantom, from=1-3, to=3-3, start anchor=center, end anchor=center]
	\arrow[""{name=2, anchor=center, inner sep=0}, "{y_B}"', from=3-1, to=3-3]
	\arrow[""{name=2p, anchor=center, inner sep=0}, phantom, from=3-1, to=3-3, start anchor=center, end anchor=center]
	\arrow["f"', from=1-1, to=3-1]
	\arrow[""{name=3, anchor=center, inner sep=0}, "UFg", curve={height=-30pt}, from=1-3, to=3-3]
	\arrow[""{name=3p, anchor=center, inner sep=0}, phantom, from=1-3, to=3-3, start anchor=center, end anchor=center, curve={height=-30pt}]
	\arrow[""{name=4, anchor=center, inner sep=0}, "f"', curve={height=30pt}, from=1-7, to=3-7]
	\arrow[""{name=4p, anchor=center, inner sep=0}, phantom, from=1-7, to=3-7, start anchor=center, end anchor=center, curve={height=30pt}]
	\arrow[""{name=5, anchor=center, inner sep=0}, "g", from=1-7, to=3-7]
	\arrow[""{name=5p, anchor=center, inner sep=0}, phantom, from=1-7, to=3-7, start anchor=center, end anchor=center]
	\arrow[""{name=6, anchor=center, inner sep=0}, "{y_A}", from=1-7, to=1-9]
	\arrow[""{name=6p, anchor=center, inner sep=0}, phantom, from=1-7, to=1-9, start anchor=center, end anchor=center]
	\arrow[""{name=7, anchor=center, inner sep=0}, "{y_B}"', from=3-7, to=3-9]
	\arrow[""{name=7p, anchor=center, inner sep=0}, phantom, from=3-7, to=3-9, start anchor=center, end anchor=center]
	\arrow["UFg", from=1-9, to=3-9]
	\arrow["\mbbd", shorten <=9pt, shorten >=9pt, Rightarrow, from=2p, to=0p]
	\arrow["{U(\exists !)}", shorten <=3pt, shorten >=3pt, Rightarrow, from=1p, to=3p]
	\arrow["\alpha"{pos=0.4}, shorten <=3pt, shorten >=3pt, Rightarrow, from=4p, to=5p]
	\arrow["\mbbd"', shorten <=9pt, shorten >=9pt, Rightarrow, from=7p, to=6p]
\end{tikzcd}
\end{equation}

The above equation makes $y$ locally natural. The associator $\gamma': F(gf) \Rightarrow Fg \circ Ff$ and the unitor $\iota': F1_A \Rightarrow 1_{FA}$ for $F$ are given as the unique 2-cells satisfying these equations:
% https://q.uiver.app/#q=WzAsMTEsWzUsMCwiQSJdLFs3LDAsIlVGQSJdLFs1LDEsIkIiXSxbNywxLCJVRkIiXSxbNSwyLCJDIl0sWzcsMiwiVUZDIl0sWzAsMCwiQSJdLFswLDIsIkMiXSxbMiwyLCJVRkMiXSxbMiwwLCJVRkEiXSxbNCwxLCI9Il0sWzAsMSwieV9BIl0sWzEsMywiVUZmIiwyXSxbMiwzLCJ5X0IiLDFdLFswLDIsImYiLDJdLFsyLDQsImciLDJdLFs0LDUsInlfQyIsMl0sWzMsNSwiVUZnIiwyXSxbMSw1LCJVKEZnIFxcY2lyYyBGZikiLDAseyJvZmZzZXQiOi0zLCJjdXJ2ZSI6LTV9XSxbNiw3LCJnZiIsMl0sWzcsOCwieV9DIiwyXSxbNiw5LCJ5X0EiXSxbOSw4LCJVRihnZikiLDJdLFs5LDgsIlUoRmcgXFxjaXJjIEZmKSIsMCx7ImxhYmVsX3Bvc2l0aW9uIjozMCwiY3VydmUiOi01fV0sWzEsNSwiIiwwLHsib2Zmc2V0IjotNSwic3R5bGUiOnsiYm9keSI6eyJuYW1lIjoibm9uZSJ9LCJoZWFkIjp7Im5hbWUiOiJub25lIn19fV0sWzEzLDExLCJcXG1iYmQiLDAseyJzaG9ydGVuIjp7InNvdXJjZSI6MjAsInRhcmdldCI6MjB9LCJlZGdlX2FsaWdubWVudCI6eyJzb3VyY2UiOmZhbHNlLCJ0YXJnZXQiOmZhbHNlfX1dLFsyMCwyMSwiXFxtYmJkIiwwLHsic2hvcnRlbiI6eyJzb3VyY2UiOjIwLCJ0YXJnZXQiOjIwfSwiZWRnZV9hbGlnbm1lbnQiOnsic291cmNlIjpmYWxzZSwidGFyZ2V0IjpmYWxzZX19XSxbMjIsMjMsIlUoXFxleGlzdHMgISkiLDAseyJzaG9ydGVuIjp7InNvdXJjZSI6MjAsInRhcmdldCI6MjB9LCJlZGdlX2FsaWdubWVudCI6eyJzb3VyY2UiOmZhbHNlLCJ0YXJnZXQiOmZhbHNlfX1dLFsxNiwxMywiXFxtYmJkIiwwLHsic2hvcnRlbiI6eyJzb3VyY2UiOjIwLCJ0YXJnZXQiOjIwfSwiZWRnZV9hbGlnbm1lbnQiOnsic291cmNlIjpmYWxzZSwidGFyZ2V0IjpmYWxzZX19XSxbMjQsMTgsIlxcZ2FtbWFeey0xfSIsMCx7InNob3J0ZW4iOnsic291cmNlIjoxMH0sImVkZ2VfYWxpZ25tZW50Ijp7InNvdXJjZSI6ZmFsc2UsInRhcmdldCI6ZmFsc2V9fV1d
\begin{equation}\label{EQP_gamma_prime}
\begin{tikzcd}
	A && UFA &&& A && UFA \\
	&&&& {=} & B && UFB \\
	C && UFC &&& C && UFC
	\arrow[""{name=0, anchor=center, inner sep=0}, "{y_A}", from=1-6, to=1-8]
	\arrow[""{name=0p, anchor=center, inner sep=0}, phantom, from=1-6, to=1-8, start anchor=center, end anchor=center]
	\arrow["UFf"', from=1-8, to=2-8]
	\arrow[""{name=1, anchor=center, inner sep=0}, "{y_B}"{description}, from=2-6, to=2-8]
	\arrow[""{name=1p, anchor=center, inner sep=0}, phantom, from=2-6, to=2-8, start anchor=center, end anchor=center]
	\arrow[""{name=1p, anchor=center, inner sep=0}, phantom, from=2-6, to=2-8, start anchor=center, end anchor=center]
	\arrow["f"', from=1-6, to=2-6]
	\arrow["g"', from=2-6, to=3-6]
	\arrow[""{name=2, anchor=center, inner sep=0}, "{y_C}"', from=3-6, to=3-8]
	\arrow[""{name=2p, anchor=center, inner sep=0}, phantom, from=3-6, to=3-8, start anchor=center, end anchor=center]
	\arrow["UFg"', from=2-8, to=3-8]
	\arrow[""{name=3, anchor=center, inner sep=0}, "{U(Fg \circ Ff)}", shift left=3, curve={height=-30pt}, from=1-8, to=3-8]
	\arrow[""{name=3p, anchor=center, inner sep=0}, phantom, from=1-8, to=3-8, start anchor=center, end anchor=center, shift left=3, curve={height=-30pt}]
	\arrow["gf"', from=1-1, to=3-1]
	\arrow[""{name=4, anchor=center, inner sep=0}, "{y_C}"', from=3-1, to=3-3]
	\arrow[""{name=4p, anchor=center, inner sep=0}, phantom, from=3-1, to=3-3, start anchor=center, end anchor=center]
	\arrow[""{name=5, anchor=center, inner sep=0}, "{y_A}", from=1-1, to=1-3]
	\arrow[""{name=5p, anchor=center, inner sep=0}, phantom, from=1-1, to=1-3, start anchor=center, end anchor=center]
	\arrow[""{name=6, anchor=center, inner sep=0}, "{UF(gf)}"', from=1-3, to=3-3]
	\arrow[""{name=6p, anchor=center, inner sep=0}, phantom, from=1-3, to=3-3, start anchor=center, end anchor=center]
	\arrow[""{name=7, anchor=center, inner sep=0}, "{U(Fg \circ Ff)}"{pos=0.3}, curve={height=-30pt}, from=1-3, to=3-3]
	\arrow[""{name=7p, anchor=center, inner sep=0}, phantom, from=1-3, to=3-3, start anchor=center, end anchor=center, curve={height=-30pt}]
	\arrow[""{name=8, anchor=center, inner sep=0}, shift left=5, draw=none, from=1-8, to=3-8]
	\arrow[""{name=8p, anchor=center, inner sep=0}, phantom, from=1-8, to=3-8, start anchor=center, end anchor=center, shift left=5]
	\arrow["\mbbd", shorten <=4pt, shorten >=4pt, Rightarrow, from=1p, to=0p]
	\arrow["\mbbd", shorten <=9pt, shorten >=9pt, Rightarrow, from=4p, to=5p]
	\arrow["{U(\exists !)}", shorten <=6pt, shorten >=6pt, Rightarrow, from=6p, to=7p]
	\arrow["\mbbd", shorten <=4pt, shorten >=4pt, Rightarrow, from=2p, to=1p]
	\arrow["{\gamma^{-1}}", shorten <=3pt, Rightarrow, from=8p, to=3p]
\end{tikzcd}
\end{equation}

% https://q.uiver.app/#q=WzAsOSxbNSwwLCJBIl0sWzcsMCwiVUZBIl0sWzUsMiwiQSJdLFs3LDIsIlVGQSJdLFs0LDEsIj0iXSxbMCwwLCJBIl0sWzAsMiwiQSJdLFsyLDAsIlVGQSJdLFsyLDIsIlVGQSJdLFswLDEsInlfQSJdLFsxLDMsIiIsMCx7ImxldmVsIjoyLCJzdHlsZSI6eyJoZWFkIjp7Im5hbWUiOiJub25lIn19fV0sWzAsMiwiIiwyLHsibGV2ZWwiOjIsInN0eWxlIjp7ImhlYWQiOnsibmFtZSI6Im5vbmUifX19XSxbMiwzLCJ5X0EiLDJdLFsxLDMsIlUxX3tGQX0iLDAseyJjdXJ2ZSI6LTV9XSxbNyw4LCJVRjFfQSIsMl0sWzcsOCwiVTFfe0ZBfSIsMCx7ImN1cnZlIjotNX1dLFs1LDcsInlfQSJdLFs1LDYsIiIsMSx7ImxldmVsIjoyLCJzdHlsZSI6eyJoZWFkIjp7Im5hbWUiOiJub25lIn19fV0sWzYsOCwieV9BIiwyXSxbMTQsMTUsIlUoXFxleGlzdHMgISkiLDAseyJzaG9ydGVuIjp7InNvdXJjZSI6MjAsInRhcmdldCI6MjB9LCJlZGdlX2FsaWdubWVudCI6eyJzb3VyY2UiOmZhbHNlLCJ0YXJnZXQiOmZhbHNlfX1dLFsxOCwxNiwiXFxtYmJkIiwwLHsic2hvcnRlbiI6eyJzb3VyY2UiOjIwLCJ0YXJnZXQiOjIwfX1dLFsxMCwxMywiXFxpb3RhXnstMX0iLDAseyJzaG9ydGVuIjp7InNvdXJjZSI6MjAsInRhcmdldCI6MjB9LCJlZGdlX2FsaWdubWVudCI6eyJzb3VyY2UiOmZhbHNlLCJ0YXJnZXQiOmZhbHNlfX1dXQ==
\begin{equation}\label{EQP_iota_prime}
\begin{tikzcd}
	A && UFA &&& A && UFA \\
	&&&& {=} \\
	A && UFA &&& A && UFA
	\arrow["{y_A}", from=1-6, to=1-8]
	\arrow[""{name=0, anchor=center, inner sep=0}, Rightarrow, no head, from=1-8, to=3-8]
	\arrow[""{name=0p, anchor=center, inner sep=0}, phantom, from=1-8, to=3-8, start anchor=center, end anchor=center]
	\arrow[Rightarrow, no head, from=1-6, to=3-6]
	\arrow["{y_A}"', from=3-6, to=3-8]
	\arrow[""{name=1, anchor=center, inner sep=0}, "{U1_{FA}}", curve={height=-30pt}, from=1-8, to=3-8]
	\arrow[""{name=1p, anchor=center, inner sep=0}, phantom, from=1-8, to=3-8, start anchor=center, end anchor=center, curve={height=-30pt}]
	\arrow[""{name=2, anchor=center, inner sep=0}, "{UF1_A}"', from=1-3, to=3-3]
	\arrow[""{name=2p, anchor=center, inner sep=0}, phantom, from=1-3, to=3-3, start anchor=center, end anchor=center]
	\arrow[""{name=3, anchor=center, inner sep=0}, "{U1_{FA}}", curve={height=-30pt}, from=1-3, to=3-3]
	\arrow[""{name=3p, anchor=center, inner sep=0}, phantom, from=1-3, to=3-3, start anchor=center, end anchor=center, curve={height=-30pt}]
	\arrow[""{name=4, anchor=center, inner sep=0}, "{y_A}", from=1-1, to=1-3]
	\arrow[Rightarrow, no head, from=1-1, to=3-1]
	\arrow[""{name=5, anchor=center, inner sep=0}, "{y_A}"', from=3-1, to=3-3]
	\arrow["{U(\exists !)}", shorten <=6pt, shorten >=6pt, Rightarrow, from=2p, to=3p]
	\arrow["\mbbd", shorten <=9pt, shorten >=9pt, Rightarrow, from=5, to=4]
	\arrow["{\iota^{-1}}", shorten <=6pt, shorten >=6pt, Rightarrow, from=0p, to=1p]
\end{tikzcd}
\end{equation}
The colax functor axioms for $F$ follow from those of $U$ and can be readily proven using the universal property of $U$-extensions. The above equations also make $y$ into a colax-natural transformation $y: 1 \Rightarrow UF$. Next, define $\epsilon_B: FUB \to B$ and $\Phi_B$ as the $U$-extension of the identity on $UB$ along $y_{UB}$:
% https://q.uiver.app/#q=WzAsMyxbMCwwLCJVQiJdLFsyLDAsIlVGVUIiXSxbMiwyLCJVQiJdLFswLDEsInlfe1VCfSJdLFsxLDIsIlVcXGVwc2lsb25fQiJdLFswLDIsIiIsMix7ImxldmVsIjoyLCJzdHlsZSI6eyJoZWFkIjp7Im5hbWUiOiJub25lIn19fV0sWzUsMywiXFxQaGlfQiIsMix7InNob3J0ZW4iOnsic291cmNlIjoyMCwidGFyZ2V0IjoyMH0sImVkZ2VfYWxpZ25tZW50Ijp7InNvdXJjZSI6ZmFsc2UsInRhcmdldCI6ZmFsc2V9fV1d
\[\begin{tikzcd}
	UB && UFUB \\
	\\
	&& UB
	\arrow[""{name=0, anchor=center, inner sep=0}, "{y_{UB}}", from=1-1, to=1-3]
	\arrow[""{name=0p, anchor=center, inner sep=0}, phantom, from=1-1, to=1-3, start anchor=center, end anchor=center]
	\arrow["{U\epsilon_B}", from=1-3, to=3-3]
	\arrow[""{name=1, anchor=center, inner sep=0}, Rightarrow, no head, from=1-1, to=3-3]
	\arrow[""{name=1p, anchor=center, inner sep=0}, phantom, from=1-1, to=3-3, start anchor=center, end anchor=center]
	\arrow["{\Phi_B}"', shorten <=4pt, shorten >=4pt, Rightarrow, from=1p, to=0p]
\end{tikzcd}\]
The colax naturality square for $\epsilon$ at a 1-cell $h:B \to C$ is the unique 2-cell $\epsilon_h$ making the 2-cells below equal (it is guaranteed to uniquely exist because of the coherence for $U$-extensions):
% https://q.uiver.app/#q=WzAsMTAsWzAsMiwiVUMiXSxbMSwyLCJVRlVDIl0sWzIsNCwiVUMiXSxbMCwwLCJVQiJdLFsyLDAsIlVGVUIiXSxbNCwwLCJVQiJdLFs0LDIsIlVDIl0sWzYsNCwiVUMiXSxbNiwwLCJVRlVCIl0sWzYsMiwiVUIiXSxbMCwxLCJ5X3tVQ30iXSxbMSwyLCJVXFxlcHNpbG9uX0MiLDFdLFswLDIsIiIsMix7ImN1cnZlIjo0LCJsZXZlbCI6Miwic3R5bGUiOnsiaGVhZCI6eyJuYW1lIjoibm9uZSJ9fX1dLFszLDAsIlVoIiwyXSxbMyw0LCJ5X3tVQn0iXSxbNCwxLCJVRlVoIiwxXSxbNSw4LCJ5X3tVQn0iXSxbOCw5LCJVXFxlcHNpbG9uX0IiLDFdLFs1LDYsIlVoIiwyXSxbNiw3LCIiLDIseyJjdXJ2ZSI6NCwibGV2ZWwiOjIsInN0eWxlIjp7ImhlYWQiOnsibmFtZSI6Im5vbmUifX19XSxbOSw3LCJVaCIsMl0sWzUsOSwiIiwxLHsiY3VydmUiOjQsImxldmVsIjoyLCJzdHlsZSI6eyJoZWFkIjp7Im5hbWUiOiJub25lIn19fV0sWzQsMiwiVShoIFxcY2lyYyBcXGVwc2lsb25fQikiLDAseyJjdXJ2ZSI6LTV9XSxbNCwyLCJVKFxcZXBzaWxvbl9DIFxcY2lyYyBGVWgpIiwxLHsibGFiZWxfcG9zaXRpb24iOjcwLCJjdXJ2ZSI6M31dLFs4LDcsIlUoaCBcXGNpcmMgXFxlcHNpbG9uX0IpIiwwLHsib2Zmc2V0IjotNSwiY3VydmUiOi01fV0sWzEwLDE0LCJcXG1iYmQiLDAseyJzaG9ydGVuIjp7InNvdXJjZSI6MzAsInRhcmdldCI6MzB9fV0sWzIxLDgsIlxcUGhpX0IiLDAseyJzaG9ydGVuIjp7InNvdXJjZSI6MjB9fV0sWzEyLDEsIlxcUGhpX0MiLDAseyJzaG9ydGVuIjp7InNvdXJjZSI6MjAsInRhcmdldCI6MjB9LCJlZGdlX2FsaWdubWVudCI6eyJzb3VyY2UiOmZhbHNlfX1dLFsyMywyMiwiVVxcZXBzaWxvbl9oIiwwLHsic2hvcnRlbiI6eyJzb3VyY2UiOjIwLCJ0YXJnZXQiOjIwfSwiZWRnZV9hbGlnbm1lbnQiOnsic291cmNlIjpmYWxzZSwidGFyZ2V0IjpmYWxzZX19XSxbMSwyMywiXFxnYW1tYV57LTF9IiwwLHsic2hvcnRlbiI6eyJ0YXJnZXQiOjIwfSwiZWRnZV9hbGlnbm1lbnQiOnsidGFyZ2V0IjpmYWxzZX19XSxbOSwyNCwiXFxnYW1tYV57LTF9IiwwLHsic2hvcnRlbiI6eyJ0YXJnZXQiOjIwfSwiZWRnZV9hbGlnbm1lbnQiOnsidGFyZ2V0IjpmYWxzZX19XV0=
\begin{equation}\label{EQP_Phi_epsilon}
\begin{minipage}{0.8\textwidth}
\adjustbox{scale=0.9,center}{
\begin{tikzcd}
	UB && UFUB && UB && UFUB \\
	\\
	UC & UFUC &&& UC && UB \\
	\\
	&& UC &&&& UC
	\arrow[""{name=0, anchor=center, inner sep=0}, "{y_{UC}}", from=3-1, to=3-2]
	\arrow["{U\epsilon_C}"{description}, from=3-2, to=5-3]
	\arrow[""{name=1, anchor=center, inner sep=0}, curve={height=24pt}, Rightarrow, no head, from=3-1, to=5-3]
	\arrow[""{name=1p, anchor=center, inner sep=0}, phantom, from=3-1, to=5-3, start anchor=center, end anchor=center, curve={height=24pt}]
	\arrow["Uh"', from=1-1, to=3-1]
	\arrow[""{name=2, anchor=center, inner sep=0}, "{y_{UB}}", from=1-1, to=1-3]
	\arrow["UFUh"{description}, from=1-3, to=3-2]
	\arrow["{y_{UB}}", from=1-5, to=1-7]
	\arrow["{U\epsilon_B}"{description}, from=1-7, to=3-7]
	\arrow["Uh"', from=1-5, to=3-5]
	\arrow[curve={height=24pt}, Rightarrow, no head, from=3-5, to=5-7]
	\arrow["Uh"', from=3-7, to=5-7]
	\arrow[""{name=3, anchor=center, inner sep=0}, curve={height=24pt}, Rightarrow, no head, from=1-5, to=3-7]
	\arrow[""{name=4, anchor=center, inner sep=0}, "{U(h \circ \epsilon_B)}", curve={height=-30pt}, from=1-3, to=5-3]
	\arrow[""{name=4p, anchor=center, inner sep=0}, phantom, from=1-3, to=5-3, start anchor=center, end anchor=center, curve={height=-30pt}]
	\arrow[""{name=5, anchor=center, inner sep=0}, "{U(\epsilon_C \circ FUh)}"{description, pos=0.7}, curve={height=18pt}, from=1-3, to=5-3]
	\arrow[""{name=5p, anchor=center, inner sep=0}, phantom, from=1-3, to=5-3, start anchor=center, end anchor=center, curve={height=18pt}]
	\arrow[""{name=5p, anchor=center, inner sep=0}, phantom, from=1-3, to=5-3, start anchor=center, end anchor=center, curve={height=18pt}]
	\arrow[""{name=6, anchor=center, inner sep=0}, "{U(h \circ \epsilon_B)}", shift left=5, curve={height=-30pt}, from=1-7, to=5-7]
	\arrow[""{name=6p, anchor=center, inner sep=0}, phantom, from=1-7, to=5-7, start anchor=center, end anchor=center, shift left=5, curve={height=-30pt}]
	\arrow["\mbbd", shorten <=14pt, shorten >=14pt, Rightarrow, from=0, to=2]
	\arrow["{\Phi_B}", shorten <=9pt, Rightarrow, from=3, to=1-7]
	\arrow["{\Phi_C}", shorten <=6pt, shorten >=6pt, Rightarrow, from=1p, to=3-2]
	\arrow["{U\epsilon_h}", shorten <=10pt, shorten >=10pt, Rightarrow, from=5p, to=4p]
	\arrow["{\gamma^{-1}}", shorten >=2pt, Rightarrow, from=3-2, to=5p]
	\arrow["{\gamma^{-1}}", shorten >=6pt, Rightarrow, from=3-7, to=6p]
\end{tikzcd}
}
\end{minipage}
\end{equation}
This also makes $\Phi$ into a modification $1_U \to U\epsilon \circ yU$.
Let us now consider the additional assumptions. It is clear that $F$ will be a pseudofunctor. Define $\Psi_A: \epsilon_{FA} \circ Fy_A \Rightarrow 1_{FA}$ as the unique 2-cell making the two 2-cells below equal:
% https://q.uiver.app/#q=WzAsOSxbMCwyLCJVRkEiXSxbMiwyLCJVRlVGQSJdLFszLDQsIlVGQSJdLFswLDAsIkEiXSxbMywwLCJVRkEiXSxbNSwwLCJBIl0sWzUsMiwiVUZBIl0sWzcsNCwiVUZBIl0sWzcsMCwiVUZBIl0sWzAsMSwieV97VUZBfSJdLFsxLDIsIlVcXGVwc2lsb25fe0ZBfSIsMV0sWzAsMiwiIiwyLHsiY3VydmUiOjQsImxldmVsIjoyLCJzdHlsZSI6eyJoZWFkIjp7Im5hbWUiOiJub25lIn19fV0sWzMsMCwieV9BIiwyXSxbMyw0LCJ5X0EiXSxbNCwxLCJVRnlfQSIsMV0sWzQsMiwiVTFfe0ZBfSIsMCx7ImN1cnZlIjotNX1dLFs4LDcsIlUxX3tGQX0iLDAseyJjdXJ2ZSI6LTV9XSxbOCw3LCIiLDIseyJsZXZlbCI6Miwic3R5bGUiOnsiaGVhZCI6eyJuYW1lIjoibm9uZSJ9fX1dLFs1LDYsInlfQSIsMl0sWzYsNywiIiwwLHsiY3VydmUiOjQsImxldmVsIjoyLCJzdHlsZSI6eyJoZWFkIjp7Im5hbWUiOiJub25lIn19fV0sWzQsMiwiVShcXGVwc2lsb25fe0ZBfSBcXGNpcmMgRnlfQSkiLDEseyJsYWJlbF9wb3NpdGlvbiI6NjB9XSxbNSw4LCJ5X0EiXSxbOSwxMywiXFxtYmJkIiwwLHsic2hvcnRlbiI6eyJzb3VyY2UiOjMwLCJ0YXJnZXQiOjMwfSwiZWRnZV9hbGlnbm1lbnQiOnsic291cmNlIjpmYWxzZSwidGFyZ2V0IjpmYWxzZX19XSxbMTcsMTYsIlxcaW90YV57LTF9IiwwLHsibGFiZWxfcG9zaXRpb24iOjQwLCJzaG9ydGVuIjp7InNvdXJjZSI6MjAsInRhcmdldCI6MjB9LCJlZGdlX2FsaWdubWVudCI6eyJzb3VyY2UiOmZhbHNlLCJ0YXJnZXQiOmZhbHNlfX1dLFsyMCwxNSwiVShcXGV4aXN0cyAhKSIsMCx7InNob3J0ZW4iOnsic291cmNlIjoyMCwidGFyZ2V0IjoyMH0sImVkZ2VfYWxpZ25tZW50Ijp7InNvdXJjZSI6ZmFsc2UsInRhcmdldCI6ZmFsc2V9fV0sWzEsMjAsIlxcZ2FtbWFeey0xfSIsMCx7InNob3J0ZW4iOnsidGFyZ2V0IjoyMH0sImVkZ2VfYWxpZ25tZW50Ijp7InRhcmdldCI6ZmFsc2V9fV0sWzExLDEsIlxcUGhpX3tGQX0iLDAseyJzaG9ydGVuIjp7InNvdXJjZSI6MjAsInRhcmdldCI6MjB9LCJlZGdlX2FsaWdubWVudCI6eyJzb3VyY2UiOmZhbHNlfX1dXQ==
\begin{equation}\label{EQP_Psi}
\begin{minipage}{0.8\textwidth}
\adjustbox{scale=0.9,center}{
\begin{tikzcd}
	A &&& UFA && A && UFA \\
	\\
	UFA && UFUFA &&& UFA \\
	\\
	&&& UFA &&&& UFA
	\arrow[""{name=0, anchor=center, inner sep=0}, "{y_{UFA}}", from=3-1, to=3-3]
	\arrow[""{name=0p, anchor=center, inner sep=0}, phantom, from=3-1, to=3-3, start anchor=center, end anchor=center]
	\arrow["{U\epsilon_{FA}}"{description}, from=3-3, to=5-4]
	\arrow[""{name=1, anchor=center, inner sep=0}, curve={height=24pt}, Rightarrow, no head, from=3-1, to=5-4]
	\arrow[""{name=1p, anchor=center, inner sep=0}, phantom, from=3-1, to=5-4, start anchor=center, end anchor=center, curve={height=24pt}]
	\arrow["{y_A}"', from=1-1, to=3-1]
	\arrow[""{name=2, anchor=center, inner sep=0}, "{y_A}", from=1-1, to=1-4]
	\arrow[""{name=2p, anchor=center, inner sep=0}, phantom, from=1-1, to=1-4, start anchor=center, end anchor=center]
	\arrow["{UFy_A}"{description}, from=1-4, to=3-3]
	\arrow[""{name=3, anchor=center, inner sep=0}, "{U1_{FA}}", curve={height=-30pt}, from=1-4, to=5-4]
	\arrow[""{name=3p, anchor=center, inner sep=0}, phantom, from=1-4, to=5-4, start anchor=center, end anchor=center, curve={height=-30pt}]
	\arrow[""{name=4, anchor=center, inner sep=0}, "{U1_{FA}}", curve={height=-30pt}, from=1-8, to=5-8]
	\arrow[""{name=4p, anchor=center, inner sep=0}, phantom, from=1-8, to=5-8, start anchor=center, end anchor=center, curve={height=-30pt}]
	\arrow[""{name=5, anchor=center, inner sep=0}, Rightarrow, no head, from=1-8, to=5-8]
	\arrow[""{name=5p, anchor=center, inner sep=0}, phantom, from=1-8, to=5-8, start anchor=center, end anchor=center]
	\arrow["{y_A}"', from=1-6, to=3-6]
	\arrow[curve={height=24pt}, Rightarrow, no head, from=3-6, to=5-8]
	\arrow[""{name=6, anchor=center, inner sep=0}, "{U(\epsilon_{FA} \circ Fy_A)}"{description, pos=0.6}, from=1-4, to=5-4]
	\arrow[""{name=6p, anchor=center, inner sep=0}, phantom, from=1-4, to=5-4, start anchor=center, end anchor=center]
	\arrow[""{name=6p, anchor=center, inner sep=0}, phantom, from=1-4, to=5-4, start anchor=center, end anchor=center]
	\arrow["{y_A}", from=1-6, to=1-8]
	\arrow["\mbbd", shorten <=14pt, shorten >=14pt, Rightarrow, from=0p, to=2p]
	\arrow["{\iota^{-1}}"{pos=0.4}, shorten <=6pt, shorten >=6pt, Rightarrow, from=5p, to=4p]
	\arrow["{U(\exists !)}", shorten <=6pt, shorten >=6pt, Rightarrow, from=6p, to=3p]
	\arrow["{\gamma^{-1}}", shorten >=5pt, Rightarrow, from=3-3, to=6p]
	\arrow["{\Phi_{FA}}", shorten <=7pt, shorten >=7pt, Rightarrow, from=1p, to=3-3]
\end{tikzcd}
}
\end{minipage}
\end{equation}
By the assumption, $\Psi_A$ is invertible. This equality also proves the first swallowtail identity. What remains to prove is the following:
\begin{itemize}
\item $\epsilon$ is colax-natural,
\item $\Phi$ is a modification,
\item the second swallowtail identity.
\end{itemize}
These are all straightforward computations and we will prove them in the Appendix as Lemma \ref{THM_lemma_epsilon_phi_second_swallowtail}.
\end{proof}

%240
%225 examples of lax adjunctions

\begin{theo}\label{THM_Bunge_opak}
Let $(\Psi, \Phi): (\epsilon, y): F \ddashv U: \cc \to \cd$ be a colax adjunction between pseudofunctors in which $\Psi$ is invertible. Then:
\begin{itemize}
\item the components of the unit $y_A : A \to UFA$ are coherently closed for $U$-extensions,
\item the unit and composition axioms \eqref{EQ_U_extensions_unit_comp} for $U$-extensions hold.
\end{itemize}
\end{theo}
\begin{proof}
Notice first that we have the following adjunction:
% https://q.uiver.app/#q=WzAsMixbMCwwLCJcXGNjKEZBLCBCKSJdLFsyLDAsIlxcY2QoQSxVQikiXSxbMCwxLCIoXFxldGFfQSleKiBcXGNpcmMgVToiLDAseyJjdXJ2ZSI6LTR9XSxbMSwwLCIoXFxlcHNpbG9uX0IpXyogXFxjaXJjIEYiLDAseyJjdXJ2ZSI6LTR9XSxbMywyLCIiLDAseyJsZXZlbCI6MSwic3R5bGUiOnsibmFtZSI6ImFkanVuY3Rpb24ifX1dXQ==
\[\begin{tikzcd}
	{\cc(FA, B)} && {\cd(A,UB)}
	\arrow[""{name=0, anchor=center, inner sep=0}, "{(\eta_A)^* \circ U:}", curve={height=-24pt}, from=1-1, to=1-3]
	\arrow[""{name=1, anchor=center, inner sep=0}, "{(\epsilon_B)_* \circ F}", curve={height=-24pt}, from=1-3, to=1-1]
	\arrow["\dashv"{anchor=center, rotate=90}, draw=none, from=1, to=0]
\end{tikzcd}\]

The counit and unit 2-cells evaluated at $h: FA \to B$ and $g: A \to UB$ are given as follows:
\[\begin{tikzcd}
	FA &&& A && UFA \\
	FUFA && FA & UB && UFUB \\
	FUB && B &&& UB
	\arrow[""{name=0, anchor=center, inner sep=0}, "{U(\epsilon_B Fg)}", shift left=5, curve={height=-30pt}, from=1-6, to=3-6]
	\arrow[""{name=0p, anchor=center, inner sep=0}, phantom, from=1-6, to=3-6, start anchor=center, end anchor=center, shift left=5, curve={height=-30pt}]
	\arrow[""{name=1, anchor=center, inner sep=0}, "{F(Uhy_A)}"', shift right=5, curve={height=30pt}, from=1-1, to=3-1]
	\arrow[""{name=1p, anchor=center, inner sep=0}, phantom, from=1-1, to=3-1, start anchor=center, end anchor=center, shift right=5, curve={height=30pt}]
	\arrow["{Fy_A}", from=1-1, to=2-1]
	\arrow["FUh", from=2-1, to=3-1]
	\arrow[""{name=2, anchor=center, inner sep=0}, "{\epsilon_B}"', from=3-1, to=3-3]
	\arrow[""{name=2p, anchor=center, inner sep=0}, phantom, from=3-1, to=3-3, start anchor=center, end anchor=center]
	\arrow[""{name=3, anchor=center, inner sep=0}, "{\epsilon_{FA}}"{description, pos=0.7}, from=2-1, to=2-3]
	\arrow[""{name=3p, anchor=center, inner sep=0}, phantom, from=2-1, to=2-3, start anchor=center, end anchor=center]
	\arrow[""{name=3p, anchor=center, inner sep=0}, phantom, from=2-1, to=2-3, start anchor=center, end anchor=center]
	\arrow["h", from=2-3, to=3-3]
	\arrow[""{name=4, anchor=center, inner sep=0}, Rightarrow, no head, from=1-1, to=2-3]
	\arrow[""{name=4p, anchor=center, inner sep=0}, phantom, from=1-1, to=2-3, start anchor=center, end anchor=center]
	\arrow["g"', from=1-4, to=2-4]
	\arrow[""{name=5, anchor=center, inner sep=0}, "{y_A}", from=1-4, to=1-6]
	\arrow[""{name=5p, anchor=center, inner sep=0}, phantom, from=1-4, to=1-6, start anchor=center, end anchor=center]
	\arrow["UFg"', from=1-6, to=2-6]
	\arrow["{U\epsilon_B}"', from=2-6, to=3-6]
	\arrow[""{name=6, anchor=center, inner sep=0}, "{y_{UB}}"{description, pos=0.3}, from=2-4, to=2-6]
	\arrow[""{name=6p, anchor=center, inner sep=0}, phantom, from=2-4, to=2-6, start anchor=center, end anchor=center]
	\arrow[""{name=6p, anchor=center, inner sep=0}, phantom, from=2-4, to=2-6, start anchor=center, end anchor=center]
	\arrow[""{name=7, anchor=center, inner sep=0}, Rightarrow, no head, from=2-4, to=3-6]
	\arrow[""{name=7p, anchor=center, inner sep=0}, phantom, from=2-4, to=3-6, start anchor=center, end anchor=center]
	\arrow[""{name=8, anchor=center, inner sep=0}, curve={height=-6pt}, draw=none, from=1-6, to=3-6]
	\arrow[""{name=8p, anchor=center, inner sep=0}, phantom, from=1-6, to=3-6, start anchor=center, end anchor=center, curve={height=-6pt}]
	\arrow[""{name=9, anchor=center, inner sep=0}, curve={height=6pt}, draw=none, from=1-1, to=3-1]
	\arrow[""{name=9p, anchor=center, inner sep=0}, phantom, from=1-1, to=3-1, start anchor=center, end anchor=center, curve={height=6pt}]
	\arrow["{\Psi_A}", shorten <=2pt, shorten >=2pt, Rightarrow, from=3p, to=4p]
	\arrow["{\epsilon_h}"', shorten <=4pt, shorten >=4pt, Rightarrow, from=2p, to=3p]
	\arrow["{y_g}", shorten <=4pt, shorten >=4pt, Rightarrow, from=6p, to=5p]
	\arrow["{\Phi_B}"', shorten <=2pt, shorten >=2pt, Rightarrow, from=7p, to=6p]
	\arrow["{\gamma^{-1}}"{pos=0.7}, shorten <=17pt, shorten >=3pt, Rightarrow, from=8p, to=0p]
	\arrow["\gamma"{pos=0.2}, shorten >=17pt, Rightarrow, from=1p, to=9p]
\end{tikzcd}\]
The triangle identities essentially follow from the swallowtail identities of the colax adjunction and we omit the proof for them. Denote by $\mbbd_g$ the unit of this adjunction evaluated at $g: A \to UB$ and denote $g^\mbbd := \epsilon_B Fg$. By definition, the pair $(g^\mbbd, \mbbd_g)$ is the left $U$-extension of $g$ along $y_A$. Next, notice that for $f: A \to B$, the invertible 2-cell:
\[
\beth_f := \Psi_B Ff \circ \epsilon_{FB} \gamma^{-1}_{f,y_B} : \epsilon_{FB} F(y_B f) \Rightarrow Ff: FA \to FB,
\]
satisfies the following equality (this again follows from a swallowtail identity): %holds and thus provides us with an isomorphism of $U$-pairs $((y_Bf)^\mbbd, \mbbd_f)$, $(UFf, y_f)$:
% https://q.uiver.app/#q=WzAsOSxbMCwwLCJBIl0sWzIsMCwiVUZBIl0sWzAsMiwiQiJdLFsyLDIsIlVGQiJdLFs0LDEsIj0iXSxbNSwwLCJBIl0sWzUsMiwiQiJdLFs3LDAsIlVGQSJdLFs3LDIsIlVGQiJdLFsxLDMsIlUoeV9CIGYpXlxcbWJiZCIsMix7ImxhYmVsX3Bvc2l0aW9uIjo3MH1dLFswLDIsImYiLDJdLFswLDEsInlfQSJdLFs1LDYsImYiLDJdLFs1LDcsInlfQSJdLFs3LDgsIlVGZiJdLFs2LDgsInlfQiIsMl0sWzEsMywiVUZmIiwwLHsiY3VydmUiOi01fV0sWzIsMywieV9CIiwyXSxbOSwxNiwiVVxcYmV0aF9mIiwwLHsic2hvcnRlbiI6eyJzb3VyY2UiOjIwLCJ0YXJnZXQiOjIwfSwiZWRnZV9hbGlnbm1lbnQiOnsic291cmNlIjpmYWxzZSwidGFyZ2V0IjpmYWxzZX19XSxbMTcsMTEsIlxcbWJiZF9mIiwwLHsic2hvcnRlbiI6eyJzb3VyY2UiOjIwLCJ0YXJnZXQiOjIwfX1dLFsxNSwxMywieV9mIiwyLHsic2hvcnRlbiI6eyJzb3VyY2UiOjIwLCJ0YXJnZXQiOjIwfX1dXQ==
\[\begin{tikzcd}
	A && UFA &&& A && UFA \\
	&&&& {=} \\
	B && UFB &&& B && UFB
	\arrow[""{name=0, anchor=center, inner sep=0}, "{U(y_B f)^\mbbd}"'{pos=0.7}, from=1-3, to=3-3]
	\arrow[""{name=0p, anchor=center, inner sep=0}, phantom, from=1-3, to=3-3, start anchor=center, end anchor=center]
	\arrow["f"', from=1-1, to=3-1]
	\arrow[""{name=1, anchor=center, inner sep=0}, "{y_A}", from=1-1, to=1-3]
	\arrow["f"', from=1-6, to=3-6]
	\arrow[""{name=2, anchor=center, inner sep=0}, "{y_A}", from=1-6, to=1-8]
	\arrow["UFf", from=1-8, to=3-8]
	\arrow[""{name=3, anchor=center, inner sep=0}, "{y_B}"', from=3-6, to=3-8]
	\arrow[""{name=4, anchor=center, inner sep=0}, "UFf", curve={height=-30pt}, from=1-3, to=3-3]
	\arrow[""{name=4p, anchor=center, inner sep=0}, phantom, from=1-3, to=3-3, start anchor=center, end anchor=center, curve={height=-30pt}]
	\arrow[""{name=5, anchor=center, inner sep=0}, "{y_B}"', from=3-1, to=3-3]
	\arrow["{U\beth_f}", shorten <=6pt, shorten >=6pt, Rightarrow, from=0p, to=4p]
	\arrow["{\mbbd_f}", shorten <=9pt, shorten >=9pt, Rightarrow, from=5, to=1]
	\arrow["{y_f}"', shorten <=9pt, shorten >=9pt, Rightarrow, from=3, to=2]
\end{tikzcd}\]
This proves that $(UFf, y_f)$ is also a $U$-extension of $y_B f$ along $y_A$. Next, notice that for an object $B$, the following invertible 2-cell:
\[
\Xi_B := \epsilon_{1_B} \circ \epsilon_B F\iota^{-1}: \epsilon_B F1_{UB} \Rightarrow \epsilon_B,
\]
satisfies this equality:% of $U$-pairs $(1_{UB}^\mbbd, \mbbd_{1_{UB}})$ and $(U\epsilon_B, \Phi_B)$:
% https://q.uiver.app/#q=WzAsNyxbMCwwLCJVQiJdLFsyLDAsIlVGVUIiXSxbMiwyLCJVQiJdLFs0LDEsIj0iXSxbNSwwLCJVQiJdLFs3LDAsIlVGVUIiXSxbNywyLCJVQiJdLFswLDIsIiIsMCx7ImxldmVsIjoyLCJzdHlsZSI6eyJoZWFkIjp7Im5hbWUiOiJub25lIn19fV0sWzAsMSwieV97VUJ9Il0sWzEsMiwiVTFeXFxtYmJkX3tVQn0iLDAseyJsYWJlbF9wb3NpdGlvbiI6ODB9XSxbMSwyLCJVXFxlcHNpbG9uX3tCfSIsMCx7Im9mZnNldCI6LTQsImN1cnZlIjotNX1dLFs0LDUsInlfe1VCfSJdLFs1LDYsIlVcXGVwc2lsb25fQiJdLFs0LDYsIiIsMSx7ImxldmVsIjoyLCJzdHlsZSI6eyJoZWFkIjp7Im5hbWUiOiJub25lIn19fV0sWzcsOSwiXFxtYmJkX3sxX3tVQn19IiwwLHsic2hvcnRlbiI6eyJzb3VyY2UiOjIwLCJ0YXJnZXQiOjIwfSwiZWRnZV9hbGlnbm1lbnQiOnsic291cmNlIjpmYWxzZSwidGFyZ2V0IjpmYWxzZX19XSxbOSwxMCwiVVxcWGlfQiIsMCx7InNob3J0ZW4iOnsic291cmNlIjoyMCwidGFyZ2V0IjoyMH0sImVkZ2VfYWxpZ25tZW50Ijp7InNvdXJjZSI6ZmFsc2UsInRhcmdldCI6ZmFsc2V9fV0sWzEzLDEyLCJcXFBoaV97Qn0iLDAseyJzaG9ydGVuIjp7InNvdXJjZSI6MjAsInRhcmdldCI6MjB9LCJlZGdlX2FsaWdubWVudCI6eyJzb3VyY2UiOmZhbHNlLCJ0YXJnZXQiOmZhbHNlfX1dXQ==
\[\begin{tikzcd}
	UB && UFUB &&& UB && UFUB \\
	&&&& {=} \\
	&& UB &&&&& UB
	\arrow[""{name=0, anchor=center, inner sep=0}, Rightarrow, no head, from=1-1, to=3-3]
	\arrow[""{name=0p, anchor=center, inner sep=0}, phantom, from=1-1, to=3-3, start anchor=center, end anchor=center]
	\arrow["{y_{UB}}", from=1-1, to=1-3]
	\arrow[""{name=1, anchor=center, inner sep=0}, "{U1^\mbbd_{UB}}"{pos=0.8}, from=1-3, to=3-3]
	\arrow[""{name=1p, anchor=center, inner sep=0}, phantom, from=1-3, to=3-3, start anchor=center, end anchor=center]
	\arrow[""{name=1p, anchor=center, inner sep=0}, phantom, from=1-3, to=3-3, start anchor=center, end anchor=center]
	\arrow[""{name=2, anchor=center, inner sep=0}, "{U\epsilon_{B}}", shift left=4, curve={height=-30pt}, from=1-3, to=3-3]
	\arrow[""{name=2p, anchor=center, inner sep=0}, phantom, from=1-3, to=3-3, start anchor=center, end anchor=center, shift left=4, curve={height=-30pt}]
	\arrow["{y_{UB}}", from=1-6, to=1-8]
	\arrow[""{name=3, anchor=center, inner sep=0}, "{U\epsilon_B}", from=1-8, to=3-8]
	\arrow[""{name=3p, anchor=center, inner sep=0}, phantom, from=1-8, to=3-8, start anchor=center, end anchor=center]
	\arrow[""{name=4, anchor=center, inner sep=0}, Rightarrow, no head, from=1-6, to=3-8]
	\arrow[""{name=4p, anchor=center, inner sep=0}, phantom, from=1-6, to=3-8, start anchor=center, end anchor=center]
	\arrow["{\mbbd_{1_{UB}}}", shorten <=7pt, shorten >=7pt, Rightarrow, from=0p, to=1p]
	\arrow["{U\Xi_B}", shorten <=8pt, shorten >=8pt, Rightarrow, from=1p, to=2p]
	\arrow["{\Phi_{B}}", shorten <=7pt, shorten >=7pt, Rightarrow, from=4p, to=3p]
\end{tikzcd}\]

This proves that $(U\epsilon_B, \Phi_B)$ is also a $U$-extension of $1_{UB}$ along $y_{UB}$. Using these two isomorphisms of $U$-extensions, it is clear that the composite 2-cell \eqref{EQ_DEFI_coherently_closed_for_Uext} in Definition \ref{DEFI_coherently_closed_for_Uext} is a $U$-extension if and only if the pair $(f^\mbbd, \mbbd_f)$ is a $U$-extension - which it is, as we have proven. We thus have that the collection $y_A: A \to UFA$ is coherently closed for $U$-extensions.

Let us now prove the composition and unit axioms \eqref{EQ_U_extensions_unit_comp}. The proof that the pair\\ $(1_{FA}, \iota^{-1}y_A)$ is a $U$-extension follows immediately from the fact that it is isomorphic to the $U$-extension $(y_A^\mbbd, \mbbd_{y_A})$ via the modification $\Psi_A$ (this is the first swallowtail identity):
% https://q.uiver.app/#q=WzAsNyxbMCwwLCJBIl0sWzIsMCwiVUZBIl0sWzIsMiwiVUZBIl0sWzQsMSwiPSJdLFs1LDAsIkEiXSxbNywwLCJVRkEiXSxbNywyLCJVRkEiXSxbMCwxLCJ5X0EiXSxbMCwyLCJ5X0EiLDJdLFsxLDIsIlUxX3tGQX0iLDAseyJjdXJ2ZSI6LTV9XSxbMSwyLCIiLDEseyJsZXZlbCI6Miwic3R5bGUiOnsiaGVhZCI6eyJuYW1lIjoibm9uZSJ9fX1dLFs0LDUsInlfQSJdLFs1LDYsIlV5X0FeXFxtYmJkIiwwLHsibGFiZWxfcG9zaXRpb24iOjgwfV0sWzQsNiwieV9BIiwyXSxbNSw2LCJVMV97RkF9IiwwLHsib2Zmc2V0IjotNSwiY3VydmUiOi01fV0sWzEwLDksIlxcaW90YV57LTF9IiwwLHsic2hvcnRlbiI6eyJzb3VyY2UiOjIwLCJ0YXJnZXQiOjIwfSwiZWRnZV9hbGlnbm1lbnQiOnsic291cmNlIjpmYWxzZSwidGFyZ2V0IjpmYWxzZX19XSxbMTMsMTIsIlxcbWJiZF97eV9BfSIsMCx7InNob3J0ZW4iOnsic291cmNlIjoyMCwidGFyZ2V0IjoyMH0sImVkZ2VfYWxpZ25tZW50Ijp7InNvdXJjZSI6ZmFsc2UsInRhcmdldCI6ZmFsc2V9fV0sWzEyLDE0LCJVXFxQc2lfQSIsMCx7InNob3J0ZW4iOnsic291cmNlIjoyMCwidGFyZ2V0IjoyMH0sImVkZ2VfYWxpZ25tZW50Ijp7InNvdXJjZSI6ZmFsc2UsInRhcmdldCI6ZmFsc2V9fV1d
\[\begin{tikzcd}
	A && UFA &&& A && UFA \\
	&&&& {=} \\
	&& UFA &&&&& UFA
	\arrow["{y_A}", from=1-1, to=1-3]
	\arrow["{y_A}"', from=1-1, to=3-3]
	\arrow[""{name=0, anchor=center, inner sep=0}, "{U1_{FA}}", curve={height=-30pt}, from=1-3, to=3-3]
	\arrow[""{name=0p, anchor=center, inner sep=0}, phantom, from=1-3, to=3-3, start anchor=center, end anchor=center, curve={height=-30pt}]
	\arrow[""{name=1, anchor=center, inner sep=0}, Rightarrow, no head, from=1-3, to=3-3]
	\arrow[""{name=1p, anchor=center, inner sep=0}, phantom, from=1-3, to=3-3, start anchor=center, end anchor=center]
	\arrow["{y_A}", from=1-6, to=1-8]
	\arrow[""{name=2, anchor=center, inner sep=0}, "{Uy_A^\mbbd}"{pos=0.8}, from=1-8, to=3-8]
	\arrow[""{name=2p, anchor=center, inner sep=0}, phantom, from=1-8, to=3-8, start anchor=center, end anchor=center]
	\arrow[""{name=2p, anchor=center, inner sep=0}, phantom, from=1-8, to=3-8, start anchor=center, end anchor=center]
	\arrow[""{name=3, anchor=center, inner sep=0}, "{y_A}"', from=1-6, to=3-8]
	\arrow[""{name=3p, anchor=center, inner sep=0}, phantom, from=1-6, to=3-8, start anchor=center, end anchor=center]
	\arrow[""{name=4, anchor=center, inner sep=0}, "{U1_{FA}}", shift left=5, curve={height=-30pt}, from=1-8, to=3-8]
	\arrow[""{name=4p, anchor=center, inner sep=0}, phantom, from=1-8, to=3-8, start anchor=center, end anchor=center, shift left=5, curve={height=-30pt}]
	\arrow["{\iota^{-1}}", shorten <=6pt, shorten >=6pt, Rightarrow, from=1p, to=0p]
	\arrow["{\mbbd_{y_A}}", shorten <=7pt, shorten >=7pt, Rightarrow, from=3p, to=2p]
	\arrow["{U\Psi_A}", shorten <=8pt, shorten >=8pt, Rightarrow, from=2p, to=4p]
\end{tikzcd}\]

Again by using the isomorphism above, the question whether the 2-cell below right is a $U$-extension is equivalent to asking whether the 2-cell below left is a $U$-extension:

% https://q.uiver.app/#q=WzAsMTIsWzAsMCwiQSJdLFsyLDAsIlVGQSJdLFswLDEsIkIiXSxbMiwxLCJVRkIiXSxbMCwyLCJDIl0sWzIsMiwiVUZDIl0sWzUsMCwiQSJdLFs3LDAsIlVGQSJdLFs1LDEsIkIiXSxbNywxLCJVRkIiXSxbNSwyLCJDIl0sWzcsMiwiVUZDIl0sWzIsMywieV9CIiwxXSxbMCwyLCJmIiwyXSxbMSwzLCJVRmYiXSxbMiw0LCJnIiwyXSxbMyw1LCJVRmciXSxbNCw1LCJ5X0MiLDJdLFsxLDUsIlVGKGdmKSIsMCx7Im9mZnNldCI6LTUsImN1cnZlIjotNX1dLFswLDEsInlfQSJdLFs2LDcsInlfQSJdLFs3LDksIlUoeV9CZileXFxtYmJkIiwyXSxbOSwxMSwiVSh5X0NnKV5cXG1iYmQiLDJdLFs4LDEwLCJnIiwyXSxbMTAsMTEsInlfQyIsMV0sWzYsOCwiZiIsMl0sWzgsOSwieV9CIiwxXSxbNywxMSwiVSgoeV9DZyleXFxtYmJkICh5X0JmKV5cXG1iYmQpIiwwLHsib2Zmc2V0IjotNSwiY3VydmUiOi01fV0sWzE3LDEyLCJ5X2ciLDIseyJzaG9ydGVuIjp7InNvdXJjZSI6MjAsInRhcmdldCI6MjB9fV0sWzMsMTgsIlxcZ2FtbWEiLDAseyJsYWJlbF9wb3NpdGlvbiI6MzAsImVkZ2VfYWxpZ25tZW50Ijp7InRhcmdldCI6ZmFsc2V9fV0sWzEyLDE5LCJ5X2YiLDIseyJzaG9ydGVuIjp7InNvdXJjZSI6MjAsInRhcmdldCI6MjB9fV0sWzI2LDIwLCJcXG1iYmRfZiIsMCx7InNob3J0ZW4iOnsic291cmNlIjoyMCwidGFyZ2V0IjoyMH19XSxbMjQsMjYsIlxcbWJiZF9nIiwwLHsic2hvcnRlbiI6eyJzb3VyY2UiOjIwLCJ0YXJnZXQiOjIwfX1dLFs5LDI3LCJcXGdhbW1hIiwwLHsiZWRnZV9hbGlnbm1lbnQiOnsidGFyZ2V0IjpmYWxzZX19XV0=
\[\begin{tikzcd}
	A && UFA &&& A && UFA \\
	B && UFB &&& B && UFB \\
	C && UFC &&& C && UFC
	\arrow[""{name=0, anchor=center, inner sep=0}, "{y_B}"{description}, from=2-1, to=2-3]
	\arrow["f"', from=1-1, to=2-1]
	\arrow["UFf", from=1-3, to=2-3]
	\arrow["g"', from=2-1, to=3-1]
	\arrow["UFg", from=2-3, to=3-3]
	\arrow[""{name=1, anchor=center, inner sep=0}, "{y_C}"', from=3-1, to=3-3]
	\arrow[""{name=2, anchor=center, inner sep=0}, "{UF(gf)}", shift left=5, curve={height=-30pt}, from=1-3, to=3-3]
	\arrow[""{name=2p, anchor=center, inner sep=0}, phantom, from=1-3, to=3-3, start anchor=center, end anchor=center, shift left=5, curve={height=-30pt}]
	\arrow[""{name=3, anchor=center, inner sep=0}, "{y_A}", from=1-1, to=1-3]
	\arrow[""{name=4, anchor=center, inner sep=0}, "{y_A}", from=1-6, to=1-8]
	\arrow["{U(y_Bf)^\mbbd}"', from=1-8, to=2-8]
	\arrow["{U(y_Cg)^\mbbd}"', from=2-8, to=3-8]
	\arrow["g"', from=2-6, to=3-6]
	\arrow[""{name=5, anchor=center, inner sep=0}, "{y_C}"{description}, from=3-6, to=3-8]
	\arrow["f"', from=1-6, to=2-6]
	\arrow[""{name=6, anchor=center, inner sep=0}, "{y_B}"{description}, from=2-6, to=2-8]
	\arrow[""{name=7, anchor=center, inner sep=0}, "{U((y_Cg)^\mbbd (y_Bf)^\mbbd)}", shift left=5, curve={height=-30pt}, from=1-8, to=3-8]
	\arrow[""{name=7p, anchor=center, inner sep=0}, phantom, from=1-8, to=3-8, start anchor=center, end anchor=center, shift left=5, curve={height=-30pt}]
	\arrow["{y_g}"', shorten <=4pt, shorten >=4pt, Rightarrow, from=1, to=0]
	\arrow["\gamma"{pos=0.3}, Rightarrow, from=2-3, to=2p]
	\arrow["{y_f}"', shorten <=4pt, shorten >=4pt, Rightarrow, from=0, to=3]
	\arrow["{\mbbd_f}", shorten <=4pt, shorten >=4pt, Rightarrow, from=6, to=4]
	\arrow["{\mbbd_g}", shorten <=4pt, shorten >=4pt, Rightarrow, from=5, to=6]
	\arrow["\gamma", Rightarrow, from=2-8, to=7p]
\end{tikzcd}\]
But this 2-cell equals $y_{gf}$ and is thus a $U$-extension by what we have proven above.

\end{proof}

%355

\begin{theo}\label{THM_korespondence_lax_adj_u_coh_ext}
Fix a pseudofunctor $U: \cd \to \cc$ between 2-categories. The following are equivalent for a collection of 1-cells $\{ y_A: A \to UFA \}$ with $A \in \cc$:
\begin{itemize}
\item the collection $y_A$ is coherently closed for $U$-extensions and satisfies composition and unit axioms \eqref{EQ_U_extensions_unit_comp},
\item there is a colax adjunction $(\Psi, \Phi): (\epsilon, \eta): F \ddashv U$ for which $\Psi$ is invertible, $F$ is a pseudofunctor and the 1-cell component of the unit at each $A \in \cc$ equals  $y_A$.
\end{itemize}
\end{theo}

\begin{pozn}
In the above theorem, we do not have a one-to-one correspondence; instead, there is a suitable ``equivalence” between these two concepts. Starting with coherent $U$-extensions $(f^\mbbd, \mbbd_f)$ of $f$ along $y_A$, producing a colax adjunction and then going back to $U$-extensions gives the $U$-extension $(\epsilon_B Ff, \gamma^{-1}y_A \circ U\epsilon_B y_f \circ \Phi_B f)$, which in general will not be equal to $(f^\mbbd, \mbbd_f)$ (but will be canonically isomorphic to it). Similarly, starting with left colax adjoint $F$, going to $U$-extensions and back only gives a pseudofunctor isomorphic to $F$.
\end{pozn}

\noindent In our applications to two-dimensional monad theory, we will encounter this very special case of $U$-extensions:

\begin{defi}\label{DEFI_strictly_closed_for_Uexts}
Let $U: \cc \to \cd$ be a 2-functor. We will say that a collection of 1-cells $y_A: A \to UFA$ is \textit{strictly closed for $U$-extensions} if:
\begin{itemize}
\item for every $f: A \to UB$ there is a $U$-extension $(f^\mbbd, 1_f)$ along $y_A$ with the 2-cell component being the identity,
\item $y_A^\mbbd = 1_{FA}$, 
\item for $f: X \to Y$, $g: Y \to Z$ we have $Ff \circ Fg = F(fg)$, where we denote $Ff := (y_Y \circ f)^\mbbd$,
\item for $f: A \to UB$ we have $\epsilon_Y \circ Ff = f^\mbbd$, where we denote $\epsilon_Y := (1_Y)^\mbbd$.
\end{itemize}
\end{defi}

\begin{pozn}\label{POZN_strict_BungesTHM}
It is clear from the proof of Theorem \ref{THM_Bunge} that a collection strictly closed for $U$-extension gives rise to a colax adjunction $(\epsilon, y): F \ddashv U$ for which:
\begin{itemize}
\item y is a 2-natural transformation,
\item $F$ is a 2-functor,
\item the modifications $\Phi, \Psi$ are the identities.
\end{itemize}
(This will in general not be a 2-adjunction because $\epsilon$ will  only be colax natural.)
\end{pozn}

\section{On the Kleisli 2-category for a left Kan pseudomonad}\label{SEKCE_Kleisli2cat}

This section is devoted to studying the Kleisli 2-category for a general left Kan pseudomonad $(D,y)$ on a 2-category $\ck$.

In \ref{SUBS_3_charakterizace_algeber} we prove a result characterizing the pseudo-$D$-algebra structure on an object in terms of the existence of certain adjoints (Theorem \ref{THM_characterization_D-algs_D-corefl-incl}).

In \ref{SUBS_3_big_lax_adj_thm} we use this result and Theorem \ref{THM_Bunge} to prove that any left biadjoint $\ck \to \cl$ that factorizes through the Kleisli 2-category gives rise to a lax left adjoint $\ck_D \to \cl$. We list several applications, one of which is the assertion that there is a canonical colax adjunction between EM and Kleisli 2-categories for left Kan pseudomonads.

Another application is given in \ref{SUBS_2_coreflector_limits} where we define \textit{coreflector-limits}, the aforementioned lax analogue of bilimits, and list elementary examples. The main result here is Theorem \ref{THM_Kleisli2cat_is_corefl_complete} which asserts that whenever the base 2-category $\ck$ admits $J$-indexed bilimits, the Kleisli 2-category for a left Kan pseudomonad on $\ck$ will admit them as coreflector-limits.

First, let us recall the following terminology:

\begin{defi}\label{DEFI_lali_rali_rari_lari}
Let the following be an adjunction in a 2-category $\ck$:
% https://q.uiver.app/#q=WzAsMyxbMSwwLCJBIl0sWzMsMCwiQiJdLFswLDAsIihcXGVwc2lsb24sIFxcZXRhKToiXSxbMCwxLCJmIiwyLHsiY3VydmUiOjR9XSxbMSwwLCJ1IiwyLHsiY3VydmUiOjR9XSxbMyw0LCIiLDAseyJsZXZlbCI6MSwic3R5bGUiOnsibmFtZSI6ImFkanVuY3Rpb24ifX1dXQ==
\[\begin{tikzcd}
	{(\epsilon, \eta):} & A && B
	\arrow[""{name=0, anchor=center, inner sep=0}, "f"', curve={height=24pt}, from=1-2, to=1-4]
	\arrow[""{name=1, anchor=center, inner sep=0}, "u"', curve={height=24pt}, from=1-4, to=1-2]
	\arrow["\dashv"{anchor=center, rotate=90}, draw=none, from=0, to=1]
\end{tikzcd}\]
\begin{itemize}
\item If the counit $\epsilon$ is invertible, call $f$ a \textit{reflector} and $u$ a \textit{reflection-inclusion}. In this case $f$. In case the counit is the identity, $f$ is called a \textit{lali} (\textit{left adjoint-left inverse}) and $u$ a \textit{rari} (\textit{right adjoint-right inverse}).
\item if the unit $\eta$ is invertible, call $f$ a  \textit{coreflection-inclusion} and $u$ a \textit{coreflector}. In case the unit is the identity, $f$ is called a \textit{lari} and $u$ a \textit{rali}.
\end{itemize}
\end{defi}

\begin{pozn}[Duals]\label{POZN_duals_lali}
A morphism $f$ is a reflector (a lali) in $\ck$ if and only if:
\begin{itemize}
\item it is a reflection-inclusion (a rari) in $\ck^{op}$,
\item it is a coreflector (a rali) in $\ck^{co}$,
\item it is a coreflection-inclusion (a lari) in $\ck^{coop}$.
\end{itemize}
\end{pozn}

%290
%257

\subsection{A characterization of algebras}\label{SUBS_3_charakterizace_algeber}

\begin{defi}\label{DEFI_D_coreflector}
Let $F: \ck \to \cl$ be a pseudofunctor. We will call a morphism $f: A \to B$ in $\ck$ an \textit{$F$-coreflector} if $Ff$ is a coreflector in the 2-category $\cl$. Similarly for the other variants from Definition \ref{DEFI_lali_rali_rari_lari}.
\end{defi}

\begin{pr}
Let $P: \CAT \to \prof$ be the canonical inclusion pseudofunctor. In Example \ref{PR_presheaf_psmonad_charakterizace_algeber} below we will show that a functor $f: \ca \to \cb$ between locally small categories is a $P$-coreflection-inclusion if and only if it is fully faithful and satisfies a certain smallness condition.
\end{pr}

\begin{pr}\label{PR_doctrinal_adj}
Consider the lax morphism classifier 2-comonad $Q_l$ associated to a 2-monad $T$ on a 2-category $\ck$. Denote by $J: \talgs \to \talgl$ the canonical inclusion to the Kleisli 2-category and by $U: \talgs \to \ck$ the forgetful 2-functor. Notice that by virtue of doctrinal adjunction \cite{doctrinaladj}, a strict algebra morphism is a $J$-reflector if and only if it is a $U$-reflector, that is, the underlying morphism in $\ck$ is a reflector.
\end{pr}

\begin{pozn}
Given a lax-idempotent pseudomonad $P$ on a 2-category $\ck$, 1-cells in $\ck$ that are $P$-left adjoints have been studied in the literature (\cite{bungebicomma}, \cite{adjfunthms}) under the name of \textit{$P$-admissible} 1-cells.
\end{pozn}

\noindent The following lemma is the left Kan pseudomonad version of \cite[Theorem 3.4]{threepieces}:

\begin{lemma}\label{THM_lemma_Dcoreflincl_is_JDcoreflincl}
Let $(D,y)$ be a left Kan pseudomonad on $\ck$. Denote by $D: \ck \to \ck$ the corresponding endo-pseudofunctor and by $J_D: \ck \to \ck_D$ the inclusion to the Kleisli 2-category. The following are equivalent for a 1-cell $f: B \to C$:
\begin{itemize}
\item $f$ is a $D$-coreflection-inclusion,
\item $f$ is a $J_D$-coreflection-inclusion.
\end{itemize}
\end{lemma}
%379
\begin{proof}
``$(1) \Rightarrow (2)$” follows from \cite[Proposition 1.3]{bungebicomma}: namely, the right adjoint to $Df$ in $\ck$ is actually a $D$-algebra homomorphism and thus is an adjoint in $\ck_D$. ``$(2) \Rightarrow (1)$” is obvious because we have the forgetful 2-functor $U_D: \ck_D \to \ck$ that satisfies $D = U_D J_D$.
\end{proof}

\begin{lemma}\label{THM_Lemma_Kan_extensions}
The following holds in a 2-category $\ck$:
\begin{itemize}
%304
\item Let $f \dashv u: B \to A$ be an adjunction with unit $\eta$ and let $(\mbbd,g^\mbbd)$ be the left Kan extension of $g: A' \to C$ along $y: A' \to A$. Then the diagram below left exhibits $g^\mbbd u$ as the left Kan extension of $g$ along $f y$:
% https://q.uiver.app/#q=WzAsOSxbMCwyLCJBJyJdLFswLDEsIkEiXSxbNCwyLCJDIl0sWzAsMCwiQiJdLFsyLDEsIkEiXSxbNSwyLCJCJyJdLFs3LDIsIkQiXSxbNSwxLCJCIl0sWzUsMCwiQyJdLFswLDEsInkiXSxbMCwyLCJnIiwyXSxbMSwzLCJmIl0sWzQsMiwiZ15cXG1hdGhiYntEfSJdLFsxLDQsIiIsMCx7ImxldmVsIjoyLCJzdHlsZSI6eyJoZWFkIjp7Im5hbWUiOiJub25lIn19fV0sWzMsNCwidSJdLFs1LDYsImYiLDJdLFs1LDcsImkiXSxbNyw4LCJrIl0sWzgsNiwiaCIsMCx7ImN1cnZlIjotNH1dLFs3LDYsImhrIiwwLHsiY3VydmUiOi0zfV0sWzEsMywiXFxldGEiLDIseyJsYWJlbF9wb3NpdGlvbiI6MTAsIm9mZnNldCI6NSwic2hvcnRlbiI6eyJ0YXJnZXQiOjUwfSwibGV2ZWwiOjJ9XSxbMCwxLCJcXG1hdGhiYntEfSIsMix7Im9mZnNldCI6NSwic2hvcnRlbiI6eyJzb3VyY2UiOjEwLCJ0YXJnZXQiOjEwfSwibGV2ZWwiOjJ9XSxbNSw3LCJcXGFscGhhIiwyLHsib2Zmc2V0Ijo1LCJzaG9ydGVuIjp7InRhcmdldCI6MjB9LCJsZXZlbCI6Mn1dXQ==
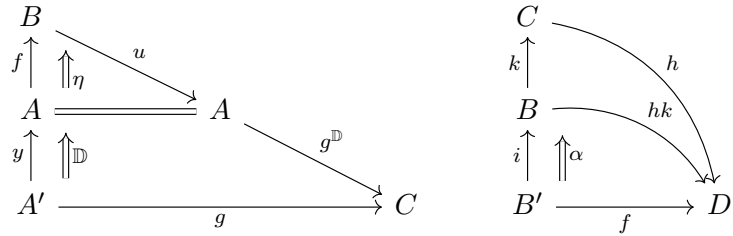
\[\begin{tikzcd}
	B &&&&& C \\
	A && A &&& B \\
	{A'} &&&& C & {B'} && D
	\arrow["y", from=3-1, to=2-1]
	\arrow["g"', from=3-1, to=3-5]
	\arrow["f", from=2-1, to=1-1]
	\arrow["{g^\mathbb{D}}", from=2-3, to=3-5]
	\arrow[Rightarrow, no head, from=2-1, to=2-3]
	\arrow["u", from=1-1, to=2-3]
	\arrow["f"', from=3-6, to=3-8]
	\arrow["i", from=3-6, to=2-6]
	\arrow["k", from=2-6, to=1-6]
	\arrow["h", curve={height=-24pt}, from=1-6, to=3-8]
	\arrow["hk", curve={height=-18pt}, from=2-6, to=3-8]
	\arrow["\eta"'{pos=0.1}, shift right=5, shorten >=5pt, Rightarrow, from=2-1, to=1-1]
	\arrow["{\mathbb{D}}"', shift right=5, shorten <=1pt, shorten >=1pt, Rightarrow, from=3-1, to=2-1]
	\arrow["\alpha"', shift right=5, shorten >=2pt, Rightarrow, from=3-6, to=2-6]
\end{tikzcd}\]
\item In the diagram above right, suppose that the top and outer diagrams are left Kan extensions. If all left Kan extensions along $k$ exist and have invertible unit, then $\alpha$ is a left Kan extension of $f$ along $j$.
\end{itemize}
\end{lemma}
\begin{proof}
The first point follows by composing the following bijections. For a 1-cell $h: B \to C$, the first one is given by the adjunction $f \dashv u$, the second one is given by the definition of $g^\mbbd$:
\[
\ck(B,C)(g^\mbbd u, h) \cong \ck(A,C)(g^\mbbd, h f) \cong \ck(A',C)(g, hfy).
\]
In the second point, assume we have a 2-cell $\beta$ as pictured below, and we want to find a unique 2-cell solving this equation:
% https://q.uiver.app/#q=WzAsNyxbMCwwLCJCIl0sWzAsMiwiQiciXSxbMiwyLCJEIl0sWzMsMSwiPSJdLFs0LDIsIkInIl0sWzQsMCwiQiJdLFs2LDIsIkQiXSxbMSwyLCJmIiwyXSxbMSwwLCJcXGlvdGEiXSxbMCwyLCJoayIsMV0sWzQsNSwiXFxpb3RhIl0sWzQsNiwiZiIsMl0sWzAsMiwibCIsMCx7ImN1cnZlIjotNX1dLFs1LDYsImwiLDAseyJjdXJ2ZSI6LTN9XSxbMSw5LCJcXGFscGhhIiwyLHsic2hvcnRlbiI6eyJzb3VyY2UiOjIwLCJ0YXJnZXQiOjIwfX1dLFs5LDEyLCJcXHRoZXRhIiwwLHsic2hvcnRlbiI6eyJzb3VyY2UiOjIwLCJ0YXJnZXQiOjIwfX1dLFs0LDEzLCJcXGJldGEiLDIseyJzaG9ydGVuIjp7InNvdXJjZSI6MjAsInRhcmdldCI6MjB9fV1d
\begin{equation}\label{EQ_PROOF_alphathetabeta}
\begin{tikzcd}
	B &&&& B \\
	&&& {=} \\
	{B'} && D && {B'} && D
	\arrow["f"', from=3-1, to=3-3]
	\arrow["\iota", from=3-1, to=1-1]
	\arrow[""{name=0, anchor=center, inner sep=0}, "hk"{description}, from=1-1, to=3-3]
	\arrow["\iota", from=3-5, to=1-5]
	\arrow["f"', from=3-5, to=3-7]
	\arrow[""{name=1, anchor=center, inner sep=0}, "l", curve={height=-30pt}, from=1-1, to=3-3]
	\arrow[""{name=2, anchor=center, inner sep=0}, "l", curve={height=-18pt}, from=1-5, to=3-7]
	\arrow["\alpha"', shorten <=6pt, shorten >=6pt, Rightarrow, from=3-1, to=0]
	\arrow["?", shorten <=5pt, shorten >=5pt, Rightarrow, from=0, to=1]
	\arrow["\beta"', shorten <=8pt, shorten >=8pt, Rightarrow, from=3-5, to=2]
\end{tikzcd}
\end{equation}

First note that we have a unique 2-cell $\theta$ making the following diagram equal (here $l^\mbba$ is the left Kan extension of $l$ along $k$ that exists by assumption):
% https://q.uiver.app/#q=WzAsMTQsWzAsMCwiQyJdLFswLDEsIkIiXSxbMCwyLCJCJyJdLFszLDIsIkQiXSxbNCwxLCI9Il0sWzUsMiwiQiciXSxbOCwyLCJEIl0sWzUsMSwiQiJdLFs1LDAsIkMiXSxbNiwxXSxbNiwyXSxbMSwxXSxbMSwyXSxbNiwwXSxbMiwzLCJmIiwyXSxbMiwxLCJcXGlvdGEiXSxbMSwzLCJoayIsMV0sWzEsMCwiayJdLFswLDMsImgiLDIseyJjdXJ2ZSI6LTF9XSxbMCwzLCJsXlxcbWJiYSIsMCx7ImN1cnZlIjotNX1dLFs1LDYsImYiLDJdLFs3LDgsImsiXSxbNSw3LCJcXGlvdGEiXSxbOCw2LCJsXlxcbWJiYSIsMCx7ImN1cnZlIjotNX1dLFs3LDYsImwiLDEseyJjdXJ2ZSI6LTF9XSxbMTAsOSwiXFxiZXRhIiwwLHsic2hvcnRlbiI6eyJzb3VyY2UiOjEwLCJ0YXJnZXQiOjQwfSwibGV2ZWwiOjJ9XSxbMTIsMTEsIlxcYWxwaGEiLDIseyJsYWJlbF9wb3NpdGlvbiI6MzAsIm9mZnNldCI6LTUsInNob3J0ZW4iOnsidGFyZ2V0Ijo4MH0sImxldmVsIjoyfV0sWzksMTMsIlxcbWJiYSIsMCx7ImxhYmVsX3Bvc2l0aW9uIjo3MCwic2hvcnRlbiI6eyJ0YXJnZXQiOjMwfSwibGV2ZWwiOjJ9XSxbMTgsMTksIlxcdGhldGEnIiwyLHsic2hvcnRlbiI6eyJzb3VyY2UiOjIwLCJ0YXJnZXQiOjIwfX1dXQ==
\[\begin{tikzcd}
	C &&&&& C & {} \\
	B & {} &&& {=} & B & {} \\
	{B'} & {} && D && {B'} & {} && D
	\arrow["f"', from=3-1, to=3-4]
	\arrow["\iota", from=3-1, to=2-1]
	\arrow["hk"{description}, from=2-1, to=3-4]
	\arrow["k", from=2-1, to=1-1]
	\arrow[""{name=0, anchor=center, inner sep=0}, "h"', curve={height=-6pt}, from=1-1, to=3-4]
	\arrow[""{name=1, anchor=center, inner sep=0}, "{l^\mbba}", curve={height=-30pt}, from=1-1, to=3-4]
	\arrow["f"', from=3-6, to=3-9]
	\arrow["k", from=2-6, to=1-6]
	\arrow["\iota", from=3-6, to=2-6]
	\arrow["{l^\mbba}", curve={height=-30pt}, from=1-6, to=3-9]
	\arrow["l"{description}, curve={height=-6pt}, from=2-6, to=3-9]
	\arrow["\beta", shorten <=1pt, shorten >=4pt, Rightarrow, from=3-7, to=2-7]
	\arrow["\alpha"'{pos=0.3}, shift left=5, shorten >=9pt, Rightarrow, from=3-2, to=2-2]
	\arrow["\mbba"{pos=0.7}, shorten >=3pt, Rightarrow, from=2-7, to=1-7]
	\arrow["{\theta'}"', shorten <=4pt, shorten >=4pt, Rightarrow, from=0, to=1]
\end{tikzcd}\]
Clearly, $\theta := \mbba^{-1} \circ \theta' k$ solves the equation \eqref{EQ_PROOF_alphathetabeta}, giving us \textbf{the existence} part of the proof. To show \textbf{the uniqueness}, let $\phi$ be a different 2-cell solving \eqref{EQ_PROOF_alphathetabeta}. Note that there exists a unique 2-cell $\phi'$ solving the following:
% https://q.uiver.app/#q=WzAsNyxbMCwyLCJCIl0sWzIsMywiRCJdLFszLDIsIj0iXSxbNCwyLCJCIl0sWzYsMywiRCJdLFswLDAsIkMiXSxbNCwwLCJDIl0sWzAsNSwiayJdLFswLDEsImhrIiwyXSxbMyw0LCJoayIsMl0sWzMsNiwiayJdLFs2LDQsImxeXFxtYmJhIiwwLHsiY3VydmUiOi01fV0sWzYsNCwiaCIsMV0sWzUsMSwibF5cXG1iYmEiLDAseyJjdXJ2ZSI6LTV9XSxbMCwxLCJsIiwxLHsibGFiZWxfcG9zaXRpb24iOjMwLCJjdXJ2ZSI6LTR9XSxbMTIsMTEsIlxccGhpJyIsMix7InNob3J0ZW4iOnsic291cmNlIjoyMCwidGFyZ2V0IjoyMH19XSxbMTQsMTMsIlxcbWJiYSIsMCx7InNob3J0ZW4iOnsic291cmNlIjoyMCwidGFyZ2V0IjoyMH19XSxbOCwxNCwiXFxwaGkiLDAseyJsYWJlbF9wb3NpdGlvbiI6MzAsInNob3J0ZW4iOnsic291cmNlIjoyMCwidGFyZ2V0IjoyMH19XV0=
\[\begin{tikzcd}
	C &&&& C \\
	\\
	B &&& {=} & B \\
	&& D &&&& D
	\arrow["k", from=3-1, to=1-1]
	\arrow[""{name=0, anchor=center, inner sep=0}, "hk"', from=3-1, to=4-3]
	\arrow["hk"', from=3-5, to=4-7]
	\arrow["k", from=3-5, to=1-5]
	\arrow[""{name=1, anchor=center, inner sep=0}, "{l^\mbba}", curve={height=-30pt}, from=1-5, to=4-7]
	\arrow[""{name=2, anchor=center, inner sep=0}, "h"{description}, from=1-5, to=4-7]
	\arrow[""{name=3, anchor=center, inner sep=0}, "{l^\mbba}", curve={height=-30pt}, from=1-1, to=4-3]
	\arrow[""{name=4, anchor=center, inner sep=0}, "l"{description, pos=0.3}, curve={height=-24pt}, from=3-1, to=4-3]
	\arrow["{\phi'}"', shorten <=5pt, shorten >=5pt, Rightarrow, from=2, to=1]
	\arrow["\mbba", shorten <=4pt, shorten >=4pt, Rightarrow, from=4, to=3]
	\arrow["\phi"{pos=0.3}, shorten <=3pt, shorten >=3pt, Rightarrow, from=0, to=4]
\end{tikzcd}\]
Pre-pasting this with $\alpha$ and using the diagram above this one, we see that $\phi' = \theta'$. From this we obtain:
\[
\mbba^{-1} \circ \theta' k = \mbba^{-1} \circ \phi' k = \mbba^{-1} \circ \mbba \circ \phi = \phi.
\]

\end{proof}
%301 proofs

\begin{prop}\label{THM_characterization_D-algs_D-corefl-incl}
Let $(D, y)$ be a left Kan pseudomonad on a 2-category $\ck$. Denote by $J_D$ the inclusion to the Kleisli 2-category and by $D$ the endo-pseudofunctor associated to the left Kan pseudomonad. The following are equivalent for an object $A \in \ck$:
\begin{enumerate}
\item $A$ admits the structure of a pseudo-$D$-algebra,
\item for every object $B \in \ck$, the left Kan extension of a 1-cell $B \to A$ along $y_B: B \to DB$ exists and has invertible unit. In other words, $\ck(-,A): \ck^{op} \to \cat$ sends each $y_B$ to a coreflector,
\item $y_A$ admits a reflector (left adjoint with invertible counit),
\item $\ck(-,A)$ sends $J_D$-coreflection-inclusions in $\ck$ to coreflectors in $\cat$,
\item $\ck(-,A)$ sends $D$-coreflection-inclusions in $\ck$ to coreflectors in $\cat$.
\end{enumerate}
\end{prop}
\begin{proof}
The equivalence ``$(1) \Leftrightarrow (3)$” is well known, for lax-idempotent 2-monads this has been done for example in \cite[Proposition 1.1.13]{elephant}, but the same argument works for lax-idempotent pseudomonads as well. ``$(1) \Rightarrow (2)$” is obvious.

For ``$(2) \Rightarrow (3)$”, denote by $(a: DA \to A,\mbba)$ the left Kan extension of $1_A$ along $y_A$. Because the identity 2-cell on $DA$ exibits $1_{DA}$ as the left Kan extension of $y_A$ along $y_A$, there exists a unique 2-cell $\eta$ making these 2-cells equal:
% https://q.uiver.app/#q=WzAsOCxbMSwwLCJEQSJdLFsyLDEsIkEiXSxbMCwxLCJBIl0sWzMsMSwiREEiXSxbNSwxLCJEQSJdLFs3LDEsIkRBIl0sWzQsMSwiQSJdLFs2LDAsIkEiXSxbMiwxLCIiLDIseyJsZXZlbCI6Miwic3R5bGUiOnsiaGVhZCI6eyJuYW1lIjoibm9uZSJ9fX1dLFsxLDMsInlfQSIsMl0sWzIsMCwieV9BIl0sWzAsMSwiYSJdLFs2LDQsInlfQSJdLFs0LDUsIiIsMCx7ImxldmVsIjoyLCJzdHlsZSI6eyJoZWFkIjp7Im5hbWUiOiJub25lIn19fV0sWzQsNywiYSJdLFs3LDUsInlfQSJdLFsxMyw3LCJcXGV0YSIsMCx7InNob3J0ZW4iOnsic291cmNlIjoyMH19XSxbOCwwLCJcXG1iYmEiLDIseyJzaG9ydGVuIjp7InNvdXJjZSI6MjB9fV1d
\[\begin{tikzcd}
	& DA &&&&& A \\
	A && A & DA & A & DA && DA
	\arrow[""{name=0, anchor=center, inner sep=0}, Rightarrow, no head, from=2-1, to=2-3]
	\arrow["{y_A}"', from=2-3, to=2-4]
	\arrow["{y_A}", from=2-1, to=1-2]
	\arrow["a", from=1-2, to=2-3]
	\arrow["{y_A}", from=2-5, to=2-6]
	\arrow[""{name=1, anchor=center, inner sep=0}, Rightarrow, no head, from=2-6, to=2-8]
	\arrow["a", from=2-6, to=1-7]
	\arrow["{y_A}", from=1-7, to=2-8]
	\arrow["\eta", shorten <=3pt, Rightarrow, from=1, to=1-7]
	\arrow["\mbba"', shorten <=3pt, Rightarrow, from=0, to=1-2]
\end{tikzcd}\]

We will now show that $(\mbba^{-1}, \eta): a \dashv y_A$ is an adjunction. The triangle identity $y_A \mbba^{-1} \circ \eta y_A = 1_{y_A}$ is guaranteed by the above formula -- let us prove the other one:
\[
\mbba^{-1}a \circ a\eta  = 1_a.
\]
Because $a$ is the left Kan extension along $y_A$, it suffices to prove that both sides of this equation become equal after pre-composing them with $y_A$. It then becomes:
\[
\mbba^{-1}ay_A \circ a\eta y_A = \mbba^{-1}ay_A \circ a y_A \mbba = \mbba^{-1}ay_A \circ \mbba a y_A = 1_{ay_A}.
\]

``$(4) \Leftrightarrow (5)$” follows from Lemma \ref{THM_lemma_Dcoreflincl_is_JDcoreflincl} and ``$(5) \Rightarrow (2)$” is obvious since $y_B$ is a $D$-coreflection-inclusion.

We will now prove ``$(2) \Rightarrow (5)$”. Let $f: B \to C$ such that there is an adjunction in $\ck_D$ where the \textbf{unit} $\eta$ is invertible:
% https://q.uiver.app/#q=WzAsMyxbMSwwLCJEQiJdLFszLDAsIkRDIl0sWzAsMCwiKFxcZXBzaWxvbixcXGV0YSk6Il0sWzAsMSwiRGYiLDIseyJjdXJ2ZSI6NH1dLFsxLDAsInIiLDIseyJjdXJ2ZSI6NH1dLFszLDQsIiIsMCx7ImxldmVsIjoxLCJzdHlsZSI6eyJuYW1lIjoiYWRqdW5jdGlvbiJ9fV1d
\[\begin{tikzcd}
	{(\epsilon,\eta):} & DB && DC
	\arrow[""{name=0, anchor=center, inner sep=0}, "Df"', curve={height=24pt}, from=1-2, to=1-4]
	\arrow[""{name=1, anchor=center, inner sep=0}, "r"', curve={height=24pt}, from=1-4, to=1-2]
	\arrow["\dashv"{anchor=center, rotate=90}, draw=none, from=0, to=1]
\end{tikzcd}\]
We wish to show that the functor $f^*: \ck(C,A) \to \ck(B,A)$ has a left adjoint with invertible unit. We will define this left adjoint by the following formula:
\[
L: (g: B \to A) \mapsto (g^\mbba \circ r \circ y_C : C \to A)
\]
Define the component of the unit $\widetilde{\eta}$ at $g: B \to A$ as the following composite 2-cell:
% https://q.uiver.app/#q=WzAsMTAsWzAsMSwiQiJdLFswLDAsIkMiXSxbNSwxLCJEQiJdLFs0LDAsIkRDIl0sWzYsMSwiQSJdLFsyLDEsIkRCIl0sWzQsMV0sWzMsMl0sWzMsMV0sWzIsMF0sWzAsNSwieV9CIl0sWzAsMSwiZiJdLFsxLDMsInlfQyJdLFszLDIsInIiXSxbNSwyLCIiLDAseyJsZXZlbCI6Miwic3R5bGUiOnsiaGVhZCI6eyJuYW1lIjoibm9uZSJ9fX1dLFsyLDQsImdeXFxtYmJhIl0sWzUsMywiRGYiXSxbMCw0LCJnIiwyLHsiY3VydmUiOjV9XSxbNiwzLCJcXGV0YSIsMCx7InNob3J0ZW4iOnsic291cmNlIjozMH0sImxldmVsIjoyfV0sWzUsOSwiXFxtYmJkXnstMX0iLDAseyJzaG9ydGVuIjp7InNvdXJjZSI6MjB9LCJsZXZlbCI6Mn1dLFs3LDgsIlxcbWJiYSIsMCx7ImxhYmVsX3Bvc2l0aW9uIjo3MCwic2hvcnRlbiI6eyJzb3VyY2UiOjgwfSwibGV2ZWwiOjJ9XV0=
\[\begin{tikzcd}
	C && {} && DC \\
	B && DB & {} & {} & DB & A \\
	&&& {}
	\arrow["{y_B}", from=2-1, to=2-3]
	\arrow["f", from=2-1, to=1-1]
	\arrow["{y_C}", from=1-1, to=1-5]
	\arrow["r", from=1-5, to=2-6]
	\arrow[Rightarrow, no head, from=2-3, to=2-6]
	\arrow["{g^\mbba}", from=2-6, to=2-7]
	\arrow["Df", from=2-3, to=1-5]
	\arrow["g"', curve={height=30pt}, from=2-1, to=2-7]
	\arrow["\eta", shorten <=3pt, Rightarrow, from=2-5, to=1-5]
	\arrow["{\mbbd^{-1}}", shorten <=2pt, Rightarrow, from=2-3, to=1-3]
	\arrow["\mbba"{pos=0.7}, shorten <=9pt, Rightarrow, from=3-4, to=2-4]
\end{tikzcd}\]

We wish to show that this has the universal property of the unit, in other words, $g^\mbba \circ r \circ y_C$ is the left Kan extension of $g: B \to A$ along $f: B \to C$.

By Lemma \ref{THM_Lemma_Kan_extensions} \textbf{point 1}, $g^\mbba \circ r$ is the left Kan extension of $g$ along $Df \circ y_B$. Equivalently it is a left Kan extension of $g$ along $y_C f$ with the 2-cell component given by the composite 2-cell above. Since $g^\mbba r$ is a $D$-morphism, $g^\mbba r$ (with the identity 2-cell component) is the left Kan extension of $g^\mbba r y_C$ along $y_C$. By Lemma \ref{THM_Lemma_Kan_extensions} \textbf{point 2}, for $h := g^\mbba r$, $k := y_C$, $i := f$, $f := g$ and $\alpha$ the 2-cell above, the result follows.

\end{proof}

% CHARAKTERIZACE D-ALGEBER OLD
%372

%270
%271

\begin{pozn}
Using the terminology of \cite[Definition 1.2]{kaninj}, in Theorem \ref{THM_characterization_D-algs_D-corefl-incl}, the equivalence ``$(1) \Leftrightarrow (4)$” says that an object $A$ is a pseudo-$D$-algebra if and only if it is \textit{left Kan injective} with respect to the class of 1-cells given by $J_D$-coreflection-inclusions.

Let us also note that a version of ``$(1) \Rightarrow (5)$” for $D$-left adjoints in Theorem \ref{THM_characterization_D-algs_D-corefl-incl} has already been proven in \cite[Proposition 1.5]{bungebicomma}.
\end{pozn}

\begin{pozn}\label{POZN_strict_charakterizace_D_algeber}
Given a left Kan 2-monad $(D,y)$, a pseudo-$D$-algebra $C$ will be said to be \textit{normal} if the left Kan extension 2-cell $\mbbc_f$ in Definition \ref{DEFI_D_algebra} is the identity for all 1-cells $f$. Notice that a variation of Proposition \ref{THM_characterization_D-algs_D-corefl-incl} may be proven for normal pseudo-$D$-algebras, where we replace all invertible 2-cells by identities, for instance replace a ``reflector” by a ``lali”.
\end{pozn}

\noindent In the remainder of this section we will demonstrate Proposition \ref{THM_characterization_D-algs_D-corefl-incl} on the case of small presheaf pseudomonad from Example \ref{PR_presheafpsmonad}. An application to the lax morphism classifier 2-comonads will be described in Section \ref{SEKCE_applications_to_twodim}.

\begin{pr}\label{PR_presheaf_psmonad_charakterizace_algeber}
Consider the small presheaf pseudomonad $P$ on $\CAT$. Note that if we pass to a bigger universe and use the bicategory $\PROF$ of locally small categories and \textbf{all} profunctors, for any functor $f: \ca \to \cb$, the small profunctor $Pf = \cb(-,f-): \cb^{op} \times \ca \to \set$ has a right adjoint:
% https://q.uiver.app/#q=WzAsMixbMCwwLCJcXGNhIl0sWzIsMCwiXFxjYiJdLFswLDEsIlxcY2IoLSxmLSkiLDIseyJjdXJ2ZSI6NH1dLFsxLDAsIlxcY2IoZi0sLSkiLDIseyJjdXJ2ZSI6NH1dLFsyLDMsIiIsMCx7ImxldmVsIjoxLCJzdHlsZSI6eyJuYW1lIjoiYWRqdW5jdGlvbiJ9fV1d
\[\begin{tikzcd}
	\ca && \cb
	\arrow[""{name=0, anchor=center, inner sep=0}, "{\cb(-,f-)}"', curve={height=24pt}, from=1-1, to=1-3]
	\arrow[""{name=1, anchor=center, inner sep=0}, "{\cb(f-,-)}"', curve={height=24pt}, from=1-3, to=1-1]
	\arrow["\dashv"{anchor=center, rotate=90}, draw=none, from=0, to=1]
\end{tikzcd}\]

We will call a functor $f: \ca \to \cb$ \textit{small} if the right adjoint is also a small profunctor (belongs to $\text{Prof}$). Clearly, this happens if and only if $Pf$ has a right adjoint in $\prof$.

Next, note that the unit of the adjunction is a collection of functions for every pair $(a',a'') \in \ca^{op} \times \ca$ like this:
\begin{align*}
\ca(a',a'') &\to \int^{b \in \cb} \cb(fa',b) \times \cb(b,fa''),\\
(\theta: a' \to a'') &\mapsto [1_{fa'},f(\theta)].
\end{align*}
As is readily seen, the unit is invertible if and only if $f$ is fully faithful. So a functor $f: \ca \to \cb$ is a $P$-coreflection-inclusion if and only if it is fully faithful and small. The precomposition functor $f^* : \CAT(\cb, \cc) \to \CAT(\ca, \cc)$ is a coreflector if and only if left Kan extensions along $f$ exist in $\cc$. Theorem \ref{THM_characterization_D-algs_D-corefl-incl} for the small presheaf pseudomonad now gives a folklore result: a category $\cc$ is cocomplete if and only if left Kan extensions along small fully faithful functors exist in $\cc$.
%375
\end{pr}

\subsection{Colax adjunctions out of the Kleisli 2-category}\label{SUBS_3_big_lax_adj_thm}

%384 ASSUMPTION

\begin{prop}\label{THM_bigadjthm_HL_is_D-algebra}
Let $(D, y)$ be a left Kan pseudomonad on a 2-category $\ck$ and assume there are pseudofunctors $G,H$ and a biadjunction as pictured below:
% https://q.uiver.app/#q=WzAsMyxbMCwwLCJcXGNrIl0sWzIsMCwiXFxja19EIl0sWzQsMCwiXFxjbCJdLFswLDEsIkpfRCIsMl0sWzEsMiwiRyIsMl0sWzIsMCwiSCIsMix7ImN1cnZlIjo1fV0sWzEsNSwiIiwwLHsibGV2ZWwiOjEsInN0eWxlIjp7Im5hbWUiOiJhZGp1bmN0aW9uIn19XV0=
\[\begin{tikzcd}
	\ck && {\ck_D} && \cl
	\arrow["{J_D}"', from=1-1, to=1-3]
	\arrow["G"', from=1-3, to=1-5]
	\arrow[""{name=0, anchor=center, inner sep=0}, "H"', curve={height=30pt}, from=1-5, to=1-1]
	\arrow["\dashv"{anchor=center, rotate=90}, draw=none, from=1-3, to=0]
\end{tikzcd}\]
Then for every object $L \in \cl$, the object $HL$ admits the structure of a pseudo-$D$-algebra.
\end{prop}
\begin{proof}
By Proposition \ref{THM_characterization_D-algs_D-corefl-incl}, it suffices to show that $\ck(-,HL): \ck^{op} \to \cat$ sends $J_D$-coreflection-inclusions to coreflectors. Notice that we have the following pseudo-natural equivalence:
\[
\ck(-,HL) \simeq \cl(GJ_D-,L) = \cl(G-,L) \circ J_D.
\]
Now, by definition, $J_D$ sends $J_D$-coreflection-inclusions to coreflection-inclusions. Since $\cl(G-,L)$ is a (contravariant) pseudofunctor, it sends coreflection-inclusions to coreflectors. We thus obtain the result.
\end{proof}

\begin{pozn}\label{POZN_explicit_form_h_L}
Going through the proof of Proposition \ref{THM_characterization_D-algs_D-corefl-incl} for the case of $HL$, we see that the algebra multiplication map $h_L : DHL \to HL$ (the reflector of the morphism $y_{HL}: HL \to DHL$) is given by the following composite:
% https://q.uiver.app/#q=WzAsNCxbMCwwLCJESEwiXSxbMiwwLCJIR0ReMkhMIl0sWzQsMCwiSEdESEwiXSxbNiwwLCJITCJdLFswLDEsImNfe0RITH0iXSxbMSwyLCJIR3Bfe0RITH0iXSxbMiwzLCJIc19MIl1d
\[\begin{tikzcd}
	DHL && {HGD^2HL} && HGDHL && HL
	\arrow["{c_{DHL}}", from=1-1, to=1-3]
	\arrow["{HGp_{DHL}}", from=1-3, to=1-5]
	\arrow["{Hs_L}", from=1-5, to=1-7]
\end{tikzcd}\]
Also, the counit of the adjunction $h_L \dashv y_{HL}$, an invertible 2-cell $\epsilon_L: h_L y_{HL} \Rightarrow 1_{HL}$, is given by the following:
% https://q.uiver.app/#q=WzAsOCxbMCwxLCJITCJdLFsxLDAsIkRITCJdLFsyLDAsIkhHRF4ySEwiXSxbNCwwLCJIR0RITCJdLFs1LDEsIkhMIl0sWzEsMSwiSEdESEwiXSxbMiwxXSxbMiwyXSxbMCwxLCJ5X3tITH0iXSxbMSwyLCJjX3tESEx9Il0sWzIsMywiSEdwX3tESEx9Il0sWzMsNCwiSHNfTCJdLFswLDUsImNfe0hMfSIsMl0sWzUsMiwiSEdEeV97SEx9IiwxXSxbNSwzLCIiLDEseyJjdXJ2ZSI6MywibGV2ZWwiOjIsInN0eWxlIjp7ImhlYWQiOnsibmFtZSI6Im5vbmUifX19XSxbMCw0LCIiLDEseyJvZmZzZXQiOjIsImN1cnZlIjo1LCJsZXZlbCI6Miwic3R5bGUiOnsiaGVhZCI6eyJuYW1lIjoibm9uZSJ9fX1dLFsxLDUsImNfe3lfe0hMfX0iLDIseyJzaG9ydGVuIjp7InNvdXJjZSI6MjB9LCJsZXZlbCI6Mn1dLFsyLDYsIihIR1xcUHNpKV97SEx9IiwwLHsibGFiZWxfcG9zaXRpb24iOjQwLCJzaG9ydGVuIjp7InRhcmdldCI6NDB9LCJsZXZlbCI6Mn1dLFs2LDcsIlxcdGF1XnstMX1fTCIsMCx7ImxhYmVsX3Bvc2l0aW9uIjo0MCwic2hvcnRlbiI6eyJ0YXJnZXQiOjEwfSwibGV2ZWwiOjJ9XV0=
\[\begin{tikzcd}
	& DHL & {HGD^2HL} && HGDHL \\
	HL & HGDHL & {} &&& HL \\
	&& {}
	\arrow["{y_{HL}}", from=2-1, to=1-2]
	\arrow["{c_{DHL}}", from=1-2, to=1-3]
	\arrow["{HGp_{DHL}}", from=1-3, to=1-5]
	\arrow["{Hs_L}", from=1-5, to=2-6]
	\arrow["{c_{HL}}"', from=2-1, to=2-2]
	\arrow["{HGDy_{HL}}"{description}, from=2-2, to=1-3]
	\arrow[curve={height=18pt}, Rightarrow, no head, from=2-2, to=1-5]
	\arrow[shift right=2, curve={height=30pt}, Rightarrow, no head, from=2-1, to=2-6]
	\arrow["{c_{y_{HL}}}"', shorten <=2pt, Rightarrow, from=1-2, to=2-2]
	\arrow["{(HG\Psi)_{HL}}"{pos=0.4}, shorten >=4pt, Rightarrow, from=1-3, to=2-3]
	\arrow["{\tau^{-1}_L}"{pos=0.4}, shorten >=1pt, Rightarrow, from=2-3, to=3-3]
\end{tikzcd}\]
\end{pozn}

%kos252 old
%387 old but newer
\begin{theo}[\textbf{The main colax adjunction theorem}]\label{THM_BIG_lax_adj_thm}
Let $(D, y)$ be a left Kan pseudomonad on a 2-category $\ck$. Any biadjunction whose left adjoint factorizes through the Kleisli 2-category $\ck_D$ induces a colax adjunction pictured below:
% https://q.uiver.app/#q=WzAsNixbMCwwLCJcXGNrIl0sWzIsMCwiXFxja19EIl0sWzQsMCwiXFxjbCJdLFs1LDAsIlxccmlnaHRzcXVpZ2Fycm93Il0sWzYsMCwiXFxja19EIl0sWzgsMCwiXFxjbCJdLFswLDEsIkpfRCIsMl0sWzEsMiwiRyIsMl0sWzIsMCwiSCIsMix7ImN1cnZlIjo1fV0sWzQsNSwiRyIsMix7ImN1cnZlIjo0fV0sWzUsNCwiSl9ESCIsMix7ImN1cnZlIjo0fV0sWzEsOCwiIiwwLHsibGV2ZWwiOjEsInN0eWxlIjp7Im5hbWUiOiJhZGp1bmN0aW9uIn19XSxbOSwxMCwiXFx0dG9wIiwxLHsic2hvcnRlbiI6eyJzb3VyY2UiOjIwLCJ0YXJnZXQiOjIwfSwic3R5bGUiOnsiYm9keSI6eyJuYW1lIjoibm9uZSJ9LCJoZWFkIjp7Im5hbWUiOiJub25lIn19fV1d
\[\begin{tikzcd}
	\ck && {\ck_D} && \cl & \rightsquigarrow & {\ck_D} && \cl
	\arrow["{J_D}"', from=1-1, to=1-3]
	\arrow["G"', from=1-3, to=1-5]
	\arrow[""{name=0, anchor=center, inner sep=0}, "H"', curve={height=30pt}, from=1-5, to=1-1]
	\arrow[""{name=1, anchor=center, inner sep=0}, "G"', curve={height=24pt}, from=1-7, to=1-9]
	\arrow[""{name=2, anchor=center, inner sep=0}, "{J_DH}"', curve={height=24pt}, from=1-9, to=1-7]
	\arrow["\dashv"{anchor=center, rotate=90}, draw=none, from=1-3, to=0]
	\arrow["\ttop"{description}, draw=none, from=1, to=2]
\end{tikzcd}\]
%383
\end{theo}
\begin{proof}
Denote the unit, counit and the modifications of the biadjunction as follows:%Note that the transformations $s,c$ and the modifications $\sigma, \tau$ have the following form:
\begin{align*}
s &: GJ_D H \Rightarrow 1_\cl, &&\sigma: sGJ_D \circ GJ_Dc \cong 1_{GJ_D} ,\\
c &: 1_\ck \Rightarrow HGJ_D,  &&\tau : 1_{H} \cong Hs \circ cH.
\end{align*}

We will show that the components of the counit $s_L: GDHL \to L$ are coherently closed for $G$-lifts. By (the dual of) Theorem \ref{THM_Bunge}, there is a right colax adjoint to $G$. We will prove that it is isomorphic to $J_D H$.

Let us first prove the following: given a 1-cell $l: GDA \to L$ in $\cl$, \textbf{any} pair $(Dl', \lambda)$ where $l': A \to HL$ is a 1-cell and $\lambda$ is an invertible 2-cell as pictured below exhibits $Dl'$ as the right $G$-lift of $l$ along $s_L$:
% https://q.uiver.app/#q=WzAsMyxbMCwwLCJMIl0sWzIsMCwiR0RITCJdLFsyLDIsIkdEQSJdLFsyLDEsIkdEbCciLDJdLFsxLDAsInNfTCIsMl0sWzIsMCwibCJdLFsxLDUsIlxcbGFtYmRhIiwyLHsic2hvcnRlbiI6eyJ0YXJnZXQiOjIwfX1dXQ==
\[\begin{tikzcd}
	L && GDHL \\
	\\
	&& GDA
	\arrow["{GDl'}"', from=3-3, to=1-3]
	\arrow["{s_L}"', from=1-3, to=1-1]
	\arrow[""{name=0, anchor=center, inner sep=0}, "l", from=3-3, to=1-1]
	\arrow["\lambda"', shorten >=6pt, Rightarrow, from=1-3, to=0]
\end{tikzcd}\]

By Theorem \ref{THM_bigadjthm_HL_is_D-algebra}, $HL$ has the structure of a $D$-algebra. Denoting its multiplication map by $h_L$ as in Remark \ref{POZN_explicit_form_h_L}, we have the following composite adjunction with invertible counit:
% https://q.uiver.app/#q=WzAsMyxbMCwwLCJcXGNrKEEsSEwpIl0sWzIsMCwiXFxjayhBLERITCkiXSxbNCwwLCJcXGNrX0QoREEsREhMKSJdLFswLDEsIih5X3tITH0pXyoiLDJdLFswLDIsIkpfRCIsMix7ImN1cnZlIjo1fV0sWzEsMCwiKGhfTClfKiIsMix7ImN1cnZlIjo1fV0sWzEsMiwicF97REhMfUpfRCgtKSIsMl0sWzIsMSwiVV9EKC0peV9BIiwyLHsiY3VydmUiOjV9XSxbMiwwLCIoLSleXFwjIiwyLHsiY3VydmUiOjV9XSxbNSwzLCIiLDIseyJsZXZlbCI6MSwiZWRnZV9hbGlnbm1lbnQiOnsic291cmNlIjpmYWxzZSwidGFyZ2V0IjpmYWxzZX0sInN0eWxlIjp7Im5hbWUiOiJhZGp1bmN0aW9uIn19XSxbNiw3LCJcXHNpbWVxIiwxLHsic2hvcnRlbiI6eyJzb3VyY2UiOjIwLCJ0YXJnZXQiOjIwfSwic3R5bGUiOnsiYm9keSI6eyJuYW1lIjoibm9uZSJ9LCJoZWFkIjp7Im5hbWUiOiJub25lIn19fV0sWzgsMSwiOj0iLDMseyJzaG9ydGVuIjp7InNvdXJjZSI6MjB9LCJzdHlsZSI6eyJib2R5Ijp7Im5hbWUiOiJub25lIn0sImhlYWQiOnsibmFtZSI6Im5vbmUifX19XSxbMSw0LCJcXGNvbmciLDEseyJzaG9ydGVuIjp7InRhcmdldCI6MjB9LCJzdHlsZSI6eyJib2R5Ijp7Im5hbWUiOiJub25lIn0sImhlYWQiOnsibmFtZSI6Im5vbmUifX19XV0=
\begin{equation}\label{EQ_hL_yA_yHL_JD}
\begin{tikzcd}
	{\ck(A,HL)} && {\ck(A,DHL)} && {\ck_D(DA,DHL)}
	\arrow[""{name=0, anchor=center, inner sep=0}, "{(y_{HL})_*}"', from=1-1, to=1-3]
	\arrow[""{name=0p, anchor=center, inner sep=0}, phantom, from=1-1, to=1-3, start anchor=center, end anchor=center]
	\arrow[""{name=1, anchor=center, inner sep=0}, "{J_D}"', curve={height=30pt}, from=1-1, to=1-5]
	\arrow[""{name=2, anchor=center, inner sep=0}, "{(h_L)_*}"', curve={height=30pt}, from=1-3, to=1-1]
	\arrow[""{name=2p, anchor=center, inner sep=0}, phantom, from=1-3, to=1-1, start anchor=center, end anchor=center, curve={height=30pt}]
	\arrow[""{name=3, anchor=center, inner sep=0}, "{p_{DHL}J_D(-)}"', from=1-3, to=1-5]
	\arrow[""{name=4, anchor=center, inner sep=0}, "{U_D(-)y_A}"', curve={height=30pt}, from=1-5, to=1-3]
	\arrow[""{name=5, anchor=center, inner sep=0}, "{(-)^\#}"', curve={height=60pt},shift right=3, from=1-5, to=1-1]
	\arrow["\dashv"{anchor=center, rotate=-93}, draw=none, from=2p, to=0p]
	\arrow["\simeq"{description}, draw=none, from=3, to=4]
	\arrow["{:=}"{marking, allow upside down}, draw=none, from=5, to=1-3]
	\arrow["\cong"{description}, draw=none, from=1-3, to=1]
\end{tikzcd}
\end{equation}

Notice that there is an isomorphism:
\[
\beth: s_L GJ_D(-)^\# \cong s_L G(-): \ck_D(DA,DHL) \to \cl(GDA,L),
\]
with the component at $f: DA \to DHL$ being given by the 2-cell\footnote{We have not shown it in this diagram, but the 2-cell has to be pre-composed with the associators for the pseudofunctor $GD$ so that its source really equals $s_L GDf^\#$.}:
% https://q.uiver.app/#q=WzAsMTAsWzAsMCwiR0ReMkhMIl0sWzIsMCwiR0RIR0ReMkhMIl0sWzQsMCwiR0RIR0RITCJdLFs2LDAsIkdESEwiXSxbMiwyLCJHRF4ySEwiXSxbNiwyLCJMIl0sWzQsMiwiR0RITCJdLFswLDIsIkdEXjJBIl0sWzIsMywiR0RBIl0sWzAsNCwiR0RBIl0sWzAsNCwiIiwxLHsibGV2ZWwiOjIsInN0eWxlIjp7ImhlYWQiOnsibmFtZSI6Im5vbmUifX19XSxbMCwxLCJHRGNfe0RITH0iXSxbMSw0LCJzX3tHRF4ySEx9IiwxXSxbNCw2LCJHcF97REhMfSIsMV0sWzYsNSwic19MIiwyXSxbMSwyLCJHREhHcF97REhMfSJdLFsyLDYsInNfe0dESEx9IiwxXSxbMiwzLCJHREhzX0wiXSxbMyw1LCJzX0wiXSxbOCw2LCJHZiIsMl0sWzcsOCwiR3Bfe0RBfSJdLFs3LDAsIkdEZiJdLFs5LDcsIkdEeV9BIl0sWzksOCwiIiwxLHsibGV2ZWwiOjIsInN0eWxlIjp7ImhlYWQiOnsibmFtZSI6Im5vbmUifX19XSxbNCw4LCIoR3ApX2YiLDAseyJsZXZlbCI6Mn1dLFswLDMsIkdEaF9MIiwwLHsib2Zmc2V0IjotMiwiY3VydmUiOi01fV0sWzE1LDEzLCJzX3tHcF97REhMfX0iLDAseyJzaG9ydGVuIjp7InNvdXJjZSI6MjAsInRhcmdldCI6MjB9LCJlZGdlX2FsaWdubWVudCI6eyJzb3VyY2UiOmZhbHNlLCJ0YXJnZXQiOmZhbHNlfX1dLFsxNywxNCwic197c19MfSIsMCx7InNob3J0ZW4iOnsic291cmNlIjoyMCwidGFyZ2V0IjoyMH0sImVkZ2VfYWxpZ25tZW50Ijp7InNvdXJjZSI6ZmFsc2UsInRhcmdldCI6ZmFsc2V9fV0sWzExLDEwLCJcXHNpZ21hX3tESEx9IiwwLHsic2hvcnRlbiI6eyJzb3VyY2UiOjIwLCJ0YXJnZXQiOjIwfSwiZWRnZV9hbGlnbm1lbnQiOnsic291cmNlIjpmYWxzZSwidGFyZ2V0IjpmYWxzZX19XSxbMjAsMjMsIihHXFxQc2kpX0EiLDIseyJzaG9ydGVuIjp7InNvdXJjZSI6MjAsInRhcmdldCI6MjB9fV0sWzI1LDE1LCIoKikiLDEseyJsYWJlbF9wb3NpdGlvbiI6MjAsInNob3J0ZW4iOnsic291cmNlIjoyMCwidGFyZ2V0IjoyMH0sImVkZ2VfYWxpZ25tZW50Ijp7InNvdXJjZSI6ZmFsc2UsInRhcmdldCI6ZmFsc2V9LCJzdHlsZSI6eyJib2R5Ijp7Im5hbWUiOiJub25lIn0sImhlYWQiOnsibmFtZSI6Im5vbmUifX19XV0=
\begin{equation}\label{EQ_sigma_s_s_p_psi}
\begin{minipage}{0.8\textwidth}
\adjustbox{scale=0.9,center}{
\begin{tikzcd}
	{GD^2HL} && {GDHGD^2HL} && GDHGDHL && GDHL \\
	\\
	{GD^2A} && {GD^2HL} && GDHL && L \\
	&& GDA \\
	GDA
	\arrow[""{name=0, anchor=center, inner sep=0}, Rightarrow, no head, from=1-1, to=3-3]
	\arrow[""{name=0p, anchor=center, inner sep=0}, phantom, from=1-1, to=3-3, start anchor=center, end anchor=center]
	\arrow[""{name=1, anchor=center, inner sep=0}, "{GDc_{DHL}}", from=1-1, to=1-3]
	\arrow[""{name=1p, anchor=center, inner sep=0}, phantom, from=1-1, to=1-3, start anchor=center, end anchor=center]
	\arrow["{s_{GD^2HL}}"{description}, from=1-3, to=3-3]
	\arrow[""{name=2, anchor=center, inner sep=0}, "{Gp_{DHL}}"{description}, from=3-3, to=3-5]
	\arrow[""{name=2p, anchor=center, inner sep=0}, phantom, from=3-3, to=3-5, start anchor=center, end anchor=center]
	\arrow[""{name=3, anchor=center, inner sep=0}, "{s_L}"', from=3-5, to=3-7]
	\arrow[""{name=3p, anchor=center, inner sep=0}, phantom, from=3-5, to=3-7, start anchor=center, end anchor=center]
	\arrow[""{name=4, anchor=center, inner sep=0}, "{GDHGp_{DHL}}", from=1-3, to=1-5]
	\arrow[""{name=4p, anchor=center, inner sep=0}, phantom, from=1-3, to=1-5, start anchor=center, end anchor=center]
	\arrow[""{name=4p, anchor=center, inner sep=0}, phantom, from=1-3, to=1-5, start anchor=center, end anchor=center]
	\arrow["{s_{GDHL}}"{description}, from=1-5, to=3-5]
	\arrow[""{name=5, anchor=center, inner sep=0}, "{GDHs_L}", from=1-5, to=1-7]
	\arrow[""{name=5p, anchor=center, inner sep=0}, phantom, from=1-5, to=1-7, start anchor=center, end anchor=center]
	\arrow["{s_L}", from=1-7, to=3-7]
	\arrow["Gf"', from=4-3, to=3-5]
	\arrow[""{name=6, anchor=center, inner sep=0}, "{Gp_{DA}}", from=3-1, to=4-3]
	\arrow["GDf", from=3-1, to=1-1]
	\arrow["{GDy_A}", from=5-1, to=3-1]
	\arrow[""{name=7, anchor=center, inner sep=0}, Rightarrow, no head, from=5-1, to=4-3]
	\arrow["{(Gp)_f}", Rightarrow, from=3-3, to=4-3]
	\arrow[""{name=8, anchor=center, inner sep=0}, "{GDh_L}", shift left=2, curve={height=-30pt}, from=1-1, to=1-7]
	\arrow[""{name=8p, anchor=center, inner sep=0}, phantom, from=1-1, to=1-7, start anchor=center, end anchor=center, shift left=2, curve={height=-30pt}]
	\arrow["{s_{Gp_{DHL}}}", shorten <=9pt, shorten >=9pt, Rightarrow, from=4p, to=2p]
	\arrow["{s_{s_L}}", shorten <=9pt, shorten >=9pt, Rightarrow, from=5p, to=3p]
	\arrow["{\sigma_{DHL}}", shorten <=4pt, shorten >=4pt, Rightarrow, from=1p, to=0p]
	\arrow["{(G\Psi)_A}"', shorten <=4pt, shorten >=4pt, Rightarrow, from=6, to=7]
	\arrow["{}"{description, pos=0.2}, draw=none, from=8p, to=4p]
\end{tikzcd}
}
\end{minipage}
\end{equation}

We have the following chain of bijections:
\begin{align*}
\ck_D(DA,DHL)(f,Dl') &\stackrel{(A)}{\cong} \ck(A,HL)(f^\#,l') \\
&\stackrel{(B)}{\cong} \cl(GDA,L)(s_L \circ GDf^\#, s_L \circ GDl')\\
&\stackrel{(C)}{\cong} \cl(GDA,L)(s_L \circ Gf, s_L \circ GDl')\\
&\stackrel{(D)}{\cong} \cl(GDA,L)(s_L \circ Gf, l)\\
\end{align*}
The bijection \textbf{(A)} follows from the adjunction \eqref{EQ_hL_yA_yHL_JD} above. \textbf{(B)} is given by the action on morphisms of the following functor:
\begin{equation}\label{EQ_sLcircGD}
s_L \circ GJ_D(-): \ck(A,HL) \to \cl(GDA,L).
\end{equation}
This functor is (by assumption) an equivalence -- in particular it is fully faithful. \textbf{(C)} is given by the pre-composition with $\beth^{-1}_f$ and \textbf{(D)} is given by the post-composition with $\lambda$. To conclude that $(Dl', \lambda)$ is a $G$-lift, it has to be shown that the composite bijection is given by the assignment $\alpha \mapsto \lambda \circ s_LG\alpha$. Equivalently, the composite of the first three bijections is the assignment $\alpha \mapsto s_LG\alpha$. We prove this fact in the appendix as Lemma \ref{THM_lemma_big_l_adj_thm}.

%339
\medskip

The pair $(Dl', \lambda)$ is thus a $G$-lift. Because the functor \eqref{EQ_sLcircGD} is essentially surjective, such a $G$-lift is guaranteed to always exist. Make now a choice of a lift for every $l: GDA \to L$ and denote it by $(Dl^\mbbl, \mbbl)$. To prove that our choice is coherently closed for $G$-lifts, the following 2-cell needs to be shown to be a $G$-lift of $l: GDA \to L$ along $s_L$:
% https://q.uiver.app/#q=WzAsNSxbMCwwLCJMIl0sWzIsMCwiR0RITCJdLFswLDIsIkdEQSJdLFsyLDIsIkdESEdEQSJdLFsyLDQsIkdEQSJdLFszLDEsIkdEKGwgc197R0RBfSleXFxtYmJsIiwxXSxbMSwwLCJzX0wiLDJdLFsyLDAsImwiXSxbMywyLCJzX3tHREF9IiwyXSxbNCwzLCJHRDFfQV5cXG1iYmwiLDFdLFs0LDIsIiIsMix7ImxldmVsIjoyLCJzdHlsZSI6eyJoZWFkIjp7Im5hbWUiOiJub25lIn19fV0sWzQsMSwiRyhEKGwgc197R0RBfSleXFxtYmJsIFxcY2lyYyBEMV9BXlxcbWJibCkiLDIseyJvZmZzZXQiOjUsImN1cnZlIjo1fV0sWzYsOCwiXFxtYmJsIiwyLHsic2hvcnRlbiI6eyJzb3VyY2UiOjIwLCJ0YXJnZXQiOjIwfSwiZWRnZV9hbGlnbm1lbnQiOnsic291cmNlIjpmYWxzZSwidGFyZ2V0IjpmYWxzZX19XSxbOCwxMCwiXFxtYmJsIiwwLHsic2hvcnRlbiI6eyJzb3VyY2UiOjIwLCJ0YXJnZXQiOjIwfSwiZWRnZV9hbGlnbm1lbnQiOnsic291cmNlIjpmYWxzZSwidGFyZ2V0IjpmYWxzZX19XSxbMywxMSwiXFxnYW1tYSIsMCx7ImxhYmVsX3Bvc2l0aW9uIjozMCwic2hvcnRlbiI6eyJ0YXJnZXQiOjIwfX1dXQ==
\[\begin{tikzcd}
	L && GDHL \\
	\\
	GDA && GDHGDA \\
	\\
	&& GDA
	\arrow["{GD(l s_{GDA})^\mbbl}"{description}, from=3-3, to=1-3]
	\arrow[""{name=0, anchor=center, inner sep=0}, "{s_L}"', from=1-3, to=1-1]
	\arrow[""{name=0p, anchor=center, inner sep=0}, phantom, from=1-3, to=1-1, start anchor=center, end anchor=center]
	\arrow["l", from=3-1, to=1-1]
	\arrow[""{name=1, anchor=center, inner sep=0}, "{s_{GDA}}"', from=3-3, to=3-1]
	\arrow[""{name=1p, anchor=center, inner sep=0}, phantom, from=3-3, to=3-1, start anchor=center, end anchor=center]
	\arrow[""{name=1p, anchor=center, inner sep=0}, phantom, from=3-3, to=3-1, start anchor=center, end anchor=center]
	\arrow["{GD1_A^\mbbl}"{description}, from=5-3, to=3-3]
	\arrow[""{name=2, anchor=center, inner sep=0}, Rightarrow, no head, from=5-3, to=3-1]
	\arrow[""{name=2p, anchor=center, inner sep=0}, phantom, from=5-3, to=3-1, start anchor=center, end anchor=center]
	\arrow[""{name=3, anchor=center, inner sep=0}, "{G(D(l s_{GDA})^\mbbl \circ D1_A^\mbbl)}"', shift right=8, curve={height=30pt}, from=5-3, to=1-3]
	\arrow["\mbbl"', shorten <=9pt, shorten >=9pt, Rightarrow, from=0p, to=1p]
	\arrow["\mbbl", shorten <=4pt, shorten >=4pt, Rightarrow, from=1p, to=2p]
	\arrow["\gamma"{pos=0.3}, shorten >=4pt, Rightarrow, from=3-3, to=3]
\end{tikzcd}\]
But this follows from the what we have shown at the beginning since this composite 2-cell is invertible and the 1-cell component of the proposed $G$-lift is (isomorphic to) $Dh$ for a 1-cell $h$ in $\ck$. For the same reasons, the unit and composition axioms in the assumptions of Theorem \ref{THM_Bunge} are satisfied.

We thus have a right colax adjoint to $G: \ck_D \to \cl$, let us denote it by $R: \cl \to \ck_D$. Since the pseudonaturality square of the counit $s$ is a $G$-lift (this again follows from what we have proven at the beginning of the proof), for any 1-cell $l: L \to K$ there exists a unique invertible 2-cell $\delta_l: Rl \Rightarrow J_D Hl$ making the following diagrams equal:
% https://q.uiver.app/#q=WzAsOSxbMCwwLCJMIl0sWzIsMCwiR0RITCJdLFswLDIsIksiXSxbMiwyLCJHREhLIl0sWzQsMSwiPSJdLFs1LDAsIkwiXSxbNSwyLCJLIl0sWzcsMiwiR0RISyJdLFs3LDAsIkdESEwiXSxbMywxLCJHREhsIl0sWzEsMCwic19MIiwyXSxbMiwwLCJsIl0sWzMsMiwic197TH0iXSxbMSwzLCJHUmwiLDAseyJjdXJ2ZSI6LTV9XSxbNiw1LCJsIl0sWzcsNiwic19LIl0sWzgsNSwic19MIiwyXSxbNyw4LCJHUmwiLDJdLFsxMCwxMiwic19sIiwwLHsic2hvcnRlbiI6eyJzb3VyY2UiOjIwLCJ0YXJnZXQiOjIwfSwiZWRnZV9hbGlnbm1lbnQiOnsic291cmNlIjpmYWxzZSwidGFyZ2V0IjpmYWxzZX19XSxbMTMsOSwiR1xcZGVsdGFfbCIsMix7InNob3J0ZW4iOnsic291cmNlIjoyMCwidGFyZ2V0IjoyMH19XSxbMTYsMTUsIlxcbWJibCIsMCx7InNob3J0ZW4iOnsic291cmNlIjoyMCwidGFyZ2V0IjoyMH19XV0=
\[\begin{tikzcd}
	L && GDHL &&& L && GDHL \\
	&&&& {=} \\
	K && GDHK &&& K && GDHK
	\arrow[""{name=0, anchor=center, inner sep=0}, "GDHl", from=3-3, to=1-3]
	\arrow[""{name=1, anchor=center, inner sep=0}, "{s_L}"', from=1-3, to=1-1]
	\arrow[""{name=1p, anchor=center, inner sep=0}, phantom, from=1-3, to=1-1, start anchor=center, end anchor=center]
	\arrow["l", from=3-1, to=1-1]
	\arrow[""{name=2, anchor=center, inner sep=0}, "{s_{L}}", from=3-3, to=3-1]
	\arrow[""{name=2p, anchor=center, inner sep=0}, phantom, from=3-3, to=3-1, start anchor=center, end anchor=center]
	\arrow[""{name=3, anchor=center, inner sep=0}, "GRl", curve={height=-30pt}, from=1-3, to=3-3]
	\arrow["l", from=3-6, to=1-6]
	\arrow[""{name=4, anchor=center, inner sep=0}, "{s_K}", from=3-8, to=3-6]
	\arrow[""{name=5, anchor=center, inner sep=0}, "{s_L}"', from=1-8, to=1-6]
	\arrow["GRl"', from=3-8, to=1-8]
	\arrow["{s_l}", shorten <=9pt, shorten >=9pt, Rightarrow, from=1p, to=2p]
	\arrow["{G\delta_l}"', shorten <=6pt, shorten >=6pt, Rightarrow, from=3, to=0]
	\arrow["\mbbl", shorten <=9pt, shorten >=9pt, Rightarrow, from=5, to=4]
\end{tikzcd}\]
It is now routine to verify that this data gives an invertible icon $\delta: R \Rightarrow J_D H$ (which is an isomorphism in $\psd{\cl}{\ck_D}$), proving that the pseudofunctor $J_D H$ is right colax adjoint to $G: \ck_D \to \cl$ as well.
\end{proof}

\medskip

\noindent Our first application will be the following:
\begin{cor}
Given a left Kan pseudomonad $(D,y)$ on a 2-category $\ck$, the biadjunction between the base 2-category and the Kleisli 2-category induces a colax adjunction on the Kleisli 2-category:
% https://q.uiver.app/#q=WzAsNSxbMCwwLCJcXGNrIl0sWzIsMCwiXFxja19EIl0sWzMsMCwiXFxyaWdodHNxdWlnYXJyb3ciXSxbNCwwLCJcXGNrX0QiXSxbNiwwLCJcXGNrX0QiXSxbMCwxLCJKX0QiLDIseyJjdXJ2ZSI6NH1dLFsxLDAsIkZfRCIsMix7ImN1cnZlIjo0fV0sWzMsNCwiIiwyLHsiY3VydmUiOjQsImxldmVsIjoyLCJzdHlsZSI6eyJoZWFkIjp7Im5hbWUiOiJub25lIn19fV0sWzQsMywiSl9ERl9EIiwyLHsiY3VydmUiOjR9XSxbNSw2LCIiLDIseyJsZXZlbCI6MSwic3R5bGUiOnsibmFtZSI6ImFkanVuY3Rpb24ifX1dLFs3LDgsIlxcdHRvcCIsMSx7ImxldmVsIjoxLCJzdHlsZSI6eyJib2R5Ijp7Im5hbWUiOiJub25lIn0sImhlYWQiOnsibmFtZSI6Im5vbmUifX19XV0=
\[\begin{tikzcd}
	\ck && {\ck_D} & \rightsquigarrow & {\ck_D} && {\ck_D}
	\arrow[""{name=0, anchor=center, inner sep=0}, "{J_D}"', curve={height=24pt}, from=1-1, to=1-3]
	\arrow[""{name=1, anchor=center, inner sep=0}, "{F_D}"', curve={height=24pt}, from=1-3, to=1-1]
	\arrow[""{name=2, anchor=center, inner sep=0}, curve={height=24pt}, Rightarrow, no head, from=1-5, to=1-7]
	\arrow[""{name=3, anchor=center, inner sep=0}, "{J_DF_D}"', curve={height=24pt}, from=1-7, to=1-5]
	\arrow["\dashv"{anchor=center, rotate=90}, draw=none, from=0, to=1]
	\arrow["\ttop"{description}, draw=none, from=2, to=3]
\end{tikzcd}\]
\end{cor}

\noindent The following is a categorification of the fact that for an idempotent monad, the Kleisli and EM-categories are equivalent:

\begin{cor}\label{THM_lax_adjunkce_EM_Kleisli}
Given a left Kan pseudomonad $(D, y)$, the associated free-forgetful biadjunction induces a colax adjunction between the Kleisli 2-category and the 2-category of algebras:
% https://q.uiver.app/#q=WzAsNSxbMCwwLCJcXGNrIl0sWzIsMCwiXFx0ZXh0e1BzLX1EXFx0ZXh0ey1BbGd9Il0sWzMsMCwiXFxyaWdodHNxdWlnYXJyb3ciXSxbNCwwLCJcXGNrX0QiXSxbNiwwLCJcXHRleHR7UHMtfURcXHRleHR7LUFsZ30iXSxbMCwxLCJGXkQiLDIseyJjdXJ2ZSI6NH1dLFsxLDAsIlVeRCIsMix7ImN1cnZlIjo0fV0sWzMsNCwiIiwyLHsiY3VydmUiOjQsInN0eWxlIjp7InRhaWwiOnsibmFtZSI6Imhvb2siLCJzaWRlIjoidG9wIn19fV0sWzQsMywiSl9EXFxjaXJjIFVeRCIsMix7ImN1cnZlIjo0fV0sWzUsNiwiIiwyLHsibGV2ZWwiOjEsInN0eWxlIjp7Im5hbWUiOiJhZGp1bmN0aW9uIn19XSxbNyw4LCJcXHR0b3AiLDEseyJsZXZlbCI6MSwic3R5bGUiOnsiYm9keSI6eyJuYW1lIjoibm9uZSJ9LCJoZWFkIjp7Im5hbWUiOiJub25lIn19fV1d
\[\begin{tikzcd}
	\ck && {\text{Ps-}D\text{-Alg}} & \rightsquigarrow & {\ck_D} && {\text{Ps-}D\text{-Alg}}
	\arrow[""{name=0, anchor=center, inner sep=0}, "{F^D}"', curve={height=24pt}, from=1-1, to=1-3]
	\arrow[""{name=1, anchor=center, inner sep=0}, "{U^D}"', curve={height=24pt}, from=1-3, to=1-1]
	\arrow[""{name=2, anchor=center, inner sep=0}, curve={height=24pt}, hook, from=1-5, to=1-7]
	\arrow[""{name=3, anchor=center, inner sep=0}, "{J_D\circ U^D}"', curve={height=24pt}, from=1-7, to=1-5]
	\arrow["\dashv"{anchor=center, rotate=90}, draw=none, from=0, to=1]
	\arrow["\ttop"{description}, draw=none, from=2, to=3]
\end{tikzcd}\]
\end{cor}

\noindent The following is a change-of-base-style theorem:
\begin{cor}\label{THM_change_of_base_kleisli}
Let $D$ be a lax-idempotent pseudomonad on a 2-category $\ck$ and $T$ be a pseudomonad on a 2-category $\cl$. Assume that:

\begin{itemize}
\item there is a biadjunction between the base 2-categories:
% https://q.uiver.app/#q=WzAsMixbMCwwLCJcXGNrIl0sWzIsMCwiXFxjbCJdLFswLDEsIkwiLDIseyJjdXJ2ZSI6NH1dLFsxLDAsIlIiLDIseyJjdXJ2ZSI6NH1dLFsyLDMsIiIsMCx7ImxldmVsIjoxLCJzdHlsZSI6eyJuYW1lIjoiYWRqdW5jdGlvbiJ9fV1d
\[\begin{tikzcd}
	\ck && \cl
	\arrow[""{name=0, anchor=center, inner sep=0}, "L"', curve={height=24pt}, from=1-1, to=1-3]
	\arrow[""{name=1, anchor=center, inner sep=0}, "R"', curve={height=24pt}, from=1-3, to=1-1]
	\arrow["\dashv"{anchor=center, rotate=90}, draw=none, from=0, to=1]
\end{tikzcd}\]
\item the left biadjoint admits an extension to the Kleisli 2-categories:
% https://q.uiver.app/#q=WzAsNCxbMCwwLCJcXGNrIl0sWzIsMCwiXFxjayJdLFswLDEsIlxcY2tfRCJdLFsyLDEsIlxcY2xfVCJdLFsyLDMsIkxee1xcI30iLDJdLFswLDEsIkwiXSxbMCwyLCJKX0QiLDJdLFsxLDMsIkpfVCJdXQ==
\[\begin{tikzcd}
	\ck && \ck \\
	{\ck_D} && {\cl_T}
	\arrow["{L^{\#}}"', from=2-1, to=2-3]
	\arrow["L", from=1-1, to=1-3]
	\arrow["{J_D}"', from=1-1, to=2-1]
	\arrow["{J_T}", from=1-3, to=2-3]
\end{tikzcd}\]
\end{itemize}
Then there is an induced colax adjunction between the Kleisli 2-categories:
% https://q.uiver.app/#q=WzAsMixbMCwwLCJcXGNrX0QiXSxbMiwwLCJcXGNsX1QiXSxbMCwxLCJMXlxcIyIsMix7ImN1cnZlIjo0fV0sWzEsMCwiIiwyLHsiY3VydmUiOjR9XSxbMiwzLCJcXHR0b3AiLDEseyJsZXZlbCI6MSwic3R5bGUiOnsiYm9keSI6eyJuYW1lIjoibm9uZSJ9LCJoZWFkIjp7Im5hbWUiOiJub25lIn19fV1d
\[\begin{tikzcd}
	{\ck_D} && {\cl_T}
	\arrow[""{name=0, anchor=center, inner sep=0}, "{L^\#}"', curve={height=24pt}, from=1-1, to=1-3]
	\arrow[""{name=1, anchor=center, inner sep=0}, curve={height=24pt}, from=1-3, to=1-1]
	\arrow["\ttop"{description}, draw=none, from=0, to=1]
\end{tikzcd}\]
\end{cor}
\begin{proof}
Composing the Kleisli biadjunction with the $L \dashv R$ biadjunction we obtain the following the following biadjunction on which we can apply the theorem:
% https://q.uiver.app/#q=WzAsNCxbMCwwLCJcXGNrIl0sWzIsMCwiXFxjbCJdLFs0LDAsIlxcY2xfVCJdLFsyLDIsIlxcY2tfRCJdLFswLDEsIkwiLDIseyJjdXJ2ZSI6NH1dLFsxLDAsIlIiLDIseyJjdXJ2ZSI6NH1dLFsxLDIsIlQiLDIseyJjdXJ2ZSI6NH1dLFsyLDEsIkZfVCIsMix7ImN1cnZlIjo0fV0sWzAsMywiSl9EIiwyLHsiY3VydmUiOjN9XSxbMywyLCJMXlxcIyIsMix7ImN1cnZlIjozfV0sWzYsNywiIiwyLHsibGV2ZWwiOjEsInN0eWxlIjp7Im5hbWUiOiJhZGp1bmN0aW9uIn19XSxbNCw1LCIiLDAseyJsZXZlbCI6MSwic3R5bGUiOnsibmFtZSI6ImFkanVuY3Rpb24ifX1dXQ==
\[\begin{tikzcd}
	\ck && \cl && {\cl_T} \\
	\\
	&& {\ck_D}
	\arrow[""{name=0, anchor=center, inner sep=0}, "L"', curve={height=24pt}, from=1-1, to=1-3]
	\arrow[""{name=1, anchor=center, inner sep=0}, "R"', curve={height=24pt}, from=1-3, to=1-1]
	\arrow[""{name=2, anchor=center, inner sep=0}, "T"', curve={height=24pt}, from=1-3, to=1-5]
	\arrow[""{name=3, anchor=center, inner sep=0}, "{F_T}"', curve={height=24pt}, from=1-5, to=1-3]
	\arrow["{J_D}"', curve={height=18pt}, from=1-1, to=3-3]
	\arrow["{L^\#}"', curve={height=18pt}, from=3-3, to=1-5]
	\arrow["\dashv"{anchor=center, rotate=90}, draw=none, from=2, to=3]
	\arrow["\dashv"{anchor=center, rotate=90}, draw=none, from=0, to=1]
\end{tikzcd}\]
\end{proof}

\begin{pr}
Consider a 2-category $\ck$ with comma objects and pullbacks and take for $D$ the fibration 2-monad on $\ck/C$ and $T$ the fibration 2-monad on $\ck/D$. For any 1-cell $k: C \to D$ gives a 2-functor $k_*: \ck/C \to \ck/D$ with a right 2-adjoint $k^*$ given by pulling back. The 2-functor $k_*$ clearly extends to the colax slice 2-categories, hence giving rise to a \textbf{lax} adjunction between the colax slices:
% https://q.uiver.app/#q=WzAsMixbMCwwLCJcXGNrLy9DIl0sWzIsMCwiXFxjay8vRCJdLFswLDEsImtfKiIsMix7ImN1cnZlIjo0fV0sWzEsMCwiRCBcXGNpcmMgRl9EIFxcY2lyYyBrXioiLDIseyJjdXJ2ZSI6NH1dLFsyLDMsIlxcdHRvcCIsMSx7ImxldmVsIjoxLCJzdHlsZSI6eyJib2R5Ijp7Im5hbWUiOiJub25lIn0sImhlYWQiOnsibmFtZSI6Im5vbmUifX19XV0=
\[\begin{tikzcd}
	{\ck//C} && {\ck//D}
	\arrow[""{name=0, anchor=center, inner sep=0}, "{k_*}"', curve={height=24pt}, from=1-1, to=1-3]
	\arrow[""{name=1, anchor=center, inner sep=0}, "{D \circ F_D \circ k^*}"', curve={height=24pt}, from=1-3, to=1-1]
	\arrow["\ttop"{description}, draw=none, from=0, to=1]
\end{tikzcd}\]
\end{pr}

\begin{pr}
In the next section (Corollary \ref{THM_COR_change_of_base_SalgTalg})  we will see how, when given a morphism of 2-monads $\theta: S \to T$, this gives rise to a colax adjunction between $\talgl$ and $S\text{-Alg}_l$.
\end{pr}
%316

\begin{pozn}[Left Kan 2-monads]\label{POZN_leftkan2monads_big_l_adj}

Assume that $(D,y)$ is a left Kan 2-monad and that we have the same starting biadjunction as in Theorem \ref{THM_BIG_lax_adj_thm}, except now the modifications $\sigma, \tau$ are the identities and the counit $s$ is 2-natural. Going through the proof, note that $s_L \circ GJ_D(-): \ck(A,HL) \to \cl(GDA,L)$ is an isomorphism of categories: for each $l: GDA \to L$ there is a \textbf{unique} $l^\mbbl: A \to HL$ such that $s_L \circ GDl^\mbbl = l$. Because $J_D: \ck \to \ck_D$ is now a 2-functor and because of the uniqueness of each $l^\mbbl$, the collection $s_L: GDHL \to L$ is  strictly closed for $G$-lifts (dual of Definition \ref{DEFI_strictly_closed_for_Uexts}). By Remark \ref{POZN_strict_BungesTHM} we obtain a colax adjunction for which the modifications are the identities and the counit $s$ is 
2-natural.
\end{pozn}

\subsection{Coreflector-limits}\label{SUBS_2_coreflector_limits}

\begin{defi}\label{DEFI_rali_colimit}
Let $\ck$ be a 2-category and $F: \cj \to \ck$, $W: \cj \to \cat$ 2-functors. A \textit{coreflector-limit} of $F$ weighted by $W$  is given by an object $L \in \ck$ and a 2-natural transformation $\lambda: W \Rightarrow \ck(L,F-)$ with the property that for every $A \in \ck$, the \textit{canonical comparison} functor
\begin{align*}
\kappa_A : \ck(A,L) &\to [\cj, \cat](W, \ck(A, F?)),\\
\kappa_A : (\theta: A \to L) &\mapsto (\ck(\theta,F?) \circ \lambda),
\end{align*}
is a coreflector in $\cat$. \textit{Coreflector-colimits} in $\ck$ are defined as coreflector-limits in $\ck^{op}$. Analogously, we say that $\lambda$ is an \textit{\textbf{X}-limit} if $\kappa_A$ is in class \textbf{X} of functors for every $A$.
%309
\end{defi}

\begin{pozn}\label{poznnaturalitalcolims}
Because the maps $\kappa_A: \ck(A,L) \to [\cj, \cat](W, \ck(A,F?))$ together form a 2-natural transformation $\kappa: \ck(-,L) \Rightarrow [\cj, \cat](W,\ck(-,F?))$, by \cite[Theorem 1]{twoconstr} the above definition is equivalent to requiring that $\kappa$ is a coreflector in the 2-category $\text{Colax}[\ck, \cat]$ (of 2-functors $\ck \to \cat$, colax natural transformations and modifications).
%291
\end{pozn}

\begin{pozn}[Enriched weakness]
The notion of a coreflector-colimit is a special case of an \textit{enriched weak colimit} in the sense of \cite[Section 4]{enrichedweakness}. The enriching category $\cv$ is equal to $\cat$ with the class $\ce$ being functors that are coreflectors. In \cite{enrichedweakness}, the authors studied coreflector-colimits for which $\kappa$ is actually a \textit{retract equivalence} -- meaning that the unit of the adjunction is the identity and the counit is invertible.
\end{pozn}

\begin{pozn}
Conical (left-adjoint)-limits of 2-functors have first been introduced \cite[I,7.9.1]{formalcat} under the name \textit{quasi-limits}\footnote{In fact, the definition in \cite{formalcat} is stronger than ours because it requires the existence of a 2-functor picking the limits that is right lax adjoint to the constant embedding 2-functor $\ck \to \ck^\cj$.}.
\end{pozn}

\begin{pozn}[Ordinary weakness]\label{POZN_ordinary_weakness}
Notice that every rali- and lali-limit cone $\lambda$ is a \textit{weak limit} of $F$ weighted by $W$. What this means is that given a different cone $\mu: W \Rightarrow \ck(A,F-)$, the left adjoint $L_A$ to $\kappa_A$ gives a comparison map $L_A \mu: A \to L$ such that:
\[
\mu = \ck(L\mu,F-) \circ \lambda.
\]
This is like the definition of a 2-limit except that there is no uniqueness requirement. It is not the case that every weak limit is a rali-limit. This is because if the 2-category $\ck$ is locally discrete, the notion of rali-limit coincides with an ordinary limit and not a weak limit.
\end{pozn}

%\medskip

\begin{pr}
An object $I$ in a 2-category $\ck$ is lali-initial if the unique functor into the terminal category admits a left adjoint for every object $A \in \ck$:
% https://q.uiver.app/#q=WzAsMixbMCwwLCJcXGNrKEksQSkiXSxbMiwwLCIqIl0sWzAsMSwiISIsMix7ImN1cnZlIjo0fV0sWzEsMCwiIiwyLHsiY3VydmUiOjR9XSxbMywyLCIiLDAseyJsZXZlbCI6MSwic3R5bGUiOnsibmFtZSI6ImFkanVuY3Rpb24ifX1dXQ==
\[\begin{tikzcd}
	{\ck(I,A)} && {*}
	\arrow[""{name=0, anchor=center, inner sep=0}, "{!}"', curve={height=24pt}, from=1-1, to=1-3]
	\arrow[""{name=1, anchor=center, inner sep=0}, curve={height=24pt}, from=1-3, to=1-1]
	\arrow["\dashv"{anchor=center, rotate=-90}, draw=none, from=1, to=0]
\end{tikzcd}\]
Clearly, this happens if and only if the hom-category $\ck(I,A)$ has an initial object for every $A \in \ck$. For a particular example, consider the 2-category $\text{MonCat}_l$ of monoidal categories and lax monoidal functors. The terminal monoidal category $*$ is lali-initial because for every monoidal category $\mbba$ we have an isomorphism between $\text{MonCat}_l(*,\mbba)$ and the category $\text{Mon}(\mbba)$ of monoids in $\mbba$, and this category has an initial object given by the monoidal unit of $\mbba$.
\end{pr}

\begin{pr}
Consider a 2-category $\ck$ with a zero object $0 \in \ck$ and with a further property that the zero morphism $0_{A,B}: A \to B$ is the initial object in $\ck(A,B)$ for every pair of objects $A,B$. Then $0$ is a conical lali-colimit of any 2-functor $F: \cj \to \ck$. This is because in the definition of a lali-colimit:
% https://q.uiver.app/#q=WzAsMixbMCwwLCJcXGNrKDAsQSkiXSxbMiwwLCJcXHRleHR7Q29jb25lfShBLEYpIl0sWzAsMSwiIiwwLHsiY3VydmUiOjR9XSxbMSwwLCIiLDAseyJjdXJ2ZSI6NH1dLFsyLDMsIiIsMCx7ImxldmVsIjoxLCJzdHlsZSI6eyJuYW1lIjoiYWRqdW5jdGlvbiJ9fV1d
\[\begin{tikzcd}
	{\ck(0,A)} && {\text{Cocone}(A,F)}
	\arrow[""{name=0, anchor=center, inner sep=0}, curve={height=24pt}, from=1-1, to=1-3]
	\arrow[""{name=1, anchor=center, inner sep=0}, curve={height=24pt}, from=1-3, to=1-1]
	\arrow["\dashv"{anchor=center, rotate=90}, draw=none, from=0, to=1]
\end{tikzcd}\]
we have $\ck(0,A) \cong *$, and so the question becomes whether the category of cocones of $F$ with apex $A$ has an initial object. But it does and it is given by the cocone whose components are the zero morphisms.
This for instance applies to the poset-enriched categories $\rel$ of sets and relations and $\parf$ of sets and partial functions.
\end{pr}

\begin{pr}\label{PR_colax_slice_2cat_weaklimit_1}
Let $\ck$ be a 2-category with pullbacks and comma objects and consider the slice 2-category $\ck / C$ and the colax slice 2-category $\ck // C$ from Example \ref{PR_colax_slice_2cat}. It is known that the product of two objects $f_1, f_2$ in $\ck / C$ is the diagonal in the pullback square of $f_1, f_2$ in $\ck$:
% https://q.uiver.app/#q=WzAsNCxbMSwyLCJDIl0sWzAsMSwiQV8xIl0sWzIsMSwiQV8yIl0sWzEsMCwiTCJdLFsxLDAsImZfMSIsMl0sWzIsMCwiZl8yIl0sWzMsMV0sWzMsMl0sWzMsMCwiIiwxLHsic3R5bGUiOnsibmFtZSI6ImNvcm5lciJ9fV1d
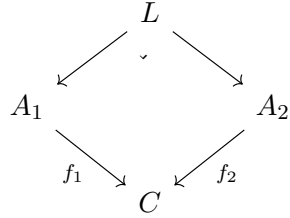
\[\begin{tikzcd}
	& L \\
	{A_1} && {A_2} \\
	& C
	\arrow["{f_1}"', from=2-1, to=3-2]
	\arrow["{f_2}", from=2-3, to=3-2]
	\arrow[from=1-2, to=2-1]
	\arrow[from=1-2, to=2-3]
	\arrow["\lrcorner"{anchor=center, pos=0.125, rotate=-45}, draw=none, from=1-2, to=3-2]
\end{tikzcd}\]

One guess would be that this becomes a weak product in $\ck//C$ after applying the inclusion 2-functor $\ck/C \xhookrightarrow{} \ck//C$, but that would be a wrong guess. To calculate the weak product of $f_1, f_2$ in $\ck // C$, we first calculate the product of the comma object projections for $f_1, f_2$ (using the notation from Example \ref{PR_colax_slice_2cat}) in $\ck/C$ as pictured below left:
% https://q.uiver.app/#q=WzAsOSxbMSwyLCJDIl0sWzAsMSwiUGZfMSJdLFsyLDEsIlBmXzIiXSxbMSwwLCJMIl0sWzUsMCwiUGZfaSJdLFs3LDAsIkFfMSJdLFs1LDIsIkMiXSxbNywyLCJDIl0sWzQsMCwiTCJdLFsxLDAsIlxccGlfe2ZfMX0iLDJdLFsyLDAsIlxccGlfe2ZfMn0iXSxbMywxLCJcXHRhdV8xIiwyXSxbMywyLCJcXHRhdV8yIl0sWzMsMCwiIiwxLHsic3R5bGUiOnsibmFtZSI6ImNvcm5lciJ9fV0sWzQsNSwiXFxyaG9faSJdLFs1LDcsImZfaSJdLFs2LDcsIiIsMix7ImxldmVsIjoyLCJzdHlsZSI6eyJoZWFkIjp7Im5hbWUiOiJub25lIn19fV0sWzQsNiwiXFxwaV97Zl9pfSIsMV0sWzgsNCwiXFx0YXVfaSJdLFs4LDYsImwiLDJdLFsxNCwxNiwiXFxjaGlfaSIsMCx7InNob3J0ZW4iOnsic291cmNlIjoyMCwidGFyZ2V0IjoyMH19XV0=
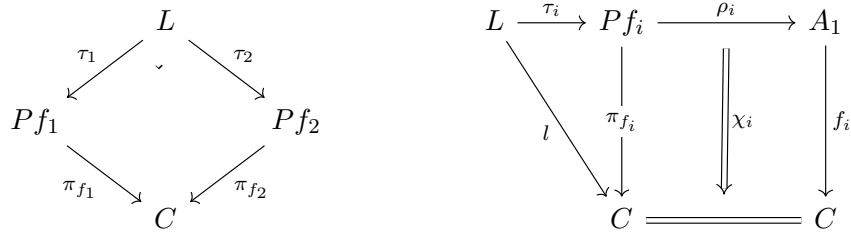
\[\begin{tikzcd}
	& L &&& L & {Pf_i} && {A_1} \\
	{Pf_1} && {Pf_2} \\
	& C &&&& C && C
	\arrow["{\tau_1}"', from=1-2, to=2-1]
	\arrow["{\tau_2}", from=1-2, to=2-3]
	\arrow["\lrcorner"{anchor=center, pos=0.125, rotate=-45}, draw=none, from=1-2, to=3-2]
	\arrow["{\tau_i}", from=1-5, to=1-6]
	\arrow["l"', from=1-5, to=3-6]
	\arrow[""{name=0, anchor=center, inner sep=0}, "{\rho_i}", from=1-6, to=1-8]
	\arrow["{\pi_{f_i}}"{description}, from=1-6, to=3-6]
	\arrow["{f_i}", from=1-8, to=3-8]
	\arrow["{\pi_{f_1}}"', from=2-1, to=3-2]
	\arrow["{\pi_{f_2}}", from=2-3, to=3-2]
	\arrow[""{name=1, anchor=center, inner sep=0}, Rightarrow, no head, from=3-6, to=3-8]
	\arrow["{\chi_i}", shorten <=9pt, shorten >=9pt, Rightarrow, from=0, to=1]
\end{tikzcd}\]
Denote $l := \pi_{f_1} \circ \tau_1$. The claim now is that the object $l \in \ck//C$ together with the colax triangles $(\rho_i \tau_i, \chi_i \tau_i): l \to f_i$ (here $\chi_i$ is the comma object square pictured above right) is the lali-product of $f_1, f_2$ in $\ck // C$. We will establish why this is the case after we prove the main theorem in this section.
\end{pr}

\begin{pr}
\textit{Bilimits} are a special case of coreflector-limits where $\kappa_A$ is an equivalence for every $A \in \ck$. In case $\kappa_A$ is an isomorphism for every $A \in \ck$, this is the notion of an ordinary \textit{2-limit}.
\end{pr}

\begin{pozn}[Uniqueness of rali-limits]\label{POZN_rali_limits_not_unique_up_to_equiv}
Rali-limits are \textbf{not} unique up to an equivalence. It is not even the case that given two rali-limit objects $L_1, L_2$, there exists a left (or right) adjoint 1-cell $L_1 \to L_2$. For a particular example, consider again the poset-enriched category $\parf$ of sets and partial functions. 
The empty set $\emptyset$ is the terminal object in $\parf$, in particular it is rali-terminal. The singleton set $*$ is rali-terminal: the ordered set $\parf(A,*)$ has a maximal element given by the unique total function $!: A \to *$. They are also ``normalized” in the sense that $1_\emptyset$ and $1_*$ are terminal objects in the hom ordered sets they belong to.

Since this 2-category is poset-enriched, equivalent objects would be isomorphic, and an isomorphism in $\parf$ has to be a total function. Thus $*$ and $\emptyset$ can not be equivalent in $\parf$. Moreover, it can be seen that there is no left adjoint 1-cell $* \to \emptyset$. 

We may now also give an example of two non-equivalent left colax adjoints that we have promised in Remark \ref{POZN_lax_adjs_not_unique_up_to_equiv}. It can be seen that given a 2-category $\ck$, the 2-functor $* \to \ck$ picking an object $L$ is a colax left adjoint to the unique 2-functor $\ck \to *$ if and only if $L$ is rali-initial. In this colax adjunction, the modification $\Psi$ is invertible if and only if the 1-cell $1_L$ is the initial object of $\ck(L,L)$. Based on above paragraphs, the unique 2-functor $\parf^{op} \to *$ has two left colax adjoints that are not equivalent.
\end{pozn}

\begin{theo}\label{THM_Kleisli2cat_is_corefl_complete}
Assume $\ck$ is a 2-category that admits weighted $J$-indexed bilimits and assume $D$ is a lax-idempotent pseudomonad on $\ck$. Then, the Kleisli 2-category $\ck_D$ admits weighted $J$-indexed coreflector-limits.
\end{theo}
\begin{proof}
Let $G: \cp \to \ck_D$ be a 2-functor, write $\widetilde{G}$ for the 2-functor $DA \mapsto \ck_D(DA,G?)$. Denoting again by $U_D$ the canonical pseudofunctor $\ck_D \to \ck$, there is a biadjunction where the right biadjoint sends a weight $W$ to the bilimit of $U_D G: \cp \to \ck$ weighted by $W$:
% https://q.uiver.app/#q=WzAsMyxbMCwwLCJcXGNrIl0sWzIsMCwiXFxja19EIl0sWzQsMCwiXFxwc2R7XFxjcH17XFxjYXR9XntvcH0iXSxbMiwwLCJcXHstLFVfREcgXFx9IiwyLHsiY3VydmUiOjV9XSxbMCwxLCJKX0QiLDIseyJzdHlsZSI6eyJ0YWlsIjp7Im5hbWUiOiJob29rIiwic2lkZSI6InRvcCJ9fX1dLFsxLDIsIlxcd2lkZXRpbGRle0d9IiwyXSxbMSwzLCJcXHRvcCIsMSx7ImxldmVsIjoxLCJzdHlsZSI6eyJib2R5Ijp7Im5hbWUiOiJub25lIn0sImhlYWQiOnsibmFtZSI6Im5vbmUifX19XV0=
\[\begin{tikzcd}
	\ck && {\ck_D} && {\psd{\cp}{\cat}^{op}}
	\arrow["{J_D}"', hook, from=1-1, to=1-3]
	\arrow["{\widetilde{G}}"', from=1-3, to=1-5]
	\arrow[""{name=0, anchor=center, inner sep=0}, "{\{-,U_DG \}}"', curve={height=30pt}, from=1-5, to=1-1]
	\arrow["\top"{description}, draw=none, from=1-3, to=0]
\end{tikzcd}\]
This is because of the following equivalences that are pseudo-natural with respect to $W \in [\cp, \cat]$ and $A \in \ck$:
\begin{align*}
\ck(A, \{ W, U_D G \}) &\simeq \text{Psd}[\cp,\cat](W, \ck(A,U_D G-))\\
&\simeq \text{Psd}[\cp,\cat](W, \ck_D(DA,G-))\\
&= \text{Psd}[\cp,\cat]^{op}(\ck_D(DA,G-),W).
\end{align*}

Transferring the identity across those equivalences, we see that the \textbf{counit} of the biadjunction is the composite pseudonatural transformation:
% https://q.uiver.app/#q=WzAsMyxbMCwwLCJXIl0sWzIsMCwiXFxjayhcXHsgVywgVV9ERyBcXH0sIFVfRCBHLSkiXSxbNCwwLCJcXGNrX0QoRCBcXHsgVywgVV9EIEcgXFx9LCBHLSApLCJdLFswLDEsIlxcbGFtYmRhX1ciLDAseyJsZXZlbCI6Mn1dLFsxLDIsIihwX3tHLX0pXyogXFxjaXJjIEpfRCIsMCx7ImxldmVsIjoyfV1d
\[\begin{tikzcd}
	W && {\ck(\{ W, U_DG \}, U_D G-)} && {\ck_D(D \{ W, U_D G \}, G- ),}
	\arrow["{\lambda_W}", Rightarrow, from=1-1, to=1-3]
	\arrow["{(p_{G-})_* \circ J_D}", Rightarrow, from=1-3, to=1-5]
\end{tikzcd}\]

where $\lambda_W$ is the $W$-weighted bilimit cone for $U_DG : \cp \to \ck$. By Theorem \ref{THM_BIG_lax_adj_thm} this induces a colax adjunction with the same counit:
% https://q.uiver.app/#q=WzAsMixbMCwwLCJcXGNrX0QiXSxbMiwwLCJcXHRleHR7UHNkfVtcXGNwLFxcY2F0XV57b3B9Il0sWzAsMSwiXFx3aWRldGlsZGV7R30iLDJdLFsxLDAsIkpfRFxcey0sVV9ERyBcXH0iLDIseyJjdXJ2ZSI6NH1dLFsyLDMsIlxcdHRvcCIsMSx7ImxldmVsIjoxLCJlZGdlX2FsaWdubWVudCI6eyJzb3VyY2UiOmZhbHNlLCJ0YXJnZXQiOmZhbHNlfSwic3R5bGUiOnsiYm9keSI6eyJuYW1lIjoibm9uZSJ9LCJoZWFkIjp7Im5hbWUiOiJub25lIn19fV1d
\[\begin{tikzcd}
	{\ck_D} && {\text{Psd}[\cp,\cat]^{op}}
	\arrow[""{name=0, anchor=center, inner sep=0}, "{\widetilde{G}}"', from=1-1, to=1-3]
	\arrow[""{name=0p, anchor=center, inner sep=0}, phantom, from=1-1, to=1-3, start anchor=center, end anchor=center]
	\arrow[""{name=1, anchor=center, inner sep=0}, "{J_D\{-,U_DG \}}"', curve={height=24pt}, from=1-3, to=1-1]
	\arrow[""{name=1p, anchor=center, inner sep=0}, phantom, from=1-3, to=1-1, start anchor=center, end anchor=center, curve={height=24pt}]
	\arrow["\ttop"{description}, draw=none, from=0p, to=1p]
\end{tikzcd}\]

Notice that from the beginning of the proof of Theorem \ref{THM_Bunge_opak} it can be seen that there is an adjunction on hom categories described below whose unit is invertible:
% https://q.uiver.app/#q=WzAsMixbMCwwLCJcXGNrX0QoQixKX0QgXFx7IFcsIFVfREdcXH0pXFxwaGFudG9tey4uLi4uLi4uLi4uLi4uLn0iXSxbMiwwLCJcXHBoYW50b217Li4uLi4uLi4uLi4uLi4uLi59XFx0ZXh0e1BzZH1bXFxjcCxcXGNhdF0oVyxcXGNrX0QoQixHLSkpIl0sWzAsMSwiXFx0aGV0YSBcXG1hcHN0byBcXGNrX0QoXFx0aGV0YSxHLSkgXFxjaXJjIChwX3tHLX0pXyogXFxjaXJjIEpfRCBcXGNpcmMgXFxsYW1iZGFfVyIsMix7ImN1cnZlIjo0fV0sWzEsMCwiIiwwLHsiY3VydmUiOjR9XSxbMiwzLCIiLDAseyJsZXZlbCI6MSwic3R5bGUiOnsibmFtZSI6ImFkanVuY3Rpb24ifX1dXQ==
\[\begin{tikzcd}
	{\ck_D(B,J_D \{ W, U_DG\})\phantom{...............}} && {\phantom{.................}\text{Psd}[\cp,\cat](W,\ck_D(B,G-))}
	\arrow[""{name=0, anchor=center, inner sep=0}, "{\theta \mapsto \ck_D(\theta,G-) \circ (p_{G-})_* \circ J_D \circ \lambda_W}"', curve={height=24pt}, from=1-1, to=1-3]
	\arrow[""{name=1, anchor=center, inner sep=0}, curve={height=24pt}, from=1-3, to=1-1]
	\arrow["\dashv"{anchor=center, rotate=90}, draw=none, from=0, to=1]
\end{tikzcd}\]
This exhibits the pseudonatural transformation $(p_{G-})_* \circ J_D \circ \lambda_W$ as the coreflector-limit of $G$ weighted by $W$.
\end{proof}

\begin{pr}\label{PR_prof_coreflector_complete}
The bicategory $\prof$ of locally small categories and small profunctors is coreflector-complete. This is because by Example \ref{PR_presheafpsmonad}, it is a Kleisli bicategory for a left Kan pseudomonad on a complete 2-category $\CAT$.
\end{pr}

\begin{pr}\label{PR_colax_slice_2cat_weaklimit_2}
The proof of Theorem \ref{THM_Kleisli2cat_is_corefl_complete} gives a concrete way to compute limits. Consider the \textbf{colax}-idempotent pseudomonad from Example \ref{PR_colax_slice_2cat}. We can see that the process of computing \textbf{lali}-product of two objects $(f_1: A_1 \to C, f_2: A_2 \to C)$ in the colax-slice 2-category $\ck//C$ agrees with the process described in Example \ref{PR_colax_slice_2cat_weaklimit_1}.
\end{pr}

\begin{pozn}\label{POZN_strict_weaklimits}%XXSTRICTFORLEFTKAN2MONADS
Given a left Kan 2-monad $(D,y)$, going through the proof of Theorem \ref{THM_Kleisli2cat_is_corefl_complete} (and considering Remark \ref{POZN_leftkan2monads_big_l_adj}) we may now replace $\psd{\cp}{\cat}$ by $[\cp, \cat]$ and the result can be changed to the claim that $\ck_D$ admits $J$-indexed lali-limit whenever $\ck$ admits them as 2-limits.
\end{pozn}

We end the section with introducing the concept of preservation of weak limits:

\begin{defi}
We say that a pseudofunctor $H: \ck \to \cl$ \textit{preserves} \textbf{X}-limits (where \textbf{X} is any of the classes of morphisms in Definition \ref{POZN_duals_lali} for the case of $\ck = \cat$) if, whenever $\lambda: W \Rightarrow \ck(L,F-)$ exhibits $L$ as a \textbf{X}-limit of $F: \cj \to \ck$ weighted by $W: \cj \to \cat$, the composite pictured below  is an \textbf{X}-limit of $HF$ weighted by $W$:
% https://q.uiver.app/#q=WzAsMyxbMCwwLCJXIl0sWzIsMCwiXFxjayhMLCBGLSkiXSxbNCwwLCJcXGNsKEhMLEhGLSkiXSxbMCwxLCJcXGxhbWJkYSIsMCx7ImxldmVsIjoyfV0sWzEsMiwiSCIsMCx7ImxldmVsIjoyfV1d
\[\begin{tikzcd}
	W && {\ck(L, F-)} && {\cl(HL,HF-)}
	\arrow["\lambda", Rightarrow, from=1-1, to=1-3]
	\arrow["H", Rightarrow, from=1-3, to=1-5]
\end{tikzcd}\]
\end{defi}

\begin{pr}
In case $\ck, \cl$ admit comma objects, their preservation as rari-limits has been studied in \cite[Definition 7.1]{yoneda2toposes} where it has been called \textit{preservation of lax pullbacks up to a right adjoint section}. For instance, given a finitely complete 2-category $\ck$, the 2-functor $(-) \times Z$ has this property for any $Z \in \ck$ (see \cite[Example 7.3]{yoneda2toposes}). In Weber's later work \cite[6.1 THEOREM]{familial}, the class of  \textit{familial} functors have been shown to preserve comma objects as lari-limits.
\end{pr}

%277

\section{Applications to two-dimensional monad theory}\label{SEKCE_applications_to_twodim}

\begin{defi}\label{DEFI_property_L}
Let $T$ be a 2-monad on a 2-category $\ck$. We will say that it satisfies \textbf{Property L} if the inclusion $\talgs \xhookrightarrow{} \talgl$ admits a left 2-adjoint and the corresponding lax-morphism classifier 2-comonad $Q_l$ on $\talgs$ is lax-idempotent.
\end{defi}

By Theorem \ref{THM_adjoint_talgs_talgl}, Proposition \ref{THM_Ql_lax_idemp_Qp_ps_idemp}, a 2-monad $T$ on $\ck$ will have this property when $\ck$ admits oplax limits of an arrow and $\talgs$ is sufficiently cocomplete (admits lax codescent objects).

To apply the (appropriate dual of the) results developed in Section \ref{SUBS_3_big_lax_adj_thm} to the lax-idempotent 2-comonad $Q_l$, notice that (with the hint of the lists in Remark \ref{POZN_duals} and \ref{POZN_duals_lali}) this amounts to ``going” from $\ck$ to $\ck^{coop}$. For instance ``coreflection-inclusion” gets replaced by ``reflector”.%, \textit{lax-natural} by \textit{colax-natural} and so on.

\subsection{Lax flexibility}\label{SUBS_4_lax_flexibility}

For this section, recall the notions of \textit{semiflexible} and \textit{flexible} algebras for a 2-monad $T$ from \cite[Remark 4.5, page 23]{twodim}. By \cite[Proposition 1]{onsemiflexible}, a $T$-algebra $(A,a)$ is semi-flexible if and only if it admits the structure of a pseudo-$Q_p$-coalgebra. A \textit{pie} $T$-algebra was then defined to be a $T$-algebra that admits a strict $Q_p$-coalgebra structure. This motivates us to define:

\begin{defi}Let $T$ be a 2-monad on a 2-category $\ck$ that satisfies \textbf{Property L}. A $T$-algebra $(A,a)$ is said to be:
\begin{itemize}
\item \textit{lax-semiflexible} if it admits a pseudo-$Q_l$-coalgebra structure.
\item \textit{lax-flexible} if it admits a normal pseudo-$Q_l$-coalgebra structure.
\item \textit{lax-pie} if it admits a strict $Q_l$-coalgebra structure.
\end{itemize}
\end{defi}

\begin{pozn}\label{POZN_lax_semiflexible_is_semiflexible}
Every lax-\textbf{Y} $T$-algebra is \textbf{Y}, where $\textbf{Y} \in \{ \text{flexible, semiflexible, pie} \}$. This is because of the fact that by Proposition \ref{THM_mor_of_comonads_QlQp} there is an induced 2-functor from pseudo-$Q_l$-coalgebras to pseudo-$Q_p$-coalgebras that commutes with the 2-functors that forget the coalgebra structure (and thus keeps the $T$-algebra structure intact). 
\end{pozn}

\begin{pr}
In Corollary \ref{THM_free_forgetful_colax_adj} we will see that every free $T$-algebra is lax-flexible; this is a strengthening of the fact that every free $T$-algebra is flexible (\cite[Corollary 5.6]{twodim}).
\end{pr}

\begin{pr}
Fix a category $\cj$ and consider the 2-monad $T$ on $[ \ob \cj, \cat]$ whose algebras are weights (2-functors) $\cj \to \cat$. Weights that index \textit{lax limits} are precisely the wights that are cofree-$Q_l$-coalgebras, i.e. those of the form $Q_l W$ (see \cite[Chapter 5]{elemobs}). Since a lax limit is in general not a pseudo-limit \cite[Remark 5.5]{twodim}, not every pie algebra is lax-pie.
\end{pr}

Following Example \ref{PR_doctrinal_adj} and Remark \ref{POZN_strict_charakterizace_D_algeber}, the application of (the dual of) Proposition \ref{THM_characterization_D-algs_D-corefl-incl} to the lax-idempotent 2-comonad $Q_l$ provides a lax version of \cite[Theorem 20 a)]{onsemiflexible}. It reads as:

\begin{theo}\label{THM_charakterizace_lax_semiflexible_algs}
Let $T$ be a 2-monad on a 2-category $\ck$ satisfying \textbf{Property L} and denote by $U: \talgs \to \ck$ the forgetful 2-functor. A $T$-algebra is lax-semiflexible, semiflexible, if and only if, respectively:
\begin{itemize}
%\item $(A,a)$ is lax-semiflexible,
\item $\talgs((A,a),-): \talgs \to \cat$ sends $U$-reflectors to reflectors in $\cat$.
\item $\talgs((A,a),-): \talgs \to \cat$ sends $U$-lalis to lalis in $\cat$.
\end{itemize}
\end{theo}

\begin{pozn}
In a future work we will study lax-pie $T$-algebras for a 2-monad $T$. Using a comonadicity theorem, it can be shown that when $T$ is a 2-monad of form $\cat(T')$ for a cartesian monad $T'$ on a category $\ce$ with pullbacks, lax-pie $T$-algebras are equivalent to $T'$-multicategories.
\end{pozn}

%317

\subsection{Colax adjunctions and lali-cocompleteness of lax morphisms}\label{SUBS_4_on_various_colax_adjs}

Considering Remark \ref{POZN_leftkan2monads_big_l_adj}, the application of Theorem \ref{THM_bigadjthm_HL_is_D-algebra} and Theorem \ref{THM_BIG_lax_adj_thm} for the 2-comonad $Q_l$ reads as:

\begin{theo}\label{THM_big_colax_adj_THM}
Let $T$ be a 2-monad satisfying \textbf{Property L}. Any 2-adjunction below left induces a colax adjunction below right:
% https://q.uiver.app/#q=WzAsNixbMCwwLCJcXHRhbGdzIl0sWzIsMCwiXFx0YWxnbCJdLFs0LDAsIlxcY2wiXSxbNSwwLCJcXHJpZ2h0c3F1aWdhcnJvdyJdLFs2LDAsIlxcdGFsZ2wiXSxbOCwwLCJcXGNsIl0sWzAsMSwiSiIsMl0sWzEsMiwiRyIsMl0sWzIsMCwiSCIsMix7ImN1cnZlIjo1fV0sWzQsNSwiSkgiLDJdLFs1LDQsIkciLDIseyJjdXJ2ZSI6NH1dLFs4LDEsIiIsMix7ImxldmVsIjoxLCJzdHlsZSI6eyJuYW1lIjoiYWRqdW5jdGlvbiJ9fV0sWzksMTAsIlxcdGJvdCIsMSx7InNob3J0ZW4iOnsic291cmNlIjoyMCwidGFyZ2V0IjoyMH0sInN0eWxlIjp7ImJvZHkiOnsibmFtZSI6Im5vbmUifSwiaGVhZCI6eyJuYW1lIjoibm9uZSJ9fX1dXQ==
\[\begin{tikzcd}
	\talgs && \talgl && \cl & \rightsquigarrow & \talgl && \cl
	\arrow["J"', from=1-1, to=1-3]
	\arrow["G"', from=1-3, to=1-5]
	\arrow[""{name=0, anchor=center, inner sep=0}, "H"', curve={height=30pt}, from=1-5, to=1-1]
	\arrow[""{name=1, anchor=center, inner sep=0}, "JH"', from=1-7, to=1-9]
	\arrow[""{name=2, anchor=center, inner sep=0}, "G"', curve={height=24pt}, from=1-9, to=1-7]
	\arrow["\dashv"{anchor=center, rotate=-90}, draw=none, from=0, to=1-3]
	\arrow["\tbot"{description}, draw=none, from=1, to=2]
\end{tikzcd}\]
Moreover, for every $L \in \cl$, the $T$-algebra $HL$ is lax-flexible.
\end{theo}

\begin{cor}\label{THM_free_forgetful_colax_adj}
The free-forgetful adjunction for a 2-monad $T$ on a 2-category $\ck$ satisfying \textbf{Property L} induces a colax adjunction between $\talgl$ and $\ck$. In particular, every free $T$-algebra is lax-flexible.
% https://q.uiver.app/#q=WzAsNixbMCwwLCJcXHRhbGdzIl0sWzQsMCwiXFxjayJdLFs1LDAsIlxccmlnaHRzcXVpZ2Fycm93Il0sWzYsMCwiXFx0YWxnbCJdLFs4LDAsIlxcY2siXSxbMiwwLCJcXHRhbGdsIl0sWzEsMCwiRl5UIiwyLHsiY3VydmUiOjV9XSxbMyw0LCJVIiwyLHsiY3VydmUiOjR9XSxbNCwzLCJKIEZeVCIsMix7ImN1cnZlIjo0fV0sWzAsNSwiSiIsMl0sWzUsMSwiVSIsMl0sWzcsOCwiXFx0Ym90IiwxLHsibGV2ZWwiOjEsInN0eWxlIjp7ImJvZHkiOnsibmFtZSI6Im5vbmUifSwiaGVhZCI6eyJuYW1lIjoibm9uZSJ9fX1dLFs2LDUsIiIsMCx7ImxldmVsIjoxLCJzdHlsZSI6eyJuYW1lIjoiYWRqdW5jdGlvbiJ9fV1d
\[\begin{tikzcd}
	\talgs && \talgl && \ck & \rightsquigarrow & \talgl && \ck
	\arrow[""{name=0, anchor=center, inner sep=0}, "{F^T}"', curve={height=30pt}, from=1-5, to=1-1]
	\arrow[""{name=1, anchor=center, inner sep=0}, "U"', curve={height=24pt}, from=1-7, to=1-9]
	\arrow[""{name=2, anchor=center, inner sep=0}, "{J F^T}"', curve={height=24pt}, from=1-9, to=1-7]
	\arrow["J"', from=1-1, to=1-3]
	\arrow["U"', from=1-3, to=1-5]
	\arrow["\tbot"{description}, draw=none, from=1, to=2]
	\arrow["\dashv"{anchor=center, rotate=-90}, draw=none, from=0, to=1-3]
\end{tikzcd}\]
\end{cor}

\begin{cor}\label{THM_COR_change_of_base_SalgTalg}
Let $T,S$ be two 2-monads on a 2-category $\ck$ satisfying \textbf{Property L} and let $\theta: S \to T$ be a strict monad morphism. Assume that the induced 2-functor $\theta^*: \talgs \to S\text{-Alg}_s$ admits a left 2-adjoint $\theta_*$ (this is the case when $\ck$ is complete and cocomplete and $T$ is finitary, see \cite[Theorem 3.9]{twodim}). Then there is an induced colax adjunction between $\talgl$ and $\salgl$:
% https://q.uiver.app/#q=WzAsNyxbMCwwLCJcXHRhbGdzIl0sWzIsMCwiU1xcdGV4dHstQWxnfV9zIl0sWzQsMCwiU1xcdGV4dHstQWxnfV9sIl0sWzIsMiwiXFx0YWxnbCJdLFs1LDAsIlxccmlnaHRzcXVpZ2Fycm93Il0sWzYsMCwiXFx0YWxnbCJdLFs4LDAsIlNcXHRleHR7LUFsZ31fbCJdLFswLDEsIlxcdGhldGFeKiIsMix7ImN1cnZlIjo0fV0sWzEsMCwiXFx0aGV0YV8qIiwyLHsiY3VydmUiOjR9XSxbMSwyLCJKIiwyLHsiY3VydmUiOjR9XSxbMiwxLCIoLSknIiwyLHsiY3VydmUiOjR9XSxbMCwzLCJKIiwyLHsiY3VydmUiOjN9XSxbMywyLCJcXHRoZXRhXioiLDIseyJjdXJ2ZSI6M31dLFs1LDYsIlxcdGhldGFeKiIsMix7ImN1cnZlIjo0fV0sWzYsNSwiIiwyLHsiY3VydmUiOjR9XSxbOSwxMCwiIiwyLHsibGV2ZWwiOjEsInN0eWxlIjp7Im5hbWUiOiJhZGp1bmN0aW9uIn19XSxbNyw4LCIiLDAseyJsZXZlbCI6MSwic3R5bGUiOnsibmFtZSI6ImFkanVuY3Rpb24ifX1dLFsxMywxNCwiXFx0dG9wIiwxLHsic2hvcnRlbiI6eyJzb3VyY2UiOjIwLCJ0YXJnZXQiOjIwfSwic3R5bGUiOnsiYm9keSI6eyJuYW1lIjoibm9uZSJ9LCJoZWFkIjp7Im5hbWUiOiJub25lIn19fV1d
\[\begin{tikzcd}[column sep=scriptsize]
	\talgs && {S\text{-Alg}_s} && {S\text{-Alg}_l} & \rightsquigarrow & \talgl && {S\text{-Alg}_l} \\
	\\
	&& \talgl
	\arrow[""{name=0, anchor=center, inner sep=0}, "{\theta^*}"', curve={height=24pt}, from=1-1, to=1-3]
	\arrow[""{name=1, anchor=center, inner sep=0}, "{\theta_*}"', curve={height=24pt}, from=1-3, to=1-1]
	\arrow[""{name=2, anchor=center, inner sep=0}, "J"', curve={height=24pt}, from=1-3, to=1-5]
	\arrow[""{name=3, anchor=center, inner sep=0}, "{(-)'}"', curve={height=24pt}, from=1-5, to=1-3]
	\arrow["J"', curve={height=18pt}, from=1-1, to=3-3]
	\arrow["{\theta^*}"', curve={height=18pt}, from=3-3, to=1-5]
	\arrow[""{name=4, anchor=center, inner sep=0}, "{\theta^*}"', curve={height=24pt}, from=1-7, to=1-9]
	\arrow[""{name=5, anchor=center, inner sep=0}, curve={height=24pt}, from=1-9, to=1-7]
	\arrow["\dashv"{anchor=center, rotate=90}, draw=none, from=2, to=3]
	\arrow["\dashv"{anchor=center, rotate=90}, draw=none, from=0, to=1]
	\arrow["\ttop"{description}, draw=none, from=4, to=5]
\end{tikzcd}\]
\end{cor}

% pseudo-morphism classifier 2-comonad
%373

\noindent The following shows lali-cocompleteness of $\talgl$:

\begin{theo}\label{THM_talgl_reflector_cocomplete}
Let $T$ be a 2-monad on a 2-category $\ck$ that admits oplax limits of an arrow. Assume that $\talgs$ is cocomplete (in particular $T$ satisfies \textbf{Property L}). Then $\talgl$ is lali-cocomplete.
\end{theo}
\begin{proof}
This follows from (the dual of) Remark \ref{POZN_strict_weaklimits}.
\end{proof}

\begin{pozn}
By Remark \ref{POZN_ordinary_weakness}, this in particular shows that $\talgl$ is weakly cocomplete.
\end{pozn}

\begin{cor}
The following 2-categories are lali-cocomplete:
\begin{enumerate}
\item for a category $\cj$, the 2-category $\lax{\cj}{\cat}$ of 2-functors $\cj \to \cat$, lax-natural transformations and modifications,
\item the 2-category of monoidal categories and lax-monoidal functors and its symmetric/braided variants,
\item the 2-category of small 2-categories, lax functors, and icons,
\item for a set $\Phi$ of small categories, the 2-category $\Phi\text{-Colim}_l$ of small categories that admit a choice of $J$-indexed colimits for $J \in \Phi$ and \textbf{all} functors between them. 
\end{enumerate}
\end{cor}
\begin{proof}
Each of these is a 2-category of form $\talgl$, where $\ck$ is a complete and cocomplete 2-category and $T$ is one of the following 2-monads:
\begin{enumerate}
\item the 2-monad $T$ on $[\ob \cj, \cat]$ given by the left Kan extension along $\ob \cj \to \cj$ followed by restriction, see \cite[6.6]{twodim},
\item the 2-monad on $\cat$ for monoidal categories, see \cite[5.5]{companion},
\item the 2-category 2-monad $T$ on the 2-category $\cat \text{-Gph}$ of $\cat$-enriched graphs, see \cite[3.3]{onsemiflexible},
\item the 2-monad $T$ described in \cite[Theorem 6.1]{onthemonadicity} whose strict $T$-morphisms are functors that preserve the choices of $\Phi$-colimits. Lax $T$-morphisms are all functors because this 2-monad is lax-idempotent by \cite[Theorem 6.3]{onthemonadicity}.
\end{enumerate}
\end{proof}

\begin{pozn}
There is also a dual version for the 2-category $\talgc$ of $T$-algebras and \textbf{colax} $T$-morphisms. If $\talgs$ is sufficiently cocomplete, there exists an induced 2-comonad $Q_c$ (the \textit{colax morphism classifier 2-comonad}) and if $\ck$ admits lax limits of arrows, $Q_c$ is colax-idempotent. If $\talgs$ is cocomplete, $\talgc$ can be seen to be rali-cocomplete.
\end{pozn}

% LAX PIE ALGEBRAS
% 315

%\newpage
\appendix

\section{Auxiliary lemmas}
\begingroup
\allowdisplaybreaks

\begin{lemma}\label{THM_lemma_epsilon_phi_second_swallowtail}
In the proof of Theorem \ref{THM_Bunge}:
\begin{itemize}
\item $\epsilon$ is colax-natural,
\item $\Phi$ is a modification,
\item the second swallowtail identity.
\end{itemize}
\end{lemma}
\begin{proof}
In this proof, we will reference the defining equations for $\gamma', \iota', F\alpha, \epsilon, \Psi$ above the equals sign. In the unlabeled equations we use the middle-four interchange rule combined with the pseudofunctor laws.

\textbf{$\epsilon$ is colax-natural}: The \textbf{composition axiom} amounts to proving the equality of the following 2-cells:
% https://q.uiver.app/#q=WzAsMTUsWzIsMCwiRlVCIl0sWzIsMSwiRlVDIl0sWzQsMCwiQiJdLFs0LDEsIkMiXSxbMiwyLCJGVUQiXSxbNCwyLCJEIl0sWzUsMCwiRlVCIl0sWzUsMiwiRlVEIl0sWzcsMCwiQiJdLFs3LDEsIkMiXSxbNywyLCJEIl0sWzEsMCwiRlVCIl0sWzEsMiwiRlVEIl0sWzAsMCwiRlVCIl0sWzAsMiwiRlVEIl0sWzAsMiwiXFxlcHNpbG9uX0IiXSxbMiwzLCJoIl0sWzMsNSwiZyJdLFswLDEsIkZVaCJdLFsxLDQsIkZVZyJdLFs0LDUsIlxcZXBzaWxvbl9EIiwyXSxbMSwzLCJcXGVwc2lsb25fQyIsMV0sWzExLDEyLCJGKFVnIFxcY2lyYyBVaCkiLDEseyJsYWJlbF9wb3NpdGlvbiI6ODB9XSxbMTMsMTQsIkZVKGdoKSIsMl0sWzE0LDEyLCIiLDEseyJsZXZlbCI6Miwic3R5bGUiOnsiaGVhZCI6eyJuYW1lIjoibm9uZSJ9fX1dLFsxMiw0LCIiLDEseyJsZXZlbCI6Miwic3R5bGUiOnsiaGVhZCI6eyJuYW1lIjoibm9uZSJ9fX1dLFsxMywxMSwiIiwxLHsibGV2ZWwiOjIsInN0eWxlIjp7ImhlYWQiOnsibmFtZSI6Im5vbmUifX19XSxbMTEsMCwiIiwxLHsibGV2ZWwiOjIsInN0eWxlIjp7ImhlYWQiOnsibmFtZSI6Im5vbmUifX19XSxbNiw3LCJGVShnaCkiLDFdLFs2LDgsIlxcZXBzaWxvbl9CIl0sWzgsOSwiaCJdLFs5LDEwLCJnIl0sWzcsMTAsIlxcZXBzaWxvbl9EIiwyXSxbMjEsMTUsIlxcZXBzaWxvbl9oIiwyLHsic2hvcnRlbiI6eyJzb3VyY2UiOjIwLCJ0YXJnZXQiOjIwfX1dLFsyMCwyMSwiXFxlcHNpbG9uX2ciLDIseyJzaG9ydGVuIjp7InNvdXJjZSI6MjAsInRhcmdldCI6MjB9fV0sWzMyLDI5LCJcXGVwc2lsb25fe2dofSIsMix7InNob3J0ZW4iOnsic291cmNlIjoyMCwidGFyZ2V0IjoyMH19XSxbMjIsMSwiXFxnYW1tYSciLDAseyJzaG9ydGVuIjp7InNvdXJjZSI6MjB9fV0sWzIzLDIyLCJGXFxnYW1tYSIsMCx7InNob3J0ZW4iOnsic291cmNlIjoyMCwidGFyZ2V0IjoyMH19XV0=
\[\begin{tikzcd}
	FUB & FUB & FUB && B & FUB && B \\
	&& FUC && C &&& C \\
	FUD & FUD & FUD && D & FUD && D
	\arrow[""{name=0, anchor=center, inner sep=0}, "{\epsilon_B}", from=1-3, to=1-5]
	\arrow["h", from=1-5, to=2-5]
	\arrow["g", from=2-5, to=3-5]
	\arrow["FUh", from=1-3, to=2-3]
	\arrow["FUg", from=2-3, to=3-3]
	\arrow[""{name=1, anchor=center, inner sep=0}, "{\epsilon_D}"', from=3-3, to=3-5]
	\arrow[""{name=2, anchor=center, inner sep=0}, "{\epsilon_C}"{description}, from=2-3, to=2-5]
	\arrow[""{name=3, anchor=center, inner sep=0}, "{F(Ug \circ Uh)}"{description, pos=0.8}, from=1-2, to=3-2]
	\arrow[""{name=4, anchor=center, inner sep=0}, "{FU(gh)}"', from=1-1, to=3-1]
	\arrow[Rightarrow, no head, from=3-1, to=3-2]
	\arrow[Rightarrow, no head, from=3-2, to=3-3]
	\arrow[Rightarrow, no head, from=1-1, to=1-2]
	\arrow[Rightarrow, no head, from=1-2, to=1-3]
	\arrow["{FU(gh)}"{description}, from=1-6, to=3-6]
	\arrow[""{name=5, anchor=center, inner sep=0}, "{\epsilon_B}", from=1-6, to=1-8]
	\arrow["h", from=1-8, to=2-8]
	\arrow["g", from=2-8, to=3-8]
	\arrow[""{name=6, anchor=center, inner sep=0}, "{\epsilon_D}"', from=3-6, to=3-8]
	\arrow["{\epsilon_h}"', shorten <=4pt, shorten >=4pt, Rightarrow, from=2, to=0]
	\arrow["{\epsilon_g}"', shorten <=4pt, shorten >=4pt, Rightarrow, from=1, to=2]
	\arrow["{\epsilon_{gh}}"', shorten <=9pt, shorten >=9pt, Rightarrow, from=6, to=5]
	\arrow["{\gamma'}", shorten <=5pt, Rightarrow, from=3, to=2-3]
	\arrow["F\gamma", shorten <=7pt, shorten >=7pt, Rightarrow, from=4, to=3]
\end{tikzcd}\]
It is enough to prove these after applying $U(-) \circ \gamma y_{UB} \circ U\epsilon_D \mbbd \circ \Phi_D U(gh)$ on both sides. We then have:
\begin{align*}
& U(LHS) \circ \gamma^{-1} y_{UB} \circ U\epsilon_D \mbbd_{y_{UC} U(gh)} \circ \Phi_D U(gh) =
\phantom{.}
\phantom{.}
\phantom{.}\\
&= U(g\epsilon_h \circ \epsilon_g FUh \circ \epsilon_D \gamma')y_{UB} \circ \gamma^{-1}y_{UB} \circ U\epsilon_D UF\gamma y_{UB} \circ U\epsilon_D \mbbd_{y_{UC} U(gh)} \circ \Phi_D U(gh)
\phantom{.}
\phantom{.}
\phantom{.}\\
&\sstackrel{\eqref{EQP_Falpha}}{=} U(g\epsilon_h \circ \epsilon_g FUh \circ \epsilon_D \gamma')y_{UB} \circ \gamma^{-1} y_{UB} \circ U\epsilon_D \mbbd_{y_{UC} Ug Uh} \circ \Phi_D Ug Uh \circ \gamma
\phantom{.}
\phantom{.}
\phantom{.}\\
&= U(g\epsilon_h \circ \epsilon_g FUh)y_{UB} \circ \gamma^{-1} y_{UB} \circ U\epsilon_D U\gamma' y_{UB} \circ U\epsilon_D \mbbd_{y_{UC} Ug Uh} \circ \Phi_D Ug Uh \circ \gamma
\phantom{.}
\phantom{.}
\phantom{.}\\
&\sstackrel{\eqref{EQP_gamma_prime}}{=} U(g\epsilon_h  \circ \epsilon_g FUh)y_{UB} \circ \gamma^{-1} y_{UB} \circ U\epsilon_D \gamma^{-1} y_{UB} \circ U\epsilon_g UFUg \mbbd_{y_{UC} Uh} \circ\\
&\phantom{=}\circ U\epsilon_D \mbbd_{y_{UD} Ug} Uh \circ \Phi_D Ug Uh \circ \gamma
\phantom{.}
\phantom{.}
\phantom{.}\\
&= U(g\epsilon_h)y_{UB} \circ \gamma^{-1} y_{UB} \circ U\epsilon_g UFUh y_{UB} \circ \gamma^{-1} UFUh y_{UB} \circ U\epsilon_D UFUg \mbbd_{y_{UC} Uh} \circ \\
&\phantom{=}\circ U\epsilon_D \mbbd_{y_{UD} Ug} Uh \circ \Phi_D Ug Uh \circ \gamma 
\phantom{.}
\phantom{.}
\phantom{.}\\
&\sstackrel{\eqref{EQP_Phi_epsilon}}{=} U(g \epsilon_h) y_{UB} \circ \gamma^{-1} y_{UB} \circ \gamma^{-1} UFUh y_{UB} \circ Ug U\epsilon_B \mbbd_{y_{UC} Uh} \circ Ug \Phi_B Uh \circ \gamma
\phantom{.}
\phantom{.}
\phantom{.}\\
&= \gamma^{-1} y_{UB} \circ Ug U\epsilon_h y_{UB} \circ Ug \gamma^{-1} y_{UB} \circ Ug U\epsilon_B \mbbd_{y_{UC} Uh} \circ Ug \Phi_B Uh \circ \gamma
\phantom{.}
\phantom{.}
\phantom{.}\\
&\sstackrel{\eqref{EQP_Phi_epsilon}}{=} \gamma^{-1}y_{UB} \circ Ug \gamma^{-1}y_{UB} \circ Ug Uh \Phi_B \circ \gamma
\phantom{.}
\phantom{.}
\phantom{.}\\
&= \gamma^{-1} y_{UB} \circ \gamma^{-1} U\epsilon_B y_{UB} \circ Ug Uh \Phi_B \circ \gamma
\phantom{.}
\phantom{.}
\phantom{.}\\
&\sstackrel{\eqref{EQP_Phi_epsilon}}{=} U(RHS) \circ \gamma^{-1} y_{UB} \circ U\epsilon_D \mbbd_{y_{UC}U(gh)} \circ \Phi_D U(gh).
\end{align*}
%352 FULL VERSION

The \textbf{unit axiom} for $\epsilon$ amounts to showing that:
% https://q.uiver.app/#q=WzAsOCxbMCwwLCJGVUIiXSxbMCwyLCJGVUIiXSxbMiwyLCJCIl0sWzIsMCwiQiJdLFszLDEsIj0iXSxbNCwxLCJGVUIiXSxbNiwxLCJGVUIiXSxbOCwxLCJCIl0sWzUsNiwiIiwxLHsiY3VydmUiOi00LCJsZXZlbCI6Miwic3R5bGUiOnsiaGVhZCI6eyJuYW1lIjoibm9uZSJ9fX1dLFs1LDYsIkZVMV9CIiwyLHsiY3VydmUiOjR9XSxbNSw2LCJGMV97VUJ9IiwxXSxbNiw3LCJcXGVwc2lsb25fQiJdLFswLDMsIlxcZXBzaWxvbl9CIl0sWzAsMSwiRlUxX0IiLDJdLFszLDIsIiIsMSx7ImxldmVsIjoyLCJzdHlsZSI6eyJoZWFkIjp7Im5hbWUiOiJub25lIn19fV0sWzEsMiwiXFxlcHNpbG9uX0IiLDJdLFsxNSwxMiwiXFxlcHNpbG9uX3sxX0J9IiwyLHsic2hvcnRlbiI6eyJzb3VyY2UiOjIwLCJ0YXJnZXQiOjIwfX1dLFs5LDEwLCJGXFxpb3RhIiwyLHsic2hvcnRlbiI6eyJzb3VyY2UiOjIwLCJ0YXJnZXQiOjIwfX1dLFsxMCw4LCJcXGlvdGEnIiwyLHsic2hvcnRlbiI6eyJzb3VyY2UiOjIwLCJ0YXJnZXQiOjIwfX1dXQ==
\[\begin{tikzcd}
	FUB && B \\
	&&& {=} & FUB && FUB && B \\
	FUB && B
	\arrow[""{name=0, anchor=center, inner sep=0}, curve={height=-24pt}, Rightarrow, no head, from=2-5, to=2-7]
	\arrow[""{name=1, anchor=center, inner sep=0}, "{FU1_B}"', curve={height=24pt}, from=2-5, to=2-7]
	\arrow[""{name=2, anchor=center, inner sep=0}, "{F1_{UB}}"{description}, from=2-5, to=2-7]
	\arrow["{\epsilon_B}", from=2-7, to=2-9]
	\arrow[""{name=3, anchor=center, inner sep=0}, "{\epsilon_B}", from=1-1, to=1-3]
	\arrow["{FU1_B}"', from=1-1, to=3-1]
	\arrow[Rightarrow, no head, from=1-3, to=3-3]
	\arrow[""{name=4, anchor=center, inner sep=0}, "{\epsilon_B}"', from=3-1, to=3-3]
	\arrow["{\epsilon_{1_B}}"', shorten <=9pt, shorten >=9pt, Rightarrow, from=4, to=3]
	\arrow["F\iota"', shorten <=3pt, shorten >=3pt, Rightarrow, from=1, to=2]
	\arrow["{\iota'}"', shorten <=3pt, shorten >=3pt, Rightarrow, from=2, to=0]
\end{tikzcd}\]
It suffices to prove that these 2-cells are equal after applying the 2-cell\\ $U(-) \circ \gamma^{-1} y_{UB} \circ U\epsilon_D \mbbd_{y_{UB} U1_B} \circ \Phi_D U1_B$ on both sides. This is done as follows:

\[
\adjustbox{scale=0.90,center}{
\begin{tikzcd}
	{U\epsilon_B UFU1_B y_{UB}} & {U(\epsilon_B FU1_B)y_{UB}} && {U(\epsilon_B F1_{UB})y_{UB}} \\
	& {U\epsilon_B UF1_B y_{UB}} & {U\epsilon_B U1_{FUB} y_{UB}} & {U\epsilon_B y_{UB}} \\
	{U\epsilon_B y_{UB}U1_B } &&& {U\epsilon_B y_{UB}} \\
	& {U1_B U\epsilon_B y_{UB}} & {\eqref{EQP_Phi_epsilon}} \\
	{U1_B} & {U\epsilon_B y_{UB} U1_B} & {U\epsilon_B UFU1_B y_{UB}} & {U(\epsilon_B \circ FU1_B)y_{UB}}
	\arrow["{\gamma^{-1}y_{UB}}", from=1-1, to=1-2]
	\arrow["{U\epsilon_B UF\iota y_{UB}}"{description}, from=1-1, to=2-2]
	\arrow["{U(\epsilon_B F\iota)y_{UB}}", from=1-2, to=1-4]
	\arrow["{U(\epsilon_B \iota') y_{UB}}", from=1-4, to=2-4]
	\arrow["{\gamma^{-1}y_{UB}}"{description}, curve={height=-12pt}, from=2-2, to=1-4]
	\arrow["{U\epsilon_B U\iota' y_{UB}}", from=2-2, to=2-3]
	\arrow["{\eqref{EQP_Falpha}}"{description}, draw=none, from=2-2, to=3-1]
	\arrow["{\gamma^{-1}y_{UB}}", from=2-3, to=2-4]
	\arrow["{U\epsilon_B \mbbd_{y_{UB} U1_B}}", from=3-1, to=1-1]
	\arrow["{U\epsilon_B y_{UB} \iota}"{description, pos=0.3}, from=3-1, to=3-4]
	\arrow[""{name=0, anchor=center, inner sep=0}, "{U\epsilon_B \mbbd_{y_{UB}}}"{pos=0.7}, from=3-4, to=2-2]
	\arrow[""{name=1, anchor=center, inner sep=0}, "{U\epsilon_B \iota^{-1}y_{UB}}"', from=3-4, to=2-3]
	\arrow["1"', from=3-4, to=2-4]
	\arrow["{\gamma^{-1} y_{UB}}"{description}, from=4-2, to=3-4]
	\arrow["{\Phi_B U1_B}", from=5-1, to=3-1]
	\arrow["{U1_B \Phi_B}", from=5-1, to=4-2]
	\arrow["{\Phi_B U1_B}"', from=5-1, to=5-2]
	\arrow["{U\epsilon_B \mbbd_{y_{UB} U1_B}}"', shift right=2, from=5-2, to=5-3]
	\arrow["{\gamma^{-1} y_{UB}}"', from=5-3, to=5-4]
	\arrow["{U\epsilon_{1_B}y_{UB}}"', from=5-4, to=3-4]
	\arrow["{\eqref{EQP_iota_prime}}"'{pos=1}, draw=none, from=1, to=0]
\end{tikzcd}
}
\]
%347 FUNKCNI SEZNAM
%344

The \textbf{local naturality} for $\epsilon$ amounts to showing that the 2-cells below are equal:
% https://q.uiver.app/#q=WzAsMTIsWzEsMCwiRlVCIl0sWzEsMSwiRlVDIl0sWzMsMSwiQyJdLFszLDAsIkIiXSxbNCwxLCJGVUMiXSxbNCwwLCJGVUIiXSxbNiwwLCJCIl0sWzYsMSwiQyJdLFswLDAsIkZVQiJdLFswLDEsIkZVQyJdLFs3LDAsIkIiXSxbNywxLCJDIl0sWzAsMSwiRlVrIl0sWzAsMywiXFxlcHNpbG9uX0IiXSxbMSwyLCJcXGVwc2lsb25fQyIsMl0sWzMsMiwiayJdLFs4LDAsIiIsMix7ImxldmVsIjoyLCJzdHlsZSI6eyJoZWFkIjp7Im5hbWUiOiJub25lIn19fV0sWzksMSwiIiwyLHsibGV2ZWwiOjIsInN0eWxlIjp7ImhlYWQiOnsibmFtZSI6Im5vbmUifX19XSxbOCw5LCJGVWgiLDJdLFs1LDYsIlxcZXBzaWxvbl9CIl0sWzYsNywiaCIsMl0sWzUsNCwiRlVoIiwyXSxbNCw3LCJcXGVwc2lsb25fQyIsMl0sWzYsMTAsIiIsMSx7ImxldmVsIjoyLCJzdHlsZSI6eyJoZWFkIjp7Im5hbWUiOiJub25lIn19fV0sWzcsMTEsIiIsMSx7ImxldmVsIjoyLCJzdHlsZSI6eyJoZWFkIjp7Im5hbWUiOiJub25lIn19fV0sWzEwLDExLCJrIl0sWzE4LDEyLCJGVVxcYWxwaGEiLDIseyJzaG9ydGVuIjp7InNvdXJjZSI6MjAsInRhcmdldCI6MjB9fV0sWzIyLDE5LCJcXGVwc2lsb25faCIsMix7InNob3J0ZW4iOnsic291cmNlIjoyMCwidGFyZ2V0IjoyMH19XSxbMjAsMjUsIlxcYWxwaGEiLDAseyJzaG9ydGVuIjp7InNvdXJjZSI6MjAsInRhcmdldCI6MjB9fV0sWzE0LDEzLCJcXGVwc2lsb25fayIsMCx7InNob3J0ZW4iOnsic291cmNlIjoyMCwidGFyZ2V0IjoyMH19XV0=
\[\begin{tikzcd}
	FUB & FUB && B & FUB && B & B \\
	FUC & FUC && C & FUC && C & C
	\arrow[""{name=0, anchor=center, inner sep=0}, "FUk", from=1-2, to=2-2]
	\arrow[""{name=1, anchor=center, inner sep=0}, "{\epsilon_B}", from=1-2, to=1-4]
	\arrow[""{name=2, anchor=center, inner sep=0}, "{\epsilon_C}"', from=2-2, to=2-4]
	\arrow["k", from=1-4, to=2-4]
	\arrow[Rightarrow, no head, from=1-1, to=1-2]
	\arrow[Rightarrow, no head, from=2-1, to=2-2]
	\arrow[""{name=3, anchor=center, inner sep=0}, "FUh"', from=1-1, to=2-1]
	\arrow[""{name=4, anchor=center, inner sep=0}, "{\epsilon_B}", from=1-5, to=1-7]
	\arrow[""{name=5, anchor=center, inner sep=0}, "h"', from=1-7, to=2-7]
	\arrow["FUh"', from=1-5, to=2-5]
	\arrow[""{name=6, anchor=center, inner sep=0}, "{\epsilon_C}"', from=2-5, to=2-7]
	\arrow[Rightarrow, no head, from=1-7, to=1-8]
	\arrow[Rightarrow, no head, from=2-7, to=2-8]
	\arrow[""{name=7, anchor=center, inner sep=0}, "k", from=1-8, to=2-8]
	\arrow["FU\alpha"', shorten <=7pt, shorten >=7pt, Rightarrow, from=3, to=0]
	\arrow["{\epsilon_h}"', shorten <=4pt, shorten >=4pt, Rightarrow, from=6, to=4]
	\arrow["\alpha", shorten <=6pt, shorten >=6pt, Rightarrow, from=5, to=7]
	\arrow["{\epsilon_k}", shorten <=4pt, shorten >=4pt, Rightarrow, from=2, to=1]
\end{tikzcd}\]

An analogous approach will be done here as well, this time pre-composing with the 2-cell $U(-) \circ \gamma^{-1} y_{UB} \circ U\epsilon_D \mbbd_{y_{UC} Uh} \circ \Phi_D Uh$:
\[
\adjustbox{scale=0.9,center}{
\begin{tikzcd}
	Uh && {U\epsilon_C y_{UC} Uh} & {U\epsilon_C UFUh y_{UB}} & {U(\epsilon_C FUh)y_{UB}} \\
	& Uk & {U\epsilon_C y_{UC} Uk} & {U\epsilon_C UFUky_{UB}} & {U(\epsilon_C FUk)y_{UB}} \\
	{U\epsilon_C y_{UC} Uh} & {Uh U\epsilon_B y_{UB}} && {Uk U\epsilon_B y_{UB}} \\
	{U\epsilon_C UFUh y_{UB}} & {U(\epsilon_C FUh)} & {U(h \epsilon_B) y_{UB}} && {U(k \epsilon_B)y_{UB}}
	\arrow["{\Phi_C Uh}", from=1-1, to=1-3]
	\arrow["U\alpha"{description}, from=1-1, to=2-2]
	\arrow["{\Phi_D Uh}"', from=1-1, to=3-1]
	\arrow["{Uh\Phi_B}"', from=1-1, to=3-2]
	\arrow[""{name=0, anchor=center, inner sep=0}, "{U\epsilon_C \mbbd_{y_{UC} Uh}}", shift left=2, from=1-3, to=1-4]
	\arrow["{U\epsilon_C y_{UC} U\alpha}"', from=1-3, to=2-3]
	\arrow["{\gamma^{-1} y_{UB}}", from=1-4, to=1-5]
	\arrow["{U\epsilon_C UFU\alpha y_{UB}}", from=1-4, to=2-4]
	\arrow["{U(\epsilon_C FU\alpha)y_{UB}}", from=1-5, to=2-5]
	\arrow["{\Phi_C Uk}", from=2-2, to=2-3]
	\arrow["{Uk\Phi_B}"{description}, from=2-2, to=3-4]
	\arrow[""{name=1, anchor=center, inner sep=0}, "{U\epsilon_D \mbbd_{y_{UC} Uk}}"', shift right=2, from=2-3, to=2-4]
	\arrow["{\gamma^{-1} y_{UB}}"', from=2-4, to=2-5]
	\arrow["{\eqref{EQP_Phi_epsilon}}"{description}, draw=none, from=2-4, to=3-4]
	\arrow["{U\epsilon_k y_{UB}}", from=2-5, to=4-5]
	\arrow["{U\epsilon_D \mbbd_{y_{UC} Uh}}"', from=3-1, to=4-1]
	\arrow["{\eqref{EQP_Phi_epsilon}}"{description}, draw=none, from=3-2, to=3-1]
	\arrow["{U\alpha U\epsilon_B y_{UB}}"', from=3-2, to=3-4]
	\arrow["{\gamma^{-1} y_{UB}}"{description, pos=0.4}, from=3-2, to=4-3]
	\arrow["{\gamma^{-1} y_{UB}}"{description}, from=3-4, to=4-5]
	\arrow["{\gamma^{-1}y_{UB}}"', from=4-1, to=4-2]
	\arrow["{U\epsilon_h y_{UB}}"', from=4-2, to=4-3]
	\arrow["{U(\alpha \epsilon_B)y_{UB}}"', from=4-3, to=4-5]
	\arrow["{\eqref{EQP_Falpha}}"{description}, draw=none, from=0, to=1]
\end{tikzcd}
}\]

%348 FUNKCNI SEZNAM

\textbf{$\Psi$ is a modification}:
This amounts to showing that these 2-cells are equal:
\[\begin{tikzcd}
	& {} & {} \\
	FA & FUFA & {} & FA & FA & FB && FA \\
	\\
	& FUFB && FB && FUFB && FB
	\arrow[shift left=3, curve={height=-24pt}, Rightarrow, no head, from=2-1, to=2-4]
	\arrow[curve={height=-24pt}, Rightarrow, no head, from=2-5, to=2-8]
	\arrow["{Fy_A}", from=2-1, to=2-2]
	\arrow["FUFf", from=2-2, to=4-2]
	\arrow[""{name=0, anchor=center, inner sep=0}, "{F(y_Bf)}"', curve={height=30pt}, from=2-1, to=4-2]
	\arrow[""{name=0p, anchor=center, inner sep=0}, phantom, from=2-1, to=4-2, start anchor=center, end anchor=center, curve={height=30pt}]
	\arrow[""{name=1, anchor=center, inner sep=0}, "{\epsilon_{FA}}"{description}, from=2-2, to=2-4]
	\arrow[""{name=2, anchor=center, inner sep=0}, "{\epsilon_{FB}}"', from=4-2, to=4-4]
	\arrow["Ff", from=2-4, to=4-4]
	\arrow[""{name=3, anchor=center, inner sep=0}, "{F(y_B f)}"{description, pos=0.7}, curve={height=30pt}, from=2-5, to=4-6]
	\arrow["Ff"{description}, from=2-5, to=2-6]
	\arrow["{Fy_B}", from=2-6, to=4-6]
	\arrow["{\epsilon_{FB}}"', from=4-6, to=4-8]
	\arrow["Ff", from=2-8, to=4-8]
	\arrow[""{name=4, anchor=center, inner sep=0}, curve={height=-12pt}, Rightarrow, no head, from=2-6, to=4-8]
	\arrow[""{name=5, anchor=center, inner sep=0}, from=2-1, to=4-2]
	\arrow[""{name=5p, anchor=center, inner sep=0}, phantom, from=2-1, to=4-2, start anchor=center, end anchor=center]
	\arrow[""{name=6, anchor=center, inner sep=0}, draw=none, from=1-3, to=1-2]
	\arrow[""{name=6p, anchor=center, inner sep=0}, phantom, from=1-3, to=1-2, start anchor=center, end anchor=center]
	\arrow[""{name=7, anchor=center, inner sep=0}, draw=none, from=2-3, to=2-2]
	\arrow[""{name=7p, anchor=center, inner sep=0}, phantom, from=2-3, to=2-2, start anchor=center, end anchor=center]
	\arrow["{\Psi_B}"', shorten >=7pt, Rightarrow, from=4-6, to=4]
	\arrow["{Fy_f}"', shorten <=5pt, shorten >=5pt, Rightarrow, from=0p, to=5p]
	\arrow["{\gamma'}", shorten <=4pt, Rightarrow, from=5, to=2-2]
	\arrow["{\epsilon_{Ff}}"', shorten <=9pt, shorten >=9pt, Rightarrow, from=2, to=1]
	\arrow["{\gamma'}", shorten <=8pt, Rightarrow, from=3, to=2-6]
	\arrow["{\Psi_A}"'{pos=0.3}, shorten <=4pt, shorten >=9pt, Rightarrow, from=7p, to=6p]
\end{tikzcd}\]

This time we precompose both sides with $U(-) \circ \gamma^{-1} y_{UB} \circ U\epsilon_{FB} \mbbd_{y_{UFB}y_B f} \circ \Phi_{FB} y_B f$ to obtain:
\begin{align*}
&U(LHS)y_A \circ \gamma^{-1} y_{UB} \circ U\epsilon_{FB} \mbbd_{y_{UFB}y_B f} \circ \Phi_{FB} y_B f =
\phantom{.}
\phantom{.}
\phantom{.}\\
&= U(Ff \Psi_A \circ \epsilon_{Ff} Fy_A)y_A \circ \gamma^{-1}y_A \circ U\epsilon_{FB} U\gamma'y_A \circ U\epsilon_{FB} UF\mbbd_{y_Bf}y_A \circ \\
&\phantom{=} \circ U\epsilon_{FB} \mbbd_{y_{UFB} y_B f} \circ \Phi_{FB} y_B f
\phantom{.}
\phantom{.}
\phantom{.}\\
&\sstackrel{\eqref{EQP_Falpha}}{=} U(Ff \Psi_A \circ \epsilon_{Ff} Fy_A)y_A \circ \gamma^{-1}y_A \circ U\epsilon_{FB} U\gamma'y_A \circ U\epsilon_{FB}\mbbd_{UFf y_A} \circ \\
&\phantom{=} \circ U\epsilon_{FB} y_{UFB} \mbbd_{y_B f} \circ \Phi_{FB} y_B f
\phantom{.}
\phantom{.}
\phantom{.}\\
&\sstackrel{\eqref{EQP_gamma_prime}}{=} U(Ff \Psi_A \circ \epsilon_{Ff} Fy_A)y_A \circ \gamma^{-1}y_A \circ U\epsilon_{FB} \gamma^{-1} y_A \circ U\epsilon_{FB} UFUFf \mbbd_{y_{UFA}y_A} \circ\\
&\phantom{=}\circ U\epsilon_{FB} UFUFf \mbbd_{y_{UFB} UFf} y_A \circ U\epsilon_{FB} y_{UFB} \mbbd_{y_B f} \circ \Phi_{FB} y_B f
\phantom{.}
\phantom{.}
\phantom{.}\\
&= U(Ff \Psi_A)y_A \circ \gamma^{-1}y_A \circ U\epsilon_{Ff} UFy_A y_A \circ U(\epsilon_{FB} FUFf)\mbbd_{y_{UFA}y_A} \circ \\
&\phantom{=}\circ \gamma^{-1} y_{UFA} y_A \circ U\epsilon_{FB} \mbbd_{y_{UFB} UFf} y_A \circ \Phi_{FB} UFf y_A \circ \mbbd_{y_B f}
\phantom{.}
\phantom{.}
\phantom{.}\\
&\sstackrel{\eqref{EQP_Phi_epsilon}}{=} U(Ff \Psi_A)y_A \circ \gamma^{-1}y_A \circ \gamma^{-1} UFy_A y_A \circ UFf U\epsilon_{FA} \mbbd_{y_{UFA} y_A} \circ \\
&\phantom{=} \circ UFf \Phi_A y_A \circ \mbbd_{y_B f}
\phantom{.}
\phantom{.}
\phantom{.}\\
&= \gamma^{-1}y_A \circ UFf U\Psi_A y_A \circ UFf \gamma^{-1} y_A \circ UFf U\epsilon_{FA} \mbbd_{y_{UFA} y_A} \circ \\
&\phantom{=}\circ UFf \Phi_A y_A \circ \mbbd_{y_B f}
\phantom{.}
\phantom{.}
\phantom{.}\\
&\sstackrel{\eqref{EQP_Psi}}{=} \gamma^{-1} y_A \circ UFf \iota^{-1} y_A \circ \mbbd_{y_B f}
\phantom{.}
\phantom{.}
\phantom{.}\\
&\sstackrel{\eqref{EQP_Psi}}{=} \gamma^{-1} y_A \circ U\Psi_B UFf y_A \circ U\epsilon_{FB} \gamma^{-1} y_A \circ U\epsilon_{FB} UFy_B \mbbd_{y_B f}  \circ \\
&\phantom{=} \circ U\epsilon_{FB} \mbbd_{y_{UFB} y_B} f \circ \Phi_{FB} y_B f
\phantom{.}
\phantom{.}
\phantom{.}\\
&= U(\Psi_B Ff) y_A \circ \gamma^{-1}y_A \circ U\epsilon_{FB} \gamma^{-1} y_A \circ U\epsilon_{FB} UFy_B \mbbd_{y_B f} \circ \\
&\phantom{=}\circ U\epsilon_{FB} \mbbd_{y_{UFB} y_B} f \circ \Phi_{FB} y_B f
\phantom{.}
\phantom{.}
\phantom{.}\\
&\sstackrel{\eqref{EQP_gamma_prime}}{=} U(\Psi_B Ff) y_A \circ \gamma^{-1}y_A \circ U\epsilon_{FB} U\gamma' y_A \circ U\epsilon_{FB} \mbbd_{y_{UFB} y_B f} \circ \Phi_{FB} y_B f
\phantom{.}
\phantom{.}
\phantom{.}\\
&= U(RHS)y_A \circ \gamma^{-1} y_{UB} \circ U\epsilon_{FB} \mbbd_{y_{UFB} y_B f} \circ \Phi_{FB} y_B f.
\end{align*}
%351 FULL VERSION (ale nedokončená)
%349 TOO LARGE PICTURE

%381 NEWER PICTURE

\textbf{The second swallowtail identity}: This amounts to showing the following equality, which we will again do by an appropriate pre-composition:

% https://q.uiver.app/#q=WzAsMTAsWzIsMiwiRlVGVSJdLFsyLDQsIkZVQiJdLFs0LDIsIkZVQiJdLFs0LDQsIkIiXSxbMCwwLCJGVUIiXSxbNSwyLCI9Il0sWzYsMiwiXFxlcHNpbG9uX0IgXFxpb3RhIl0sWzAsMl0sWzMsMV0sWzEsMl0sWzIsNCwiIiwyLHsiY3VydmUiOjQsImxldmVsIjoyLCJzdHlsZSI6eyJoZWFkIjp7Im5hbWUiOiJub25lIn19fV0sWzQsMSwiRjFfe1VCfSIsMix7Im9mZnNldCI6NSwiY3VydmUiOjR9XSxbMCwyLCJcXGVwc2lsb25fe0ZVQn0iXSxbMCwxLCJGVSBcXGVwc2lsb25fQiIsMV0sWzEsMywiXFxlcHNpbG9uX0IiLDJdLFsyLDMsIlxcZXBzaWxvbl9CIl0sWzEsMiwiXFxlcHNpbG9uX3tcXGVwc2lsb25fQn0iLDEseyJsZXZlbCI6Mn1dLFs0LDAsIkZ5X3tVQn0iLDFdLFs0LDEsIkYoVVxcZXBzaWxvbl9CIHlfe1VCfSkiLDEseyJsYWJlbF9wb3NpdGlvbiI6NzB9XSxbNyw5LCIoRlxcUGhpKV9CIiwyLHsibGFiZWxfcG9zaXRpb24iOjgwLCJzaG9ydGVuIjp7InNvdXJjZSI6NjB9LCJsZXZlbCI6Mn1dLFs5LDAsIlxcZ2FtbWEnIiwwLHsic2hvcnRlbiI6eyJzb3VyY2UiOjMwfSwibGV2ZWwiOjJ9XSxbMTcsOCwiXFxQc2lfe1VCfSIsMix7ImxhYmVsX3Bvc2l0aW9uIjo2MCwib2Zmc2V0IjozLCJzaG9ydGVuIjp7InNvdXJjZSI6NDAsInRhcmdldCI6NDB9fV1d
\[\begin{tikzcd}
	FUB \\
	&&& {} \\
	{} & {} & FUFU && FUB & {=} & {\epsilon_B \iota} \\
	\\
	&& FUB && B
	\arrow[""{name=0, anchor=center, inner sep=0}, "{Fy_{UB}}"{description}, from=1-1, to=3-3]
	\arrow["{F1_{UB}}"', shift right=5, curve={height=24pt}, from=1-1, to=5-3]
	\arrow["{F(U\epsilon_B y_{UB})}"{description, pos=0.7}, from=1-1, to=5-3]
	\arrow["{(F\Phi)_B}"'{pos=0.8}, shorten <=10pt, Rightarrow, from=3-1, to=3-2]
	\arrow["{\gamma'}", shorten <=4pt, Rightarrow, from=3-2, to=3-3]
	\arrow["{\epsilon_{FUB}}", from=3-3, to=3-5]
	\arrow["{FU \epsilon_B}"{description}, from=3-3, to=5-3]
	\arrow[curve={height=24pt}, Rightarrow, no head, from=3-5, to=1-1]
	\arrow["{\epsilon_B}", from=3-5, to=5-5]
	\arrow["{\epsilon_{\epsilon_B}}"{description}, Rightarrow, from=5-3, to=3-5]
	\arrow["{\epsilon_B}"', from=5-3, to=5-5]
	\arrow["{\Psi_{UB}}"'{pos=0.6}, shift right=3, shorten <=25pt, shorten >=25pt, Rightarrow, from=0, to=2-4]
\end{tikzcd}\]

%393 druhej radek puvodne
%394 nad radkem "radek nade mnou" \phantom{X}
%395 \phantom{Y}
%396 \phantom{Z}
\begin{align*}
&U(LHS)y_{UB} \circ \gamma^{-1} y_{UB} \circ U\epsilon_B \mbbd_{y_{UB}} \circ \Phi_B = \\
\phantom{.}
\phantom{.}
\phantom{.}
&= U(\epsilon_B \Psi_{UB} \circ \epsilon_{\epsilon_B} Fy_{UB})y_{UB} \circ \gamma^{-1}y_{UB} \circ U\epsilon_B U\gamma' y_{UB} \circ U\epsilon_B UF\Phi_B y_{UB} \circ \\
&\phantom{=}\circ U\epsilon_B \mbbd_{y_{UB}} \circ \Phi_B \\
\phantom{.}
\phantom{.}
\phantom{.}
&\sstackrel{\eqref{EQP_Falpha}}{=} U(\epsilon_B \Psi_{UB} \circ \epsilon_{\epsilon_B}Fy_{UB})y_{UB} \circ \gamma^{-1}y_{UB} \circ U\epsilon_B U\gamma' y_{UB} \circ U\epsilon_B \mbbd_{y_{UB} U\epsilon_B y_{UB}} \circ \\
&\phantom{=}\circ \Phi_B U\epsilon_B y_{UB} \circ \Phi_B\\
\phantom{.}
\phantom{.}
\phantom{.}
&\sstackrel{\eqref{EQP_gamma_prime}}{=} U(\epsilon_B \Psi_{UB} \circ \epsilon_{\epsilon_B}Fy_{UB})y_{UB} \circ \gamma^{-1}y_{UB} \circ U\epsilon_B \gamma^{-1}y_{UB} \circ \\
&\phantom{=}\circ U\epsilon_B UFU\epsilon_B \mbbd_{y_{UFUB}y_{UB}} \circ U\epsilon_B \mbbd_{y_{UB}U\epsilon_B} y_{UB} \circ \Phi_B U\epsilon_B y_{UB} \circ \Phi_B \\
\phantom{.}
\phantom{.}
\phantom{X}
&= U(\epsilon_B \Psi_{UB}) y_{UB} \circ \gamma^{-1} y_{UB} \circ U\epsilon_{\epsilon_B} UFy_{UB} y_{UB} \circ \gamma^{-1} UFy_{UB} y_{UB} \circ \\
&\phantom{=} \circ U\epsilon_B UFU\epsilon_B \mbbd_{y_{UFUB}y_{UB}} \circ U\epsilon_B \mbbd_{y_{UB}\epsilon_B} y_{UB} \circ  U\epsilon_B y_{UB} \Phi_B \circ \Phi_B\\
\phantom{.}
\phantom{.}
\phantom{.}
&\sstackrel{\eqref{EQP_Phi_epsilon}}{=} U(\epsilon_B \Psi_{UB})y_{UB} \circ \gamma^{-1} y_{UB} \circ \gamma^{-1} UFy_{UB} y_{UB} \circ U\epsilon_B U\epsilon_{FUB} \mbbd \circ \\
&\phantom{=}\circ U\epsilon_B \Phi_B y_{UB} \circ \Phi_B      \\
\phantom{.}
\phantom{.}
\phantom{Y}
&= \gamma^{-1}y_{UB} \circ U\epsilon_B U\Psi_B y_{UB} \circ U\epsilon_B U\epsilon_{FUB} \mbbd_{y_{UFUB}y_{UB}} \circ U\epsilon_B \Phi_B y_{UB} \circ \Phi_B      \\
\phantom{.}
\phantom{.}
\phantom{.}
&\sstackrel{\eqref{EQP_Psi}}{=} \gamma^{-1}y_{UB} \circ U\epsilon_B \iota^{-1}y_{UB} \circ \Phi_B \\
\phantom{.}
\phantom{.}
\phantom{.}
&\sstackrel{\eqref{EQP_iota_prime}}{=} \gamma^{-1}y_{UB} \circ U\epsilon_B U\iota' y_{UB} \circ U\epsilon_B \mbbd_{y_{UB}} \circ \Phi_B\\
\phantom{.}
\phantom{.}
\phantom{.}
&= U(RHS) y_{UB} \circ \gamma^{-1}y_{UB} \circ U\epsilon_B \mbbd_{y_{UB}} \circ \Phi_B.
\end{align*}

%346

\end{proof}

\begin{lemma}\label{THM_lemma_big_l_adj_thm}
The composite bijection \textbf{(A)+(B)+(C)} in the proof of Theorem \ref{THM_BIG_lax_adj_thm}:
\[ \ck_D(DA,DHL)(f,Dl') \cong \cl(GDA,L)(s_L \circ Gf, s_L \circ GDl'), \]
is given by the assignment:
\[
\alpha \mapsto s_LG\alpha.
\]
\end{lemma}
\begin{proof}

Because of the swallowtail identity for the biadjunction in Proposition \ref{THM_biadjunkce_ck_ck_D}, it can be seen that the counit of the adjunction \eqref{EQ_hL_yA_yHL_JD} evaluated at $f: A \to HL$ is equal to the following (using notation from Remark \ref{POZN_explicit_form_h_L}):
% https://q.uiver.app/#q=WzAsNixbMCwwLCJBIl0sWzIsMCwiREEiXSxbNCwwLCJESEwiXSxbNiwwLCJITCJdLFsyLDEsIkhMIl0sWzQsMV0sWzAsMSwieV9BIl0sWzEsMiwiRGYiXSxbMiwzLCJoX0wiXSxbMCw0LCJmIiwyXSxbNCwyLCJ5X3tITH0iLDFdLFs0LDMsIiIsMix7ImN1cnZlIjozLCJsZXZlbCI6Miwic3R5bGUiOnsiaGVhZCI6eyJuYW1lIjoibm9uZSJ9fX1dLFsxLDQsInleey0xfV9mIiwwLHsibGV2ZWwiOjJ9XSxbMiw1LCJcXGVwc2lsb25fTCIsMCx7ImxhYmVsX3Bvc2l0aW9uIjozMCwic2hvcnRlbiI6eyJ0YXJnZXQiOjYwfSwibGV2ZWwiOjJ9XV0=
\[\begin{tikzcd}
	A && DA && DHL && HL \\
	&& HL && {}
	\arrow["{y_A}", from=1-1, to=1-3]
	\arrow["Df", from=1-3, to=1-5]
	\arrow["{h_L}", from=1-5, to=1-7]
	\arrow["f"', from=1-1, to=2-3]
	\arrow["{y_{HL}}"{description}, from=2-3, to=1-5]
	\arrow[curve={height=18pt}, Rightarrow, no head, from=2-3, to=1-7]
	\arrow["{y^{-1}_f}", Rightarrow, from=1-3, to=2-3]
	\arrow["{\epsilon_L}"{pos=0.3}, shorten >=6pt, Rightarrow, from=1-5, to=2-5]
\end{tikzcd}\]
The composite bijection \textbf{(A)+(B)+(C)} is thus the assignment:
\[
\alpha \mapsto s_L \circ GD(\epsilon_L f \circ h_L y^{-1}_{Dl'} \circ h_L \alpha y_A) \circ \beth_f^{-1}.
\]

Unwrapping the definitions of variables $h_L, \epsilon_L, \beth_f^{-1}$, what we need to show that the composite 2-cell below equals $s_L G\alpha$ (note that we use the same convention for the modifications on which a pseudofunctor is applied as in Definition \ref{DEFI_lax_adjoint}):
\[
\adjustbox{scale=0.85,center}{%
\begin{tikzcd}[column sep=small]
	&&&& GDA \\
	&&&& {GD^2HL} && GDHL & L \\
	\\
	GDA & {GD^2A} && {GD^2HL} & {GDHGD^2HL} && GDHGDHL & GDHL \\
	\\
	&& GDHL & GDHGDHL
	\arrow[""{name=0, anchor=center, inner sep=0}, Rightarrow, no head, from=4-4, to=2-5]
	\arrow[""{name=0p, anchor=center, inner sep=0}, phantom, from=4-4, to=2-5, start anchor=center, end anchor=center]
	\arrow[""{name=1, anchor=center, inner sep=0}, "{GDc_{DHL}}"', from=4-4, to=4-5]
	\arrow[""{name=1p, anchor=center, inner sep=0}, phantom, from=4-4, to=4-5, start anchor=center, end anchor=center]
	\arrow["{s_{GD^2HL}}"{description}, from=4-5, to=2-5]
	\arrow[""{name=2, anchor=center, inner sep=0}, "{Gp_{DHL}}"{description}, from=2-5, to=2-7]
	\arrow[""{name=2p, anchor=center, inner sep=0}, phantom, from=2-5, to=2-7, start anchor=center, end anchor=center]
	\arrow[""{name=3, anchor=center, inner sep=0}, "{s_L}", from=2-7, to=2-8]
	\arrow[""{name=3p, anchor=center, inner sep=0}, phantom, from=2-7, to=2-8, start anchor=center, end anchor=center]
	\arrow[""{name=4, anchor=center, inner sep=0}, from=4-5, to=4-7]
	\arrow[""{name=4p, anchor=center, inner sep=0}, phantom, from=4-5, to=4-7, start anchor=center, end anchor=center]
	\arrow["{s_{GDHL}}"{description}, from=4-7, to=2-7]
	\arrow[""{name=5, anchor=center, inner sep=0}, "{GDHs_L}"', from=4-7, to=4-8]
	\arrow[""{name=5p, anchor=center, inner sep=0}, phantom, from=4-7, to=4-8, start anchor=center, end anchor=center]
	\arrow["{s_L}"', from=4-8, to=2-8]
	\arrow["Gf", from=1-5, to=2-7]
	\arrow[""{name=6, anchor=center, inner sep=0}, "{Gp_{DA}}"{description, pos=0.6}, from=4-2, to=1-5]
	\arrow[""{name=7, anchor=center, inner sep=0}, "GDf", curve={height=-12pt}, from=4-2, to=4-4]
	\arrow["{GDy_A}"', from=4-1, to=4-2]
	\arrow[""{name=8, anchor=center, inner sep=0}, curve={height=-24pt}, Rightarrow, no head, from=4-1, to=1-5]
	\arrow[""{name=9, anchor=center, inner sep=0}, "{GD^2l'}"', curve={height=12pt}, from=4-2, to=4-4]
	\arrow["{GDl'}"', curve={height=18pt}, from=4-1, to=6-3]
	\arrow["{GDy_{HL}}"{description, pos=0.7}, from=6-3, to=4-4]
	\arrow[from=6-3, to=6-4]
	\arrow["{GDHGDy_{HL}}"{description}, from=6-4, to=4-5]
	\arrow[""{name=10, anchor=center, inner sep=0}, curve={height=18pt}, Rightarrow, no head, from=6-4, to=4-7]
	\arrow[""{name=11, anchor=center, inner sep=0}, shift right=5, curve={height=30pt}, Rightarrow, no head, from=6-3, to=4-8]
	\arrow["{s_{Gp_{DHL}}^{-1}}", shorten <=9pt, shorten >=9pt, Rightarrow, from=2p, to=4p]
	\arrow["{s_{s_L}^{-1}}", shorten <=9pt, shorten >=9pt, Rightarrow, from=3p, to=5p]
	\arrow["GD\alpha", shorten <=3pt, shorten >=3pt, Rightarrow, from=7, to=9]
	\arrow["{(GDy)_{l'}}"'{pos=0.7}, shorten <=12pt, Rightarrow, from=9, to=6-3]
	\arrow["{(GDc)_{y_{HL}}}"'{pos=1}, shift right=3, shorten <=21pt, Rightarrow, from=1, to=6-4]
	\arrow["{(GDHG\Psi)_{HL}}", shorten >=7pt, Rightarrow, from=4-5, to=10]
	\arrow["{(GD\tau)^{-1}_L}", shorten <=3pt, shorten >=3pt, Rightarrow, from=10, to=11]
	\arrow["{(G\Psi)^{-1}_A}"', shorten <=6pt, shorten >=6pt, Rightarrow, from=8, to=6]
	\arrow["{\sigma_{DHL}^{-1}}", shorten <=4pt, shorten >=4pt, Rightarrow, from=0p, to=1p]
	\arrow["{(Gp)_f^{-1}}", shorten <=6pt, shorten >=9pt, Rightarrow, from=6, to=4-4]
\end{tikzcd}
}\]

The diagram below proves this equality:
\[
\adjustbox{scale=0.8,center}{
\begin{tikzcd}
	{s_L Gf} & {s_L Gf Gp_{DA}GDy_A} & {s_L Gp_{DHL} GDf GDy_A} \\
	{s_L GDl'} & {s_L GDl' Gp_{DA}GDy_A} & {s_L Gp_{DHL} GD^2l' GDy_A} \\
	{s_L GDl'} && {s_L Gp_{DHL} GDy_{HL} GDl'} \\
	& {s_L Gp_{DHL} GDy_{HL}s_{GDHL} GDc_{HL} GDl'} & {s_L Gp_{DHL} s_{GD^2HL} GDc_{DHL} GDy_{HL} GDl'} \\
	{s_L GDl'} && {s_L Gp_{DHL} s_{GD^2HL}GDHGDy_{HL} GDc_{HL} GDl'} \\
	{s_L GDHs_L GDc_{HL} GDl'} & {s_L s_{GDHL} GDc_{HL} GDl'} & {s_L s_{GDHL} GDHGp_{DHL} GDHGDy_{HL} GDc_{HL} GDl'}
	\arrow[""{name=0, anchor=center, inner sep=0}, "{s_L Gf (G\Psi)^{-1}_A}", from=1-1, to=1-2]
	\arrow[""{name=0p, anchor=center, inner sep=0}, phantom, from=1-1, to=1-2, start anchor=center, end anchor=center]
	\arrow["{s_L(Gp)_f^{-1}GDy_A}", from=1-2, to=1-3]
	\arrow["{s_L Gp_{DHL} GD\alpha GDy_A}", from=1-3, to=2-3]
	\arrow["{s_L Gp_{DHL} GDy_{l'}^{-1}}", from=2-3, to=3-3]
	\arrow["{s_L Gp_{DHL} \sigma^{-1}_{DHL} GDy_{HL} GDl'}", from=3-3, to=4-3]
	\arrow["{s_L (GD\tau)^{-1}_L GDl'}", from=6-1, to=5-1]
	\arrow["{s_LG\alpha Gp_{DA} GDy_A}", from=1-2, to=2-2]
	\arrow[""{name=1, anchor=center, inner sep=0}, "{s_L (Gp)^{-1}_{Dl}GDy_A}"{description}, from=2-2, to=2-3]
	\arrow["{s_L G\alpha}"', from=1-1, to=2-1]
	\arrow[""{name=2, anchor=center, inner sep=0}, "{s_L GDl'(G\Psi)^{-1}_A}", from=2-1, to=2-2]
	\arrow[""{name=2p, anchor=center, inner sep=0}, phantom, from=2-1, to=2-2, start anchor=center, end anchor=center]
	\arrow[""{name=3, anchor=center, inner sep=0}, "{s_L (G\Psi)^{-1}_{HL}GDl'}"{description}, from=2-1, to=3-3]
	\arrow["1"', from=3-1, to=5-1]
	\arrow["1"', from=2-1, to=3-1]
	\arrow["{s_L Gp_{DHL} s_{GD^2HL} GDc_{y_{HL}} GDl'}", from=4-3, to=5-3]
	\arrow["{s_L s^{-1}_{Gp_{DHL}} GDHGDy_{HL} GDc_{HL} GDl'}", from=5-3, to=6-3]
	\arrow[""{name=4, anchor=center, inner sep=0}, "{s_L s_{GDHL} (GDHG\Psi)_{HL}GDc_{HL} GDl'}", shift left=3, from=6-3, to=6-2]
	\arrow[""{name=5, anchor=center, inner sep=0}, "{s^{-1}_{s_L} GDc_{HL} GDl'}", from=6-2, to=6-1]
	\arrow[""{name=6, anchor=center, inner sep=0}, "{s_L \sigma_{HL} GDl'}"{description}, from=6-2, to=5-1]
	\arrow[""{name=7, anchor=center, inner sep=0}, "{s_L Gp_{DHL} s_{GDy_{HL}} GDc_{HL} GDl'}"{description}, from=5-3, to=4-2]
	\arrow["{s_L (G\Psi)_{HL} s_{GDHL} GDc_{HL} GDl'}", from=4-2, to=6-2]
	\arrow[""{name=8, anchor=center, inner sep=0}, "{s_L Gp_{DHL}GDy_{HL} \sigma^{-1}_{HL} GDl'}"{description}, from=3-3, to=4-2]
	\arrow[""{name=9, anchor=center, inner sep=0}, "{s_L (G\Psi)_{HL} GDl'}"{description}, from=3-3, to=3-1]
	\arrow["{s_L \sigma^{-1}_{HL} GDl'}"{description}, from=3-1, to=6-2]
	\arrow["{(a)}"{description}, draw=none, from=1-3, to=1]
	\arrow["{(b)}"{description}, draw=none, from=1, to=3]
	\arrow["{(c)}"{description}, draw=none, from=6, to=5]
	\arrow["{(d)}"{description, pos=0.8}, draw=none, from=7, to=4]
	\arrow["{(*)}"{description}, draw=none, from=9, to=4-2]
	\arrow["{(*)}"{description}, draw=none, from=0p, to=2p]
	\arrow["{(e)}"{description}, draw=none, from=8, to=7]
\end{tikzcd}}\]

In this diagram:
\begin{itemize}
\item $(a)$ is the local naturality of $(Gp)^{-1}$,
\item $(b)$ is the modification axiom for $(G\Psi)^{-1}$
\item $(c)$ is the swallowtail identity for $(s,c)$,
\item $(d)$ is the equation derived from the local naturality of $s$,
\item $(e)$ is the modification axiom for $\sigma^{-1}$,
\item $(*)$'s are the middle-four-interchange laws.
\end{itemize}

%382 puvodni

\end{proof}
\endgroup

%\LaTeX{} \cite{latex2e} \cite{knuth:1984} %is a set of macros built atop \TeX{} \cite{texbook}.
\bibliographystyle{plain} % We choose the "plain" reference style
\bibliography{ClanekIIREF} % Entries are in the refs.bib file

\begin{thebibliography}{10}

\bibitem{adjfunthms}
Nathanael Arkor, Ivan Di~Liberti, and Fosco Loregian.
\newblock Adjoint functor theorems for lax-idempotent pseudomonads.
\newblock {\em ArXiv preprint arXiv:2306.10389}, 2023.

\bibitem{twodim}
Robert Blackwell, Gregory~M Kelly, and A~John Power.
\newblock Two-dimensional monad theory.
\newblock {\em Journal of Pure and Applied Algebra}, 59(1):1--41, 1989.

\bibitem{onsemiflexible}
John Bourke and Richard Garner.
\newblock On semiflexible, flexible and pie algebras.
\newblock {\em Journal of Pure and Applied Algebra}, 217(2):293--321, 2013.

\bibitem{bunge}
Marta Bunge.
\newblock Coherent extensions and relational algebras.
\newblock {\em Transactions of the American Mathematical Society},
  197:355--390, 1974.

\bibitem{bungebicomma}
Marta Bunge and Jonathon Funk.
\newblock On a bicomma object condition for {KZ}-doctrines.
\newblock {\em Journal of Pure and Applied Algebra}, 143(1-3):69--105, 1999.

\bibitem{laxcomma2cats}
Maria~Manuel Clementino and Fernando~Lucatelli Nunes.
\newblock Lax comma $2 $-categories and admissible $2 $-functors.
\newblock {\em ArXiv preprint arXiv:2002.03132}, 2020.

\bibitem{kaninj}
Ivan Di~Liberti, Gabriele Lobbia, and Lurdes Sousa.
\newblock {KZ}-pseudomonads and {K}an injectivity.
\newblock {\em ArXiv preprint arXiv:2211.00380}, 2022.

\bibitem{formalcat}
John~Walker Gray.
\newblock {\em Formal category theory: adjointness for 2-categories}, volume
  391.
\newblock Springer, 2006.

\bibitem{2dimensionalcats}
Niles Johnson and Donald Yau.
\newblock {\em 2-dimensional categories}.
\newblock Oxford University Press, USA, 2021.

\bibitem{elephant}
Peter~T Johnstone.
\newblock {\em Sketches of an elephant: a topos theory compendium: volume 2},
  volume~2.
\newblock Oxford University Press, 2002.

\bibitem{doctrinaladj}
G~Max Kelly.
\newblock Doctrinal adjunction.
\newblock In {\em Category Seminar: Proceedings Sydney Category Theory Seminar
  1972/1973}, pages 257--280. Springer, 2006.

\bibitem{elemobs}
Gregory~Maxwell Kelly.
\newblock Elementary observations on 2-categorical limits.
\newblock {\em Bulletin of the Australian Mathematical Society},
  39(2):301--317, 1989.

\bibitem{onthemonadicity}
Gregory~Maxwell Kelly and Stephen Lack.
\newblock On the monadicity of categories with chosen colimits.
\newblock {\em Theory and Applications of Categories}, 7(7):148--170, 2000.

\bibitem{codobjcoh}
Stephen Lack.
\newblock Codescent objects and coherence.
\newblock {\em Journal of Pure and Applied Algebra}, 175(1-3):223--241, 2002.

\bibitem{limitsforlax}
Stephen Lack.
\newblock Limits for lax morphisms.
\newblock {\em Applied Categorical Structures}, 13:189--203, 2005.

\bibitem{companion}
Stephen Lack.
\newblock A 2-categories companion.
\newblock In {\em Towards Higher Categories}, pages 105--191. Springer, 2009.

\bibitem{enrichedweakness}
Stephen Lack and Ji{\v{r}}{\'\i} Rosick{\`y}.
\newblock Enriched weakness.
\newblock {\em Journal of Pure and Applied Algebra}, 216(8-9):1807--1822, 2012.

\bibitem{enhanced2cats}
Stephen Lack and Michael Shulman.
\newblock Enhanced 2-categories and limits for lax morphisms.
\newblock {\em Advances in Mathematics}, 229(1):294--356, 2012.

\bibitem{kanext}
Francisco Marmolejo and Richard~James Wood.
\newblock Kan extensions and lax idempotent pseudomonads.
\newblock {\em Theory and Applications of Categories}, 26(1):1--29, 2012.

\bibitem{threepieces}
Robert Par{\'e}.
\newblock Three easy pieces: Imaginary seminar talks in honour of {B}ob
  {R}osebrugh.
\newblock {\em Theory and Applications of Categories}, 36(6):171--200, 2021.

\bibitem{twoconstr}
Ross Street.
\newblock Two constructions on lax functors.
\newblock {\em Cahiers de Topologie et G{\'e}om{\'e}trie Diff{\'e}rentielle
  Cat{\'e}goriques}, 13(3):217--264, 1972.

\bibitem{streetfib}
Ross Street.
\newblock Fibrations and {Y}oneda's lemma in a 2-category.
\newblock In {\em Category Seminar: Proceedings Sydney Category Theory Seminar
  1972/1973}, pages 104--133. Springer, 2006.

\bibitem{familial}
Mark Weber.
\newblock Familial 2-functors and parametric right adjoints.
\newblock {\em Theory and Applications of Categories}, 18(22):665--732, 2007.

\bibitem{yoneda2toposes}
Mark Weber.
\newblock Yoneda structures from 2-toposes.
\newblock {\em Applied Categorical Structures}, 15:259--323, 2007.

\end{thebibliography}

%138 PAST BIBLIOGRAPHY

\end{document}